\documentclass[aap,preprint]{imsart}
\RequirePackage{amsthm,amsmath,amsfonts,amssymb,mathtools}
\RequirePackage[numbers,sort&compress]{natbib}
\RequirePackage[colorlinks,citecolor=blue,urlcolor=blue]{hyperref}

\usepackage[capitalize, nameinlink]{cleveref} 
\usepackage{enumitem} 
\setlist[enumerate,1]{label={(\roman*)}}
\usepackage{lmodern}
\usepackage{xspace} 

\startlocaldefs
\theoremstyle{plain}

\newtheorem{theorem}{Theorem}[section]
\newtheorem{lemma}[theorem]{Lemma}
\newtheorem{proposition}[theorem]{Proposition}
\newtheorem{corollary}[theorem]{Corollary}
\theoremstyle{remark}
\newtheorem{definition}[theorem]{Definition}
\newtheorem{example}[theorem]{Example}
\newtheorem{remark}[theorem]{Remark}

\newcommand{\bR}{\mathbb{R}}\newcommand{\R}{\mathbb{R}}
\newcommand{\oR}{\overline{\R}}
\newcommand{\bE}{\mathbb{E}}
\newcommand{\bP}{\mathbb{P}}\renewcommand{\P}{\mathbb{P}}
\newcommand{\1}{\textbf 1}
\newcommand{\dd}{\,\mathrm{d}} 
\newcommand{\ddd}{\mathrm{d}} 
\newcommand{\ii}{\mathrm{i}} 
\newcommand{\ee}{\mathrm{e}} 
\newcommand{\abs}[1]{\lvert#1\rvert}
\newcommand{\eps}{\varepsilon}
\newcommand{\tw}{{}^tw}
\newcommand{\tv}{{}^tv}
\newcommand{\DRR}{D(\R,\oR)}
\newcommand{\cadlag}{càdlàg\xspace}
\newcommand{\as}{a.s.\xspace}
\newcommand{\ccdot}{\,\cdot\,}
\newcommand{\exc}{\mathbf{e}} 
\newcommand{\W}{\mathcal W} 
\newcommand{\D}{ D} 

\newcommand{\eqdis}{\overset{(d)}{=}}
\newcommand{\ind}[1]{\mathbf{1}_{\{#1\}}}

\endlocaldefs

\begin{document}

\begin{frontmatter}
\title{The structure of entrance and exit at infinity for time-changed L\'evy processes}
\runtitle{Time-changed L\'evy processes}

\begin{aug}
\author[A]{\fnms{Samuel}~\snm{Baguley}\ead[label=e1]{Samuel.Baguley@hpi.de}},

\author[B]{\fnms{Leif}~\snm{D\"oring}\ead[label=e2]{doering@uni-mannheim.de}}
\and
\author[C]{\fnms{Quan}~\snm{Shi}\ead[label=e3]{quan.shi@amss.ac.cn}}
\address[A]{
Hasso Plattner Institute, University of Potsdam\printead[presep={,\ }]{e1}}

\address[B]{
Institute of Mathematics, University of Mannheim\printead[presep={,\ }]{e2}}

\address[C]{AMSS, Chinese Academy of Sciences\printead[presep={,\ }]{e3}}
\end{aug}

\begin{abstract}
Studying the behavior of Markov processes at boundary points of the state space has a long history, dating back all the way to William Feller. With different motivations in mind entrance and exit questions have been explored for different discontinuous Markov processes in the past two decades. Proofs often use time-change techniques and rely on problem specific knowledge such branching or scaling properties. In this article we ask how far techniques can be pushed with as little as possible model assumptions. We give sharp conditions on time-changed L\'evy processes to allow entrance and regular boundary point at infinity. The main tool we introduce is a generalised scaling property that holds for all time-changed L\'evy processes and can be used to extend scaling arguments for self-similar Markov processes.
\end{abstract}



\end{frontmatter}

\section{Introduction}
How a Markov process behaves at its boundary points is a classical problem, going back to Feller's seminal articles in the 1950s \cite{Feller52,Feller54}. Feller studied Markov diffusion processes, which have continuous sample paths; in that setting extensions to boundary points can for instance be constructed analytically using the Hille-Yosida theory and explicit solutions to second order PDEs. 
For processes with discontinuities, the picture is less clear.
The objective of this article is to progress the understanding of jump-type Markov processes at boundary points.
For such processes
the analytic approach is significantly more delicate as direct PDE solutions using variation of constants methods are infeasible for general jump-type operators. Results in the literature are mainly based on abstract excursion theory and time-change constructions for specific examples.\smallskip

For applications of these approaches to the boundary behavior of branching processes at infinity we refer for instance to \cite{KypEtAl-17,Fou19,LiYangZhou}, for the boundary behavior of positive self-similar Markov processes at zero see for instance \cite{Riv07,Fit06,BerSav11,DDK17,BarBer11,BD12,CabCha06,D12} and for driftless SDEs driven by stable processes see \cite{Zanzotto02,DK20,BDK}.
A recurring tool in these analyses is the Lamperti transform.
Lamperti's representation for branching processes directly links continuous-state branching processes and their variants to time-changed L\'evy processes, and Lamperti's representation of positive self-similar Markov processes (pssMps) gives a link to time-changed exponentials of L\'evy processes. In both cases the boundary points of the original Markov process can be studied via the infinite boundary points of the time-changed L\'evy process.\smallskip

In the present article we study how far ideas can be extended to general time-changed L\'evy processes without specific time-change structure. The setup can be seen to extend the study of all martingale diffusion processes (time-changes of Brownian motion) to processes with discontinuities, an approach that has also appears in the random walk world \cite{BCKW,BerKor16}. We introduce a new space-time invariance property that is as effective as the classical scaling property from stochastic process theory in the study of the entrance behavior of pssMps at zero. The new space-time invariance property is used to characterise entrance and regular boundary for time-changed L\'evy processes without any structural assumption on the L\'evy process or the time-change.

\subsection{Organisation of the paper}
In \cref{main results section} we state our main theorems and give their applications to a number of examples, connecting to various results in previous literature. In \cref{sec:X} we establish several properties of time-changed Lévy processes, including the Feller property, explosion behaviour, and the new space-time invariance property that we call the \emph{$R$-scaling invariance}. 
In \cref{sec:P-infty} we study the existence of an entrance law and prove \cref{thm:1}. Finally, 
\cref{sec:rec} is concerned with the construction of recurrent extensions, where we give the proof of \cref{thm:2,thm:convergence to n} and establish more properties of the excursion.
	
\section{Main results}\label{main results section}
Let us first fix some notation and terminology that will be used throughout the article.
Suppose $\xi$ is a Lévy process on $\R$ that is non-lattice and unkilled, with characteristic exponent
$\mathrm{E}[\ee^{\ii q \xi_t}] = \ee^{-t\Psi(q)}$, where $\Psi$ is defined by
$$
	\Psi(q) = \frac{1}{2} \sigma^2 q^2 + \ii a q +  \int_{\mathbb{R}} \left( 1- \ee^{\ii qx} + \ii q x\ind{\abs{x}<1} \right)\Pi(\ddd x), \qquad q\in \R, 
$$
for some fixed $a\in \mathbb{R}$, $\sigma^2\ge 0$, and L\'evy measure $\Pi$ that satisfies $\int_{\mathbb{R}} (1\wedge x^2) \Pi(\ddd x)<\infty$. 
We use $\mathrm P_x$ to denote the law of $\xi$ issued from $x$.

Let $R:\R\to(0,\infty)$ be continuous, and let $X$ be the stochastic process defined by the time-change 
\begin{align}\label{eq:time-change}
	X_t = x+ \xi_{\eta(t)},\quad \text{ where }\eta(t) = \inf\left\{ s> 0\colon \int_0^{s} \frac{1}{R(x+ \xi_r)}\dd r  >t\right\}
\end{align}
for $t<\zeta:= \int_0^\infty \frac{1}{R(x+\xi_r)}\dd r.$ 
The corresponding law of $X$ under $\mathrm P_0$ will be denoted by $\P_x$. In what follows $X$ will be referred to as the \emph{time-changed Lévy process} and $\zeta$ will be called its explosion time. 
 One significant effect of the time-change is that $X$ can have a finite lifetime, which $\xi$ (by assumption) could not. We are interested in deriving analytic conditions on $R$ which ensure the existence of non-trivial measures $\P_\infty$ that allow $X$ to be started from infinity in a way that is consistent with other properties of the process.
We shall assume additionally that
\begin{equation}\label{eq:R-Condition}
    \liminf_{x\to-\infty} \frac{1}{R(x)}>0,
\end{equation}
which is shown in \cref{prop;explosion} to imply that explosion can only occur if $\lim_{n\to\infty}\xi_t=+\infty$; the condition is not an essential part of our analysis, but allows us to avoid case distinctions on the limit of $\xi$ in every result.

\begin{definition}\label{defn:extension}
Let $X$ be a strong Markov process on $\R$.
A Feller process $\overline{X}$ on state space $(-\infty, \infty]$ of law $\overline{\bP}$ is called an \emph{extension at infinity} of $X$ if
\begin{enumerate}
	\item the boundary point $+\infty$ is not a trap for $\overline{X}$; and
	\item issued from $x\in \bR$, $\overline{X}$ up to the first hitting time $ \inf\{s> 0\colon \overline{X}_s = \infty\}$ is equal in distribution to $X$.
\end{enumerate}
\end{definition}

As mentioned in the introduction, we are interested in studying analytic conditions (in terms of $\xi$) on the function $R$ that ensure the existence of such extensions. 
This question is classified into two situations, according to whether the point $+\infty$ is an \emph{entrance boundary} or \emph{regular boundary}; the meaning in our circumstance is specified as follows, compare for instance \cite{KT}, Section 15.6.
\begin{definition}\label{def:down from infty}
A process $X$ with an extension $\overline{X}$ at infinity is said to \emph{come down from infinity} 
if $\overline\P_x(\inf\{s> 0\colon \overline{X}_s = \infty\}<\infty)=0$ for all $x\in(-\infty,\infty]$.
In this case $+\infty$ is called an \emph{entrance boundary} for $X$.
\end{definition}
If $X$ comes down from infinity, then 
we must have $\overline{\bP}_x = \bP_x$ for all $x\in (-\infty, \infty)$, and thus in particular it holds that $\bP_x\Rightarrow \overline{\bP}_{\infty}$ as $x\to \infty$.
Thus there is at most one extension of $X$ that comes down from infinity.

\begin{definition}\label{def:regular boundary}
An extension at infinity $\overline X$ of a process $X$ is called a \emph{recurrent extension at infinity} 
if $\overline\P_x(\inf\{s> 0\colon \overline{X}_s = \infty\}<\infty)=1$ for all $x\in(-\infty,\infty]$.
In this case $+\infty$ is called a \emph{regular boundary} for $X$.
\end{definition}

Unlike the case of coming down from infinity, a recurrent extension at infinity need not be unique; 
we could for example hold $\overline{X}$ at infinity for an exponential time, and then have it jump away according to any distribution on $\mathbb{R}$. 
Below we introduce an additional scaling condition that ensures that $+\infty$ cannot be a holding point for $\overline X$, and which generalises the well-known scaling property of pssMps.

In what follows we collect our results that characterise whether $+\infty$ is an entrance boundary, a regular boundary, or neither.
When $X$ is the time-changed Lévy process, \cref{def:down from infty,def:regular boundary} form a dichotomy, because the finiteness of $\inf\{s> 0\colon \overline{X}_s = \infty\}$ is a zero-one law.

\subsection{Entrance at infinity}
Let us start with entrance from infinity. For that sake the following conditions \ref{H1} and \ref{H2} turn out to be crucial:
\begin{enumerate}[label=(H\arabic*)]
	\item\label{H1} Either (i) or (ii) holds: 
	\begin{enumerate}
		\item[(i)] $\mathrm{E}_0[\xi_1]\in (-\infty,0)$;
		\item[(ii)]
	$\mathrm{E}_0[\xi_1]=0$  and  $\displaystyle \int_{(-\infty,-1]} \frac{x \Pi(-\infty, x)}{\int_{x}^0 \dd y \int_{-\infty}^y \dd z \Pi(-z, \infty)}\dd x<\infty.$
	\end{enumerate} 
	\item\label{H2} There exists $\theta>0$ such that $\mathrm{E}_0[\ee^{-\theta \xi_1}] \le 1.$
\end{enumerate}
Condition \ref{H2} forces $\mathrm{E}_0[\xi_1]\in (0,\infty]$ so it is incompatible with \ref{H1}.
When equality holds in \ref{H2} the condition is known as Cram\'er's condition. 
Both \ref{H1} and \ref{H2} are well-known from fluctuation theory for L\'evy processes; an equivalent form of \ref{H1} is given in \eqref{eq:H1} below. We will also need the so-called renewal function $\nu_+$  (c.f.~\cref{sec:Levy} below for a definition that requires some notation from fluctuation theory and \cite[Section VI.4]{Ber-Book}), which is non-negative, increasing, and continuous.\smallskip


Our first theorem gives necessary and sufficient existence criteria for $+\infty$ being an entrance boundary.
\begin{theorem}[ $+\infty$ is an entrance boundary]~\label{thm:1}
The time-changed L\'evy process has $+\infty$ as an entrance boundary if and only if
\begin{align*}
	\text{\textup{\ref{H1}} holds \quad and }\quad \int_{\cdot}^{\infty}\frac{\nu_+(y)}{R(y)}  \dd y<\infty. 
\end{align*}
\end{theorem}
Here and throughout the paper $\int_{\cdot}^{\infty}$ means the integrability is determined only in a neighbourhood of infinity,  i.e.\ the integral is finite for some (then all, as $\frac{\nu_+}{R}$ is locally bounded) lower limit of integration $b>0$. \smallskip

We next turn to the regular boundary case. Since Feller's early works on boundary classifications for diffusions the notion of regular boundary point has seen a number of definitions that are all equivalent for stochastic processes with continuous sample paths. In the discontinuous setting different notions for regular boundary points have been used. In what follows we want to generalise a notion from the study of positive self-similar Markov processes (see Rivero \cite{Riv05}, \cite{Riv07}). We recall that for pssMps the origin is called a regular boundary point if $(\P_z)_{z>0}$ can be extended to $(\P_z)_{z\geq 0}$ such that the scaling property continues to hold.
If we want to extend a similar notion of regularity to general time-changed L\'evy processes then we need to establish a more general scaling property. 
This is one of the main contributions of this article.\smallskip

For fixed $z\in\R$ and $R:\R\to(0,\infty)$ define a path transform
$\Phi^R_z$ on the space $D([0,\infty),\oR)$ of \cadlag paths in $\oR$ by
\begin{equation}\label{eq:Phi-z}
    \Phi^R_z(w)_t = z+ w_{h_z(t)}, \qquad t\ge 0, 
\end{equation}
where $h_z(t) = \inf\big\{s\ge 0\colon \int_0^s g(z,w_u)\dd u>t\big\}$ and $g:\R\times\overline\R\to[0,\infty)$ is the measurable map defined by
$$
	g(z,x) = 
	\begin{cases}
		\frac{R(x)}{R(z+x)} \quad &\text{if }x\in\R\\
		1 &\text{if }x=\infty
	\end{cases}.
$$
By convention $\Phi^R_z(w)_t =+\infty$ for $t\ge  \int_0^{\infty} g(z,w_u)\dd u$. 
We denote the family of maps $\{\Phi^R_z, z\in \bR\}$ by $\Phi^R$. 
Again, to focus on the behaviour around $+\infty$ we assume throughout this work that  
\begin{equation}\label{eq:R-z-condition}
    \liminf_{x\to-\infty} \frac{R(x)}{R(z+x)}>0, \quad \forall z\in \bR. 
\end{equation}

\begin{definition}[$R$-scaling invariance]\label{defn:Phi-ss}
A stochastic process $X$ on $\bR$ with law $(\mathbb{P}_x, x\in \bR)$ is called \emph{$R$-scaling invariant} if 
\begin{equation}
 \text{for every $x,z\in \bR$ the transform }\Phi^R_z(X) \text{ under $\P_x$ has distribution }\mathbb P_{x + z}.
\end{equation}
$R$-scaling invariance is extended to stochastic processes on $(-\infty, \infty]$ if in addition  $X$ and $\Phi^R_z(X)$ have the same distribution under $\P_\infty$ for all $z\in \R$.
\end{definition}
The property of $R$-scaling invariance is somewhat a combination of the classical scaling property of self-similar processes and the shift-invariance of L\'evy processes, i.e. that $\xi$ under $\mathrm P_x$ has the same law as $(\xi-y+x)$ under $\mathrm P_y$ for any $x,y\in\bR$.
\begin{example}
 Choosing $R\equiv 1$ the translation invariance property of a L\'evy process is equivalent to the $1$-scaling property.   
\end{example}
Next, time-changing a L\'evy process yields an $R$-scaling invariant process. For $R\equiv 1$ this of course yields again the previous example.
\begin{example}
    A L\'evy process time-changed as in \eqref{eq:time-change} is $R$-scaling invariant; see \cref{prop:Phi-z} below.
\end{example}
Before turning to the use of $R$-scaling invariance for regular boundary point $+\infty$ we make the connection to the classical scaling self-similarity property:
\begin{example}
For positive self-similar Markov processes the Lamperti time-change representation (see e.g.\ \cite[Chapter~13]{Kyp-Book}) shows that the new invariance property is nothing but the scaling property. To see why let $\alpha>0$ and $R(x) = \ee^{x/\alpha}$. With the time-changed $X$ from \eqref{eq:time-change}, the exponential $Y=\exp(-X)$ is a positive self-similar Markov process of index $\alpha$ according to Lamperti's representation. First note that cancellation gives $h_z(t)=c^{-\alpha } t$ for any $z\in\R$ so that $X$ satisfies the new invariance property: 
\begin{align*}
	\Phi^R_z(X) = (z+X_{\ee^{z/\alpha} t})_{t\geq 0} \text{ under }\mathrm P_x \quad \overset {(d)} = \quad (X_t)_{t\geq 0} \text{ under }\mathrm P_{x + z}.
\end{align*}
Reformulated in terms of $Y$ and substituting $c=\exp(-z)$ gives the scaling property $(c Y_{c^{-1/\alpha}t}^{(x)})_{t\geq 0}\overset{(d)}{=} (Y_t^{(cx)})_{t\geq 0}$, i.e.\ the classical $\alpha$-self-similarity. 
\end{example}

Summarizing,, a L\'evy process time-changed as in \eqref{eq:time-change} is $R$-scaling invariant; for some particular functions of $R$, the converse in also true: by the well-known Lamperti representation, the 
processes that has the $R$-scaling property with $R(x)=\ee^{x/\alpha}$ are the $R$-time-changed L\'evy processes. 
In light of this, we introduced the concept of $R$-scaling invariance to identify what a canonical regular extension should be, similar to canonical extensions of pssMps at zero (compare \cite{Riv05}, \cite{Riv07}). 
%
In general an extension at infinity can enter $\R$ according to any distribution.
We characterise all possible extensions that 
are in addition $R$-scaling invariant.
To see why this is natural let us first check that the extension in the natural entrance case is automatically $R$-scaling invariant:
\begin{example}
    If $+\infty$ is an entrance boundary then the extension constructed in \cref{thm:1} satisfies the extended $R$-scaling invariance, see \cref{prop:P-infty-ss} below.
\end{example}
%
Our main theorem on regular boundary points identifies necessary and sufficient conditions under which the time-changed L\'evy process can be extended to an $R$-scaling invariant process on $(-\infty,+\infty]$.
\begin{theorem}[$+\infty$ is a regular boundary]\label{thm:2}
There is an $R$-scaling invariant recurrent extension $\overline{X}$ of $X$ from  \eqref{eq:time-change} if and only if 
	\begin{align*}
	\text{\textup{\ref{H2}} holds\quad and }\quad \int_{\cdot}^{\infty}\frac{\ee^{\theta y}}{R(y)} \dd y<\infty. 
\end{align*}
If $\mathrm{E}_0[\ee^{-\theta \xi_1}] = 1$, then $\overline{X}$ is the unique $R$-scaling invariant  extension such that the excursion measure $\mathbf{n}$  away from $+\infty$ is supported by paths that leave $+\infty$ continuously; if $\mathrm{E}_0[\ee^{-\theta \xi_1}] < 1$, then $\mathbf{n}$ is supported by paths that leave $+\infty$ by a jump and $\mathbf{n}(\mathbf{e}_{0+} \in \dd x) = C \ee^{\theta x} \dd x$ for a certain $C>0$.
\end{theorem}
There exists at most one $\theta>0$ such that $\mathrm{E}_0[\ee^{-\theta \xi_1}] = 1$; so the $R$-scaling invariant extension with continuous entrance, if exists, is unique. There are possibly uncountable many $\theta$ with 
$\mathrm{E}_0[\ee^{-\theta \xi_1}] < 1$. For each such $\theta$, when $\int_{\cdot}^{\infty}\frac{\ee^{\theta y}}{R(y)} \dd y<\infty$, there exists a unique $R$-scaling invariant extension with jumping-in measure $\ee^{\theta x} \dd x$ (up to multiplication of a constant). Note that, although it is possible that an excursion leaves $+\infty$ by a jump, the recurrent extension $\overline{X}$ we obtained from the theorem above is still Feller with c\`adl\`ag paths. In particular, $\overline{X}$ leaves from $+\infty$ continuously and is quasi-left-continuous.\smallskip

Recall that, when $\infty$ is an entrance boundary, the Feller property implies that 
$\bP_x \Rightarrow \overline{\mathbb{P}}_{\infty}$, as $x\to \infty$. 
For the regular boundary case under Cram\'er's condition, the theorem below obtains a corresponding convergence result, which describes a dichotomy in terms of the  integrability condition
\begin{equation}\label{eq:H-exp}
	\mathrm{E}\big[|\xi_1| \ee^{-\theta \xi_1}\big] <\infty.
\end{equation}

\begin{theorem}[Convergence to $\mathbf{n}$]\label{prop:n-cv}\label{thm:convergence to n}
Suppose that  
$\mathrm{E}\big[\ee^{-\theta \xi_1}\big] =1$ 
and $\int_{\cdot}^{\infty} \frac{\ee^{\theta y}}{R(y)}\dd y<\infty$.  
If \eqref{eq:H-exp} holds, then there exists a constant $c(\theta)\in (0,\infty)$ such that
\[
    c(\theta) \ee^{\theta x} \mathbb{P}_x \;\Rightarrow\; \mathbf{n}, \quad \text{as } x\to \infty,
\]
where $\mathbf{n}$ is the excursion measure uniquely characterised by \cref{prop:first}. Conversely, if \eqref{eq:H-exp} does not hold, then
$$
    \ee^{\theta x} \mathbb{P}_x \;\Rightarrow\; \mathbf{0}, \quad \text{as } x\to \infty,
$$
with $\mathbf{0}$ a null measure. 
Both limits are in the sense of weak convergence of finite measures away from neighbourhoods of the identical $\infty$ function.
\end{theorem}

The constant $c(\theta)$ will be given in \eqref{eq:c-theta}. 
For the case of self-similar processes, the role of \eqref{eq:H-exp} is discussed in \cite{Riv07}.

\subsection{Techniques used for the proofs}
For the sufficiency direction of the theorems we can use two classical constructions that have been both used in examining the entrance boundaries of positive self-similar Markov processes. One method was used by Fitzsimmons \cite{Fit06} and Dereich--D\"oring--Kyprianou \cite{DDK17}, while the other by Bertoin--Savov \cite{BerSav11} and Barczy--Bertoin \cite{BarBer11}. Both methodologies trace back to the general theory of Markov processes under the names \emph{Kuznetsov measures} and \emph{quasi-processes}. The latter can be constructed in a two-sided manner (see Mitro \cite{mitro}) and is intricately linked to overshoots, whereas the former is based on excessive functions and Kolmogorov's extension theorem.
For the proof of \cref{thm:1} we will use the more explicite Mitro's construction, but the arguments could equally be carried out using the Kuznetsov measures.
For the sufficiency part of \cref{thm:2}, we will use the Kuznetsov measure construction. 
Let us stress that the quasi-processes approach in Barczy--Bertoin \cite{BarBer11} requires an additional assumption \eqref{eq:H-exp}, which is proved to be not necessary by \cref{thm:2}. We also give this quasi-process approach in \cref{sec:exc} under Assumption \eqref{eq:H-exp}, which is more practical to prove \cref{thm:convergence to n}.\smallskip

For the necessity part of our theorems we rely crucially on the $R$-scaling invariance property and establish a connection with the potential measure and the mean occupation measure, in \cref{prop:P-infty-ss} and \cref{prop:occup-meas} respectively. \smallskip

Proofs of both directions rely on abstract finiteness results on so-called perpetual integrals $\int_0^\infty b(X_s)\dd s$ for Markov processes for which the interested reader is for instance referred  
\cite{BDK,DK16,EriMal05,KolSav19}.
 
\subsection{Examples and connections with previous results}
Let us now give a few examples to illustrate our theorems and make connections to the literature. In \cref{ex:drift-,ex:neg,ex:specpos,ex:fin2mom,ex:stable} we consider L\'evy processes that satisfy certain additional conditions; in \cref{ex:stable,ex:pssMp} we cover  specific functions $R$.  All L\'evy processes below are assumed to be non-lattice.
\begin{example}[$\xi$ drifts to $-\infty$]\label{ex:drift-} 
    If $\xi$ drifts to $-\infty$, then there exists a constant $k>0$ such that 
    $\nu_+(x) = k \mathrm{P} (X_t \le x \text{ for all } t\ge 0)$, see for instance \cite[Proposition~VI.17]{Ber-Book}. Therefore, the integral test in \cref{thm:1} is equivalent to 
 \[
 \int^{\infty}_{\cdot} \frac{1}{R(y)} \dd y<\infty. 
 \]
\end{example}
\begin{example}[Spectrally negative]\label{ex:neg} 
 Let $\xi$ be spectrally negative (i.e.\ with absence of positive jumps) with $\mathrm{E}[ \xi_1] = 0$.
 Then $\nu_+(x) = x$ and hence the integral test in \cref{thm:1} becomes 
 	\[
	\int^{\infty}_{\cdot} \frac{y}{R(y)} \dd y<\infty. 
	\]
\end{example}
\begin{example}[Spectrally positive]\label{ex:specpos}
    Let $\xi$ be spectrally positive with $\mathrm{E}[ \xi_1] \leq 0$. Denote by $\phi(q) = \log \mathrm{E}[\ee^{q \xi_1}], q\ge 0$, the Laplace exponent and by $W$ the scale function uniquely defined by the Laplace transform
	\[
	 \int_{0}^{\infty} \ee^{-qy} W(y)\dd y = \frac{1}{\phi(q)}, \qquad q\ge 0. 
	\]
The time-changed L\'evy process $X$ killed at first passage below zero is called in  \cite{FLZ21} a continuous-state non-linear branching process with branching mechanism $\phi$ and branching rate $R$. In this case, because the subordinator $H_-$ is a unit drift, the Condition~\ref{H1} is satisfied automatically (see \eqref{eq:H1} below). 
Moreover, due to the relation \cite[Equation (9.7)]{Kyp-Book} between the scale function $W$ and the potential measure, we have $W = \nu_+$. Combining this with the change of variables $u+y\to x$, we can rewrite the 
	integral test \eqref{eq:H0} in terms of the scale function:
	\begin{equation}
		\int_{\cdot}^\infty \frac{W(x-b)}{R(x)} \dd x <\infty,\quad \forall b>0.
	\end{equation}
	Recall from \cite[page 197]{Ber-Book} that $W(x) \asymp \frac{1}{x \phi(1/x)}$, so the integral test above is equivalent to 
	\[
	\int_{\cdot}^{\infty} \frac{1}{x \phi(1/x) R(x)} \dd x <\infty. 
	\]
	This recovers \cite[Theorem~3.1]{FLZ21}.
\end{example}

\begin{example}[Finite second moment]\label{ex:fin2mom}
    If $\bE[\xi_1]=0$ with finite second moment $\bE[\xi_1^2]<\infty$, we know from \cite[Theorems~7-8]{DonMal02} that \ref{H1} holds and  
	$\nu_+(y)\sim c y,y\to +\infty,$ for some finite constant $c>0$. 
	Then the integral test is simply 
	\[
	\int^{\infty}_{\cdot} \frac{y}{R(y)} \dd y<\infty. 
	\]
 When $\xi$ is a Brownian motion, this recovers Feller's classical result on entrance boundary from $+\infty$ for SDEs without drift. 
\end{example}

\begin{example}[Stable jump processes]\label{ex:stable}
 Let $\xi$ be an $\alpha$-stable process with $\alpha\in (0,2)$ (c.f.\ \cite{KypPar-Book}) and $R$ a continuous strictly positive function. 

(1) \emph{entrance boundary:}
 If $\alpha\in (0,1]$, then \ref{H1} is not satisfied because the first moment of $\xi_1$ is not finite. By \cref{thm:1}  $+\infty$ is not an entrance boundary for $X$. If $\alpha\in (1,2)$, then $\nu_+(x) = c x^{\alpha \rho}$ with a certain $\rho\in [1-1/\alpha,1/\alpha]$ and $H_-$ in \eqref{eq:H1} is a stable subordinator with index $\alpha \rho$; see \cite{KypPar-Book}, Section 3.5. Then $\mathrm{E}[H_-(1)]=\infty$, unless $\xi$ is spectrally positive, so \ref{H1} does not hold if negative jumps are present. 
If $\xi$ has only positive jumps, then $\alpha \rho =\alpha -1$ and thus $+\infty$ is an entrance boundary if and only if 
        \[
	\int^{\infty}_{\cdot} \frac{y^{\alpha-1}}{R(y)} \dd y<\infty. 
	\]
These results partially encompass \cite[Theorem~2.2]{DK20}.  

(2) \emph{regular boundary:}
When $\xi$ is spectrally positive and stable with $\alpha\in (0,1)\cup (1,2)$, we have 
$\mathrm{E}_0[\ee^{-\theta \xi_1}] = \ee^{-\theta^{\alpha}}$ for $\theta\ge 0$, and the exponential moment does not exist for other stable processes. 
So \ref{H2} only holds when $\xi$ is spectrally positive, with the strict inequality $\mathrm{E}_0[\ee^{-\theta \xi_1}]<1$ for every $\theta>0$, but we cannot have the equality. Therefore, for $\xi$ stable and spectrally positive, 
an R-scaling invariant extension exists with jumping-in if and only if $\int_{\cdot}^{\infty} \frac{\ee^{\theta y}}{R(y)} \dd y<\infty$ for some $\theta>0$, but continuous entrance is not possible.  
\end{example}

\begin{example}[Positive self-similar Markov process]\label{ex:pssMp}
    Suppose $R(x)= \exp(x/\alpha)$ with $\alpha>0$. As we have recalled before, by the well-known transformation of Lamperti, $\exp(-X)$ is a positive self-similar Markov process (pssMp). Since  
	$
	\int_{\cdot}^{\infty} \ee^{-\frac{1}{\alpha} y}\nu_+(y) \dd y =\frac{1}{\kappa(0,1/\alpha)}<\infty, 
	$
	where $\kappa$ is the bivariate Laplace exponent of the ladder process for $\xi$, the integral test in \cref{thm:1} always holds. Then $+\infty$ is an entrance boundary for $X$ (or $0$ for $\exp(-X)$) if and only if \ref{H1} holds. 
	This well-known theorem is proved by \cite{CabCha06,CKPV12}, see also \cite[Theorem 13.6]{Kyp-Book}. 
 
 For the existence of an $R$-scaling invariant  recurrent extension of $X$ whose excursions leave $+\infty$ continuously, we have by \cref{thm:2} that a necessary and sufficient condition is Cramér's condition $\mathrm{E} [\ee^{-\theta \xi_1}] =1$ together with  
	\begin{equation}\label{eq:pssMp-rec}
	    	\int_{\cdot}^{\infty} \frac{\ee^{\theta y}}{R(y)} \dd y =	\int_{\cdot}^{\infty} \ee^{(\theta -1/\alpha) y} \dd y  <\infty \quad \iff \quad \alpha\theta <1. 
	\end{equation}
	Applied to $\exp(-X)$ this is the well-known theorem of Rivero \cite{Riv07} and Fitzsimmons \cite{Fit06}.
 On the other hand, \cite{Riv05} (see also \cite[Theorem~1]{Riv07}) showed that, for $0<\theta< 1/\alpha$, there exists a self-similar recurrent extension of $\exp(-X)$  that leaves $0$ by a jump (equivalently an $R$-scaling invariant extension of $X$ leaving $+\infty$ with a jump) if and only if $\mathrm{E} [\ee^{-\theta \xi_1}] < 1$. Moreover, the jumping-in measure is $\mathbf{n}(\mathbf{e}_{0+} \in \dd x) = C \ee^{\theta x} \dd x$ for a certain $C>0$. 
This is in consistent with \cref{thm:2} which confirms the necessity of $\theta< 1/\alpha$.

Let us briefly spell out the continuous case. Consider a Brownian motion with drift $\xi_t = 2 B_t + \lambda t$ and set $R(x)= \exp(x)$.  Note that the exponential of the time-changed drifted Brownian motion $\exp(-X_t)$ is a squared Bessel process with dimension $d =(2 -\lambda)$ killed at zero (see e.g.\ \cite[Chapter~XI]{RevuzYor} and \cite{GoinYor03}). 
 If $\lambda\le 0$, then \ref{H1} is satisfied and the integral test $\int_{\cdot}^{\infty} y \ee^{-y} \dd y<\infty$ holds. By \cref{thm:1}, $+\infty$ is an entrance boundary for $X$. On the other hand, if $\lambda>0$ Cram\'er's condition is satisfied with $\mathrm{E} [\ee^{-\frac{\lambda}{2} \xi_1}] =1$. 
	Then by \cref{thm:2} a continuous $R$-scaling invariant extension of $X$ at $+\infty$ exists if and only if $\lambda\in (0,2)$.
\end{example}

 Note that the analysis for pssMps in Rivero \cite{Riv05,Riv07} (and thus \cite{Fit06} which requires \cite[Lemma~2]{Riv07}) relies crucially on the study of the moments of the perpetual integral $\int_{0}^{\infty} \frac{1}{R(\xi(t))} \dd t$ with $R(x)=\exp(x/\alpha)$. This is an exponential
functional of a Lévy process for which quite a number of methods are available, see in particular \cite{BerYor}. It is not obvious how to directly extend this approach to a non-exponential function $R$.

\section{Time-changed L\'evy processes and the R-scaling invariance property}\label{sec:X}
In this section we establish several fundamental properties for the time-changed L\'evy processes: criteria for explosion, the Feller property, and the $R$-scaling invariance property that is both novel and crucial for the main proofs of this article.

\subsection{Preliminaries: fluctuation theory of Lévy processes}\label{sec:Levy}
Our study requires several constructions from the fluctuation theory of L\'evy processes. To facilitate the readability of our article we will summarise the main tools in this section. 

Recall that we denote by $\mathrm{P}_x$ the law of $\xi+x$ under $\mathrm{P}_0$, a L\'evy process starting from $x$. For simplicity we abbreviate $\mathrm{P}\coloneqq \mathrm{P}_0$. The law of the dual L\'evy process $\widehat \xi\coloneqq -\xi$ will be denoted by $(\widehat{\mathrm{P}}_x, x\in \bR)$. We will use $T_b= \inf\{s>0\colon \xi_s<b\}$ for the first entrance time of the set $(-\infty,b)$. In order to describe the entrance behaviour of time-changed L\'evy processes we need classical results on overshoot distributions of L\'evy processes. To introduce the needed quantities let $(L_t,t\ge 0)$ be the local time at zero of the L\'evy process reflected at its maximum, that is $(\sup_{0\le s\le t} \xi_s -\xi_t)_{t\ge 0}$. The ascending ladder height process of $\xi$ is defined by
\begin{equation}
H_+(t) = \xi_{L^{-1}(t)} =\sup_{0\le s\le L^{-1}(t)} \xi_s, \qquad t\ge 0,
\end{equation}
where $L^{-1}$ is the right-inverse of the local time process $L$. It is well-known that $H_+$ is a (possibly killed) subordinator. Analogously, we define the descending ladder height process $H_-$ as the ladder height process of the dual process $\widehat \xi$. The renewal functions of $H_+$ (resp. $H_-$) will be denoted by $\nu_+$ (resp. $\nu_-$):
\begin{equation}
	\nu_+(x)=  \int_0^{\infty} \mathrm{P} \left(H_+(t)\le x, L^{-1}(t)<\infty\right) \dd t, \qquad x\in [0,\infty). 
\end{equation}
The function $\nu_+$ is finite, continuous and increasing, with $x\mapsto R(x)-R(0)$ sub-additive. 
The renewal functions are crucial for understanding undershoots (or overshoots) over given levels, that is, the amount by which the process first exceeds downwards (or upwards) a given level $z<0$:
\begin{align*}
	z- \xi_{T_{z}}\text{ under } \mathrm P_0\qquad \text{ or equivalently}\qquad -\xi_{T_0}\text{ under }\mathrm P_{z}.
\end{align*}
It is well known (see for instance \cite[Corollary~3]{BerSav11}) that the undershoot distribution converges weakly as $z\to -\infty$ to a non-trivial limiting distribution, the stationary overshoot distribution, if and only if 
\begin{equation}\label{eq:H1}
	\mathrm{E}_0[H_-(1)]\in (0,\infty) \quad \Longleftrightarrow \quad \ref{H1}.
\end{equation}
The fact that this condition is equivalent to \ref{H1} can be found in \cite[Lemma~2]{BerSav11}.
If the descending ladder height process $H_-$ has infinite mean, then the limiting undershoots degenerate to an atom at infinity in probability (see \cite[proof of Lemma 7]{DonMal02}). 
For our purposes we present a stronger version below, which gives the joint convergence of the undershoots distribution and the pre-undershoot distribution.
\begin{lemma}[{\cite[Lemma~3]{BerSav11}}]\label{lem:overshoot}
	Suppose that \ref{H1} holds. Then the following weak convergence holds, as $z\to \infty$:
	\begin{equation}\label{eq:overshoot}
	\mathrm{P}_z (\xi_{T_0-}\in \dd y_1 , -\xi_{T_0}\in \dd y_2) \Rightarrow \rho (\ddd y_1, \ddd y_2),
	\end{equation}
	where $\rho$ is the probability measure on $[0,\infty)^2$ given by 
\begin{equation}
\rho (\ddd y_1, \ddd y_2) 
= \frac{1}{\mathrm{E}[H_-(1)]} 
	\Big(\nu_+(y_1)\Pi(-x-\dd y_2)\dd y_1  +a_- \,\delta_0(\ddd y_1) \delta_0(\ddd y_2)\Big),
\end{equation} 
where $a_-\ge 0$ is the drift coefficient of the subordinator $H_-$. 
\end{lemma}

We will call the limiting measure $\rho$ the stationary over- and undershoot measure. Note that, when $\xi$ is spectrally positive, $H_-(t)=t$ is just a linear drift. Then \eqref{eq:H1} is automatically satisfied and $\rho=\delta_{(0,0)}$ is just a Dirac measure. The marginals of $\rho$ shall be denoted by $\rho_1$ and $\rho_2$. Note that the measures $\rho, \rho_1, \rho_2$ degenerate if $\mathrm{E}[H_-(1)]=+\infty$ only due to the normalising factor. Without the normalising factor the (infinite) measures are still useful tools to study the fluctuations (see \cite{DT23} where the convergence was also significantly strengthened).\smallskip

A second crucial application of the renewal functions are L\'evy processes conditioned to be positive (or negative), which is discussed in detail in \cite{ChaumontDoney}.
It is well-known that the renewal function $\nu_+$ is harmonic for the L\'evy process killed on the negative half-line: 
\[
\nu_+(x) = \mathrm{E}_x[\nu_+(\xi_t) \mathbf{1}_{T_0>t}] \quad \text{for all } x> 0, t\ge 0. 
\]
For $x>0$, in light of Doob's $h$-transform this allows to define a change of measure by
\begin{equation}
\mathrm{P}^{\uparrow}_x |_{\mathcal{F}_t} = \frac{\nu_+(\xi_t)}{\nu_+(x)}\mathbf{1}_{\{T_0>t\}}\mathrm{P}_x |_{\mathcal{F}_t} . 
\end{equation}
Using a limiting procedure one can show that $\mathrm{P}^{\uparrow}_x$ should be viewed as the law of the L\'evy process $\xi$ conditioned to stay positive. Indeed, in case of $\mathrm{P}_x(T_0=\infty)>0$, one can identify $\mathrm{P}^{\uparrow}_x$ with the conditional law $\mathrm{P}_x(\cdot \mid T_0=\infty)$.

Correspondingly, we define the law of a L\'evy process $\xi$ conditioned to stay negative: for $x> 0$,  
\begin{equation}\label{eq:P-downarrow}
\mathrm{P}_{-x}^{\downarrow}|_{\mathcal{F}_t} = \frac{\nu_-(\xi_t)}{\nu_-(x)}\mathbf{1}_{\{\widehat{T}_0 >t\}}\mathrm{P}_x |_{\mathcal{F}_t}, ~\text{where}~ \widehat{T}_0 = \inf\{s\ge 0\colon \xi_s >0\}. 
\end{equation} 
Let $\widehat{\mathrm{P}}^{\uparrow}$  be the law of the dual L\'evy process $-\xi$ conditioned to stay positive. 
If $\xi^{\downarrow}$ has law $\mathbb{P}_{-x}^{\downarrow}$, then $-\xi^{\downarrow}$ has the law of $\widehat{\mathrm{P}}^{\uparrow}_x$. 

\subsection{Explosion time}\label{sec:explosion}
Recall that $X$ is obtained by time-changing the L\'evy process $\xi$ with a strictly positive continuous function $R$ on $\R$ as defined in \eqref{eq:time-change}, and that $\P_x$ denotes the law of $X$ started in $x$. We also suppose that \eqref{eq:R-Condition} holds.

\begin{lemma}\label{lem:explosion}
If $\xi$ oscillates or diverges to $-\infty$, then $X$ does not  explode. 
\end{lemma}
\begin{proof}
By assumption we have $\liminf_{s\to\infty} \xi_s = -\infty$. Since \eqref{eq:R-Condition} and the continuity of $R$ implies that  $\inf_{x\in(-\infty, M)} \frac{1}{R(x)}>0$ for any $M\in \mathbb{R}$, it is clear that 
$\int_{0}^{\infty} \frac{1}{R(\xi_s)} \dd s =\infty$ $\mathrm{P}_x$-a.s.\ for all $x\in \mathbb{R}$. This perpetual integral gives the explosion time for the time-changed process.  
\end{proof}

Let us borrow a notion from Kolb and Savov \cite{KolSav19}. 
For a Borel set  $E\subseteq \mathbb{R}$,  we say that $\bR\setminus E$ is $\mathrm{P}_x$-transient if 
$
	\mathrm{P}_x\left(\exists t \ge 0\colon \xi_s\in E, \forall s>t \right) =1. 
$
\begin{proposition}\label{prop;explosion}
    The following three statements are equivalent: 
    \begin{enumerate}
        \item $X$ explodes in finite time a.s., i.e.\ $\mathbb{P}_x (\zeta<\infty) =1$; 
        \item $\int_0^{\infty} \frac{1}{R(\xi_s)} \dd s<\infty \quad \mathrm{P}_x\text{-a.s.}$;
        \item\label{little condition} $\lim_{t\to \infty} \xi_t =+\infty$ a.s.\ and there exists a Borel set  $E$ with $\bR\setminus E$ being $\mathrm{P}_x$-transient, such that
        \[
\int_{E} \frac{1}{R(y)} U(\ddd y)<\infty, 
\]
where $U$ is the potential measure of $\xi$. 
    \end{enumerate}
    If these equivalent conditions above are not satisfied, then $X$ a.s.~does not explode, i.e.\ $\mathbb{P}_x (\zeta<\infty) =0$. 
\end{proposition}
\begin{proof}
The equivalence of $(i)$ and $(ii)$ is clear from the time-change. 
Since $R$ is continuous and positive, and \eqref{eq:R-Condition} holds, we have $(ii)\Rightarrow \lim_{t\to \infty} \xi_t =+\infty$ a.s. Then the equivalence of $(ii)$ and $(iii)$ follows from Theorem~2.1 in \cite{KolSav19}.
If $ \mathrm{P}_x (\int_0^{\infty} \frac{1}{R(\xi_s)} \dd s<\infty ) <1$, then by a zero-one law \cite[Lemma~5]{DK16}, the probability is zero. 
\end{proof}

Some special cases of condition $(iii)$ of \cref{prop;explosion} can be found in \cite[Section~3]{KolSav19} and \cite{BDK,DK16,EriMal05}. Let us mention a few: 
\begin{itemize}
\item If $\mathrm{E}[\xi_1]\in (0,\infty)$, then by the renewal theorem we deduce that   
	\[
		\int_{\mathbb{R}} \frac{1}{R(y)} U (\ddd y)<\infty  \iff \int_{\cdot}^{\infty} \frac{1}{R(y)} \dd y<\infty. 
	\]
By \cref{prop;explosion} we have explosion under this condition. 
However, we stress that $\int_{\cdot}^{\infty} \frac{1}{R(y)} \dd y<\infty$ is not always a necessary condition; a counterexample is given in \cite[page 8]{KolSav19}. 
	\item If $\mathrm{E}[\xi_1]\in (0,\infty)$ and $\xi$ possesses a local time then  
	\[
		\int_0^{\infty} \frac{1}{R(\xi_s)} \dd s<\infty \quad \mathrm{P}_x\text{-a.s.}\  \iff \int_{\cdot}^{\infty} \frac{1}{R(y)} \dd y<\infty. 
	\]
	Note that this is not always true if $\xi$ does not possess a local time; see a counterexample in \cite[page 8]{KolSav19}. 
	\item If $\mathrm{E}[\xi_1]\in (0,\infty)$ and $R$ is ultimately non-decreasing then 
	\[
		\int_0^{\infty} \frac{1}{R(\xi_s)} \dd s<\infty \quad \mathrm{P}_x\text{-a.s.}\  \iff \int_{\cdot}^{\infty} \frac{1}{R(y)} \dd y<\infty. 
	\]
\end{itemize}


\subsection{Continuity at the initial state}\label{sec:conti}
It turns out that under mild assumptions on the behaviour of $\xi$, the time-changed Lévy process can inherit continuity in its initial law from $\xi$.
We shall prove this in two results below.

Let $\oR$ be equipped with the metric $\rho(x,y)= \ee^{-x}-\ee^{-y}$,
whose induced topology is locally compact with countable base, and thus Polish. In particular, the point $+\infty$ is now at finite distance from the rest of the space. Let $D([0,t],\oR)$ be the Skorokhod space of càdlàg maps $w:[0,t]\to \oR$, equipped with Skorokhod's $J_1$ topology, which is induced by the metric
$$
	d_t(v,w) = \inf_{\lambda \in \Lambda_t} 
	\Big(
		\sup_{s\in[0,t]} \rho(v_s, w_{\lambda(s)})
		\vee
		\sup_{s\in[0,t]} \abs{\lambda(s) - s}
	\Big),
$$
where $\Lambda_t$ is the set of strictly increasing bijections $:[0,t]\to[0,t]$. Skorokhod's $J_1$ topology on $\D\coloneqq D([0,\infty),\oR)$ is then induced by the metric
$$
	d(v,w) = \int_0^\infty \ee^{-t}d_t(v,w)\dd t.
$$
Let $f:\R\to(0,\infty)$ be a continuous map, extended to $\oR$ by an arbitrary assignment of $f(\infty)=1$. In order to study the path transformation $\psi : w \mapsto w\circ \eta(w)$, where
$$
	\eta(w)_t = \inf\Big\{s > 0 : I(w)_s > t \Big\}, 
	\qquad I(w)_s =\int_0^s f(w_r)\dd r,
$$
we will need to restrict our attention to certain subspaces of $D([0,\infty),\oR)$, because $\psi(w)$ is not necessarily càdlàg.\footnote{This happens if $\liminf_t w_t < \limsup_t w_t$ while at the same time $\eta(w)_t = \infty$ for some $t<\infty$, in which case the transformed path $\psi(w)$ does not have a left limit at $t$.} 

\begin{proposition}\label{cty prop 1}
Let $G\subset \D$ be the set of paths for which $I(w)_s<\infty$ for all $s<\infty$ and $I(w)_\infty = \infty$. 
Then the map
$
	\psi|_G : (G, d)\to (\D,d)
$
is continuous.
\end{proposition}
\begin{proof}
Because $f>0$, the map $:s\mapsto I(w)_s$ is strictly increasing everywhere; this ensures that $\eta(w)$ is continuous everywhere.
Because $I(w)_s$ is finite at all finite times, we have that $\eta(w)$ is strictly increasing on $[0,I(w)_\infty)$; since $I(w)_\infty = \infty$ it follows that $\eta(w)$ is strictly increasing everywhere.

These properties actually allow us to show continuity of $\psi$ in the usual Skorokhod space $D([0,\infty),\R)$, which clearly implies the same in $\D$.

Specifically, the map $I$ from $D([0,\infty),\R)$ to the space of continuous functions in $D([0,\infty),\R)$, equipped with the topology of uniform convergence on compact intervals, is continuous, as shown for example in \cite[Thm.~11.5.1]{Whitt}. Then continuity of $\eta$ follows from \cite[Cor.~13.6.4]{Whitt}.
Thus Whitt's continuity of composition map \cite[Theorem 3.1]{Whitt80} yields that $\psi$ is continuous at $w$.
\end{proof}

The next continuity result is more involved, as it allows the time-changed paths to explode.

\begin{proposition}\label{cty prop 2}
Let $L\subset \D$ be the set of paths that have limit $\lim_{t\to\infty} w_t = \infty$ and for which $I(w)_s<\infty$ for all $s<\infty$. 
Then the map
$
	\psi|_L : (L, d)\to (\D,d)
$
is continuous.
\end{proposition}
\begin{proof}
Fix $w\in L$ and let $w^n$ be a sequence of paths in $L$ converging to $w$ under $d$.
Let $\tau^n = \sup\{t\ge 0 : \rho(w^n_t,\infty)\ge\eps\}$, and $\tau = \sup\{t\ge 0 : \rho(w_t,\infty)\ge\eps\}$. We shall first show that
\begin{align}\label{mandinga}
	\limsup_n \tau^n \le \tau.
\end{align}
Equivalently, there exists a sequence $(\lambda_n)$ in $\Lambda$ such that $\lambda_n\to\textrm{id}$ uniformly and $w^n\circ\lambda_n \to w$ uniformly under $\rho$. Suppose for contradiction that there is some $c>0$ such that $\limsup_n \tau^n > \tau + c$. Then for any $s\in[\tau,\tau+c)$, there are infinitely many $n$ such that
$
	\rho(w_s,\infty)<\eps \le \rho(w^n_s,\infty).
$
Thus
$$
	\limsup_n \rho(w_s, w^n_s) \ge 
	\limsup_n \abs{ \rho(w_s,\infty) - \rho(w^n_s,\infty) } > 0.
$$
Since this holds for all $s\in[\tau,\tau+c)$, an interval of positive length independent of $n$, we conclude that $\limsup_n d(w, w^n)>0$, which contradicts the assumption that $w^n\to w$ under $d$; thus \eqref{mandinga} holds. 
It follows that for any positive $c$, there is $n_c$ such that for all $n\ge n_c$, $\tau^n < \tau+c$ and thus
$$
	\sup_{s\ge\tau+c} \rho(w_s, w^n_s) \le \sup_{s\ge\tau+c} (\rho(w_s, \infty) + \rho(\infty, w^n_s)) < 2\eps.
$$
For that same choice of $c$, noting that $\eta(w)_t\ge\tau+c$ for any $t\ge I(w)_{\tau+c}$,
we have
\begin{align*}
&	\sup_{t\ge I(w)_{\tau+c}} \rho(\psi(w)_t,\psi(w^n)_t)\\
 &\le \sup_{t\ge I(w)_{\tau+c}}\big(\rho(w^n\circ\eta(w^n)_t, w^n\circ\eta(w)_t) + \rho(w\circ\eta(w)_t, w^n\circ\eta(w)_t)\big)\\
	&< \sup_{t\ge I(w)_{\tau+c}}\rho(w^n\circ\eta(w^n)_t, w^n\circ\eta(w)_t) + 2\eps.
\end{align*}
Because \eqref{mandinga} implies uniform convergence of $w^n\circ\lambda_n\to w$ for some sequence $(\lambda_n)$ as above, we have that $I(w^n)\to I(w)$ pointwise. In particular, we can take a possibly larger $n$, and some $b>0$ such that
$$
	t \ge b + I(w)_{\tau+c} 
	\Rightarrow t \ge I(w^n)_{\tau+c}
	\Rightarrow t \ge I(w^n)_{\tau^n}
	\Rightarrow \eta(w^n)_t\ge\tau^n,
$$
so that $\sup_{t\ge b + I(w)_{\tau+c}} \rho(\psi(w)_t,\psi(w^n)_t) < 2\eps + 2\eps$. We have found a finite time $t_0 \coloneqq b + I(w)_{\tau+c}$ such that for any $t>0$,
$$
	d_t(\psi(w^n), \psi(w)) \le d_{t_0}(\psi(w^n), \psi(w)) + \sup_{t\ge t_0} \rho(\psi(w)_t,\psi(w^n)_t) < d_{t_0}(\psi(w^n), \psi(w)) + 4\eps.
$$
It remains to note that since $\tau+c<\infty$, and we can take $b$ arbitrarily close to 0, we have that $t_0 < I(w)_\infty \wedge I(w^n)_\infty$, which means that $d_{t_0}(\psi(w^n), \psi(w))\to0$ as $n\to\infty$ by virtue of the same arguments as in the proof of \cref{cty prop 1}. So $\lim d_t(\psi(w^n), \psi(w)) < 4\eps$, and since $\eps>0$ is arbitrary we have that $\psi(w^n) \to \psi(w)$ under $d_t$. This holds for all $t>0$, which implies that $\psi(w^n) \to \psi(w)$ under $d$.
\end{proof}

\begin{remark}
The statement of \cref{cty prop 2} does not hold without the condition that all paths $(w^n_t, t\ge 0)$ have the same limit as $t\to\infty$ as $(w_t)$. A counterexample inspired by that of Caballero, Lambert, and Uribe Bravo \cite[Example 1]{CLU} is as follows: let $f$ be integrable, $w=\mathrm{id}$, and $w^n = w\circ\1_{[0,n]} + \1_{(n,\infty)}$.
Then $\sup_{s\le t}\rho(w^n_s,w_s)\to0$ as $n\to\infty$ for all $t$, which implies $w^n\to w$ under $d$, but for $t>I(w)_\infty$ it holds that $\sup_{s\le t}\rho(\psi(w^n_s),\psi(w_s))\ge\rho(1,\infty)$.

Caballero et.~al.'s solution in the general case is to require the uniform convergence equivalent to \eqref{mandinga}, which they define via a metric $d_\infty$. Under the tighter condition of $d_\infty(w^n,w)\to0$, which we deduce from the ``shared limit'' condition, Caballero et.~al.~show that $d(\psi(w^n),\psi(w))\to0$.
\end{remark}

\begin{theorem}\label{cty hunt thm}
Let $\xi$ be a Hunt process on $\R$, with almost surely infinite lifetime, and law $\mathcal P_x, x\in\R$, and let $\P_x$ denote the law of the time-changed process $\psi(\xi)$.
Suppose that one of the following conditions holds:
\begin{enumerate}
	\item $\mathcal P_x(I(\xi)_\infty = \infty)=1$ for all $x\in \R$;
	\item $\mathcal P_x(\lim_{t\to\infty}\xi_t = \infty)=1$ for all $x\in \R$.
\end{enumerate}	
Then, in $(\D,d)$, weak continuity points of $x\mapsto\mathcal P_x$ are also weak continuity points of $x\mapsto \P_x$.
\end{theorem}
\begin{proof}
The fact that $\{\xi_s, 0\le s<t\}$ is almost surely contained in a compact set in $\R$ for every $t<\infty$ is well-known, see Blumenthal and Getoor \cite[Prop.~I.9.3]{BG}; since $f$ is continuous, it is bounded on compact sets, and thus we have that $I(\xi)_s<\infty$ for all $s<\infty$ almost surely. Under condition (i), $\xi$ is almost surely contained in the set $G$ defined in \cref{cty prop 1}, and so that result in combination with the continuous mapping theorem \cite[Thm.~3.4.3]{Whitt} applied to $\psi$ yields the required continuity. Similarly \cref{cty prop 2} applies under condition (ii).
\end{proof}

Note that condition (i) includes the case that $\xi$ almost surely converges to a finite limit in $\R$. However, if $\xi$ is only recurrent, $I(\xi)_\infty$ may still be finite. For more precise conditions on (in)finiteness of $I(\xi)_\infty$ see \cite{BDK}.

\begin{lemma}\label{lem:X-conti}
The law $\P_x$ of the time-changed Lévy process from \eqref{eq:time-change} is weakly continuous in $(\D,d)$ at all $x\in\R$.
\end{lemma}
\begin{proof}
Since $\xi$ is a Lévy process, $\P_x$ is weakly continuous everywhere in $\R$ under the uniform topology on $\D$, and thus also under $d$.
If $\xi$ oscillates or diverges to $-\infty$, then by \cref{lem:explosion} we have $\mathcal P_x(I(\xi)_\infty = \infty)=1$ for all $x\in \R$ and \cref{cty hunt thm} yields the result. If $\xi$ diverges to $\infty$ then \cref{cty hunt thm} applies directly.
\end{proof}

\subsection{A new invariance property for time-changed Lévy processes}
In this section we connect the time-changed L\'evy processes to the $R$-scaling invariance property introduced in \cref{defn:Phi-ss}.


We start with a study of the path transformations $( \Phi^R_z , z\in \bR)$ used in the definition of the $R$-scaling invariance property. In what follows in case there is no ambiguity about $R$ we drop the superscript and write $\Phi_z:= \Phi^R_z$ to simplify the notation. 

The path transformation $\Phi$ can be used to formulate a new invariance property of time-changed L\'evy processes that turns out to be surprisingly useful.

\begin{proposition}\label{prop:Phi-z}
Let $z,x\in \bR$ and $X$ a time-changed L\'evy process from \eqref{time-change eq} of law $\mathbb{P}_x$ with lifetime $\zeta$. Then $\Phi_z(X)\sim \mathbb{P}_{x+z}$ and the lifetime of $\Phi_z(X)$ is $\zeta' = \int^{\zeta}_0 \frac{R(X_s)}{R(z+X_s)} \dd s$. 
\end{proposition}

\begin{proof}
	Let $X'\coloneqq  \Phi_z(X)$ with $\zeta' = \int^{\zeta}_0 \frac{R(X_s)}{R(z+X_s)} \dd s$. 
 Changing variables by \eqref{eq:dh_z} yields that 
 \[
 \int_0^{\zeta'} R(X'_s)\dd s
 =\int_0^{\zeta'} R(z + X_{h_z(s)})\dd s =   \int_0^{\zeta} R(X_u)\dd u.
 \]
 Write $X= \xi_{\eta(t)}$ with $\xi$ a L\'evy process starting from $x$, by changing variables we have 
 \[
\int_0^{\zeta} R(X_u)\dd u = \int_0^{\zeta} R(\xi_{\eta(u)})\dd u  =  \int_0^{\infty} R(\xi_{t})\frac{1}{ R(\xi_{t})}\dd t =\infty.  
 \]
Moreover, for $t\ge 0$, set  
\[
J'(t) = \int_0^{t\wedge \zeta'} R(X'_s)\dd s\quad\text{and}\quad  I'_t= \inf\{s>0\colon J'(s)\!>\!t\}. 
\]
Since $\int_0^{\zeta'} R(X'_s)\dd s =\infty$, we have  $\lim_{t\to \infty} J'(t) =\infty$ and hence $J'(I'_t)=t$ for all $t\ge 0$. It follows that $	\frac{\dd I'_t}{\dd t}=  \frac{1}{R(X'_{I'_t})}$ and thus $I'_t = \int_0^t \frac{1}{R(X'_{I'_s})} \dd s$, $\forall t\ge 0$. 
Then we deduce the identity 
\[
J'(t)= \inf\{s>0\colon I'_s\!>\!t\} 
= 
\inf\bigg\{r\ge 0\colon \int_0^r \frac{1}{R(X'_{I'_s})} \dd s>t\bigg\}, \quad \forall t\ge 0. 
\]
 
Therefore, with $Z'_t\coloneqq  X'_{I'_t}$, we have the representation $X'_t\coloneqq  Z'_{J'(t)}, t\ge 0$, and it suffices to prove that  $Z'$ is a L\'evy process with distribution $\mathrm{P}_{x+z}$, which implies that $X'\sim \mathbb{P}_{x+z}$ by definition.

Write $X= \xi_{\eta(t)}$ with $\xi$ a L\'evy process starting from $x$. 
Then we have 
\[Z'_t = X'_{I'_t}= z + \xi_{\eta(h_z(I'_t))}.\] 
Using the chain rule and the definitions of $\eta$, $h_z$, and $\Phi_z$, we have 
\begin{align*}
\frac{\dd \eta(h_z(I'_t))}{\dd t}
&= \frac{\dd \eta}{\dd t}\bigg|_{h_z(I'_t)}\frac{\dd h_z}{\dd t}\bigg|_{I'_t}
\frac{\dd I'_t}{\dd t} = R(X_{h_z(I'_t)})\frac{R(z+ X_{h_z(I'_t)})}{R(X_{h_z(I'_t)})} \frac{1}{R(X'(I'_t))}=1.
\end{align*}
	It follows that $Z'_t = z+\xi_t$. As $\xi\sim \mathrm{P}_x$, we conclude that $Z'\sim \mathrm{P}_{x+z}$, which completes the proof. 
\end{proof}
In the following we use synonymously $U^X$ for the potential measure and the potential operator, that is
\[
U^X (x,A)=
 \mathbb{E}_x \left[\int_0^{\zeta} \mathbf 1_{A}(X_t) \dd t \right]\quad \text{and}\quad
U^X f (x)= 
 \mathbb{E}_x \left[\int_0^{\zeta} f(X_t) \dd t \right]. 
\]

\begin{proposition}\label{prop:ss-potential}
The process $X$ is $R$-scaling invariant if and only if 
the potential measure satisfies 
	\begin{equation}\label{eq:ss-potential}
	    	R(z+y) U^X (x,z+\dd y)
	= R(y) U^X (x+z, \dd y) \quad \text{for all } z\in \bR.  
	\end{equation}
\end{proposition}
\begin{proof}
    	Let $z\in \bR$. 
	Using change of variables $u= h_z(s)$ with $\frac{\dd s}{\dd u}= \frac{R(X_u)}{R(z+X_u)}$, we have 
	\begin{align}
		 \mathbb{E}_x
		\bigg[ \int_0^{\zeta} f ( (\Phi_z X)_s) \dd s\bigg]
		&= \mathbb{E}_x
		\bigg[ \int_0^{\infty} \mathbf{1}_{\{h_z(s)<\zeta\}} f (z+ X_{h_z(s)}) \dd s\bigg]\notag \\
		&= \mathbb{E}_x
		\bigg[ \int_0^{\infty}  \mathbf{1}_{\{u<\zeta\}}f (z+ X_{u})  \frac{R(X_u)}{R(z+X_u)} \dd u\bigg] \notag \\
		&= \int_{\bR}\frac{ f(z+y) R(y)}{R(z+y)} U^X(x,\dd y) . \label{eq:Phi-pot}
	\end{align}
	
	If $X$ is $R$-scaling invariant, then 
	\[
		 \mathbb{E}_x
		\bigg[ \int_0^{\zeta} f ( (\Phi_z X)_s) \dd s\bigg]
  =  \mathbb{E}_{x+z}\bigg[ \int_0^{\zeta} f(X_s) \dd s\bigg] = \int_{\bR} f(y) U^X (x+z, \dd y). 
	\]
 Combing the two expressions yields the identity \eqref{eq:ss-potential}. 
 
	Conversely, suppose that the potential measure satisfies \eqref{eq:ss-potential}, then the computation above shows that 
	\[
	 \mathbb{E}_x
	\bigg[ \int_0^{\zeta} f ((\Phi_z X)_s) \dd s\bigg]
	= \int_{\bR}\frac{ f(y) R(y-z)}{R(y)} U^X(x,z+\dd y)
	=  U^X f (x+z). 
	\]
Then the $R$-scaling invariance of $X$ follows. 
\end{proof}

\section{Coming down from infinity}\label{sec:P-infty}
In this section we study the question of when a time-changed Lévy process can be made to \emph{come down from infinity}, and prove \cref{thm:1}. 
By definition we are concerned with the non-explosive case, and study sufficient and necessary conditions for the existence of a law $\P_\infty$ that extends the Markov family $(\P_x)_{x\in\R}$ in the sense of \cref{def:down from infty}.  

 Note that, our definition of coming down from infinity is in a different form from the classical literature on diffusion processes; for example, in \cite[Chapter 23]{Kallenberg} it is defined that $+\infty$ is an entrance boundary, if $X$ does not explode and there exist $b\ge 0$ and $t>0$, such that 
 \[
 \liminf_{x\to \infty} \mathbb{P}_x (T_b\le t) >0. 
 \]
 For processes with no negative jumps, a discussion of equivalent conditions can be found in \cite{FLZ21}. 
 However, our work is concerned with the more general situation where negative jumps appear, such that these results cannot apply; we provide a discussion on this problem in an accompanying paper \cite{BDS-2}. 
Therefore, in this work we do not use their characterisation of coming down from infinity which is no longer valid in the general case, but give an alternative definition, \cref{def:down from infty}, in the sense that the Markov process can be extended to include the boundary point $+\infty$.

\subsection{Sufficient conditions}\label{sec:P-infty-suff}

In this section we shall prove the sufficiency direction of \cref{thm:1} by constructing a law $\bP_{\infty}$ that satisfies \cref{def:down from infty}.
We use ideas from the study of self-similar Markov processes combined with results on path integrals for general Markov processes.
We suppose for the remainder of this subsection that the Lévy process $\xi$ satisfies \ref{H1}. 

In what follows we will work with càdlàg processes indexed by the entire real-line $\bR$, which are well-known to be a useful tool for studying the entrance behaviour from the boundary of the state-space in a purely probabilistic way. 
We first recall the (spatially) \emph{stationary L\'evy process indexed by $\R$} studied by Bertoin and Savov \cite{BerSav11} based on the stationary over- and undershoot measure $\rho$ from \cref{lem:overshoot}, whose law is denoted by $\mathcal{P}$. 
Let $(Z_t)_{t\in\R}$ be the canonical process under $\mathcal{P}$. 
If we define the canonical backwards process $\xi'_t\coloneqq Z_{(-t)-}, t\geq 0,$ and the canonical forwards process $\xi''_t\coloneqq Z_t, t\geq 0$, then $Z$ is the stationary L\'evy process indexed by $\R$ if under $\mathcal P$ it holds that
\begin{align*}
	(\xi'_0, -\xi''_0)\sim \rho,\quad (\xi'',\xi') \sim \mathrm{P}_{-\xi_0''}\otimes  \widehat{\mathrm{P}}^{\uparrow}_{\xi_0''}.
\end{align*}
In words: sampled from the stationary over- and undershoot distribution $\rho$, the L\'evy process runs to the right and the conditioned dual process runs to the left.
Due to the transience of $\widehat{\mathrm{P}}^{\uparrow}$, and since \ref{H1} implies that the underlying L\'evy process either drifts to $-\infty$ or oscillates, we know that trajectories under $\mathcal P$ are issued from $+\infty$ at time $-\infty$ and either drift down to $-\infty$ or oscillate. The importance of the measure $\mathcal P$ is the spatial invariance from \cite{BerSav11}: roughly speaking, seen from any first passage time the process has the same distributional property, L\'evy process forwards, conditioned process backwards, under- and overshoots distributed according to $\rho$. Formally, this is stated as follows. 
\begin{lemma}[{\cite[Theorem~2 and Corollary~3]{BerSav11}}]\label{lem:BerSav}
	For $b\in\R$, let  $T_b=\inf\{t\in\R: Z_t<b\}$ and $Y_1= Z_{T_b-}-b$,  $Y_2= b - Z_{T_b}$. Then, under $\mathcal P$, $(Y_1,Y_2)\sim \rho$ and, conditionally on $(Y_1,Y_2)= (y_1, y_2)$, the backwards process  $(Z_{(T_b - t)-}-b)_{t\ge 0}$ has law $\widehat{\mathrm{P}}^{\uparrow}_{y_1}$, and the forwards process $(Z_{T_b+t}-b)_{t\ge 0}$ is an independent L\'evy process with law $\mathrm{P}_{-y_2}$. 
\end{lemma}
We can view $Z$ as a process starting from $+\infty$ at initial time $-\infty$.  
In order to get a process indexed by $[0,\infty)$ started from $+\infty$, we will time-change $Z$ in a similar way as in \eqref{eq:time-change}. We define, under $\mathcal P$, for $t\ge 0$,  
\begin{align}\label{time-change eq}
	X_t = Z_{\eta(t)},\quad \text{ where }\eta(t) = \inf\left\{ s\in\R\colon \int_{-\infty}^{s} \frac{1}{R(Z_r)}\dd r  >t\right\}. 
\end{align}
Denote by $\mathbb{P}_{\infty}$ the law of $X$ under $\mathcal{P}$. 
The lifetime of $X$ is given by $\zeta= \int_{-\infty}^\infty \frac{1}{R(Z_r)}\dd r$, which was shown in \cref{prop;explosion} to be equal to $\infty$ almost surely under \ref{H1}.

To ensure that the process $(X_t = Z_{\eta(t)}, 0\le t <\zeta)$ is non-trivial, we need 
\begin{equation}\label{eq:T0}
\mathcal{P}\left(\int_{-\infty}^{s} \frac{\dd r}{R(Z_r)}<\infty\right) = 1,\quad \text{for some (then all) } s\in\R. 
\end{equation}
Here, the finiteness of the integral for some finite $s$ is equivalent to the finiteness for all finite $s$, due to the local boundedness of $R$. 
Let $\mathcal{E}$ denote expectation under $\mathcal{P}$. 

\begin{lemma}\label{lem:H0}
The integrability condition 
\begin{equation}\label{eq:H0}
	\int_b^{\infty}\frac{\nu_+(y)}{R(y)}  \dd y<\infty ,\quad \text{for some (then all) } b\in\R, 
\end{equation}
is equivalent to 
\begin{equation}\label{eq:H0-bis}
	\mathcal{E}\left[\int_{-\infty}^{T_b} \frac{1}{R(Z_r)}\dd r\right]<\infty ,\quad \text{for some (then all) } b\in\R.  
\end{equation}
Under this condition, we have \eqref{eq:T0}. 
\end{lemma}

\begin{proof}
We first notice that, since $\frac{v_+}{R}$ is locally bounded, if the integral in \eqref{eq:H0} is finite for some $b\in \bR$, then indeed it is finite for every $b\in \bR$.  

Time-reversing and using \cref{lem:BerSav}, we have 
\begin{equation}\label{eq:P-Tb}
\mathcal{E}\left[\int_{-\infty}^{T_b} \frac{1}{R(Z_r)}\dd r \right]  = 	\widehat{\mathrm{E}}^{\uparrow}_{\rho_1}\left[\int_{0}^{\infty} \frac{1}{R(\xi_r+b)}\dd r\right] \quad \text{for all } b\in \mathbb{R}.
\end{equation}
	We use \cite[Theorem~1(ii) and Equation~(5)]{BerSav11} to deduce that 
	\[
\widehat{\mathrm{E}}^{\uparrow}_{\rho_1}\left[\int_{0}^{\infty} \frac{1}{R(\xi_r+b)}\dd r\right] =\frac{1}{\mathrm{E}[H_-(1)]}\int_{0}^{\infty}\frac{ \nu_+(y)}{R(y+b)}\dd y \quad \text{for all } b\in\R.
	\]
 Since $\nu_+$ is increasing and $\nu_+-\nu_+(0)$ is sub-additive, we have   
 \[
\frac{\nu_+(y+b) -\nu_+(b)+\nu_+(0)}{\nu_+(y+b) } \frac{\nu_+(y+b) }{R(y+b)} \le  \frac{\nu_+(y)}{R(y+b)} \le \frac{\nu_+(y+b)}{R(y+b)}.
 \]
 It follows that 
 \[
 \int_{0}^{\infty}\frac{ \nu_+(y)}{R(y+b)}\dd y \le \int_{b}^{\infty}\frac{ \nu_+(y)}{R(y)}\dd y
 \]
 and that 
 \[
 \int_{0}^{\infty}\frac{ \nu_+(y)}{R(y+b)}\dd y \ge \frac{\nu_+(2b) -\nu_+(b)+\nu_+(0)}{\nu_+(2b) }\int_{2b}^{\infty}\frac{ \nu_+(y)}{R(y)}\dd y
 + \int_{0}^{b} \frac{ \nu_+(y)}{R(y+b)} \dd y,
 \]
 where we have used the fact that $\frac{\nu_+(y+b) -\nu_+(b)+\nu_+(0)}{\nu_+(y+b) } \ge \frac{\nu_+(2b) -\nu_+(b)+\nu_+(0)}{\nu_+(2b) }$ for all $y\ge b$. 
 Therefore, the finiteness of the integral $\int_{0}^{\infty}\frac{ \nu_+(y)}{R(y+b)}\dd y$ is equivalent to \eqref{eq:H0}. 
Since we have by \eqref{eq:H1} that $\mathrm{E}[H_-(1)]\in (0,\infty)$ under \ref{H1}, we conclude the equivalence between \eqref{eq:H0} and \eqref{eq:H0-bis}. 
 
Therefore, under \eqref{eq:H0} we have
 \begin{equation}
\mathcal{P}\left(\int_{-\infty}^{T_b} \frac{\dd r}{R(Z_r)}<\infty\right) = 1,\quad \text{for all } b\in\R. 
\end{equation}
Since $T_b\in \mathbb{R}$ $\mathcal{P}$-a.s., it follows that \eqref{eq:T0} holds. 
  \end{proof}

With \cref{lem:H0} we have checked that, under Condition \eqref{eq:H0}, $t\mapsto I_t =\int_{-\infty}^{t} \frac{\dd r}{R(Z_r)}$ is $\mathcal{P}$-a.s.\ a  well-defined function, which is  continuous, non-negative, and strictly increasing on $\mathbb{R}$. Therefore, the time-changed process $X$ given by \eqref{time-change eq} is indeed non-trivial under Condition~\eqref{eq:H0}. 

Let us give some properties of the process starting from $+\infty$.

\begin{proposition}\label{prop:P-infty-ss}
Let $X$ be defined as in \eqref{time-change eq} with distribution $\mathbb{P}_{\infty}$. 
\begin{enumerate}
    \item 	For $b\in\R$, let  $T_b=\inf\{t\ge 0: X_t<b\}$, $V_b= X_{T_b-}-b$, and $O_b= b-X_{T_b}$. Then $(V_b,O_b)\sim \rho$ and, conditionally on $(V_b,O_b)= (y_1, y_2)$, the process $(X_{t+T_b})_{t\ge 0}$ has law $\mathbb{P}_{b+y_2}$.    
    \item For every $z\in \mathbb{R}$, $\Phi_z^* \mathbb{P}_{\infty}=\mathbb{P}_{\infty}$, where $\Phi_z^* \mathbb{P}_{\infty}$ is the push-forward of $\mathbb{P}_{\infty}$ via $\Phi_z$ defined in \eqref{eq:Phi-z}. 
    \item The potential measure of $X$ satisfies $U^{X}(+\infty, \ddd y) = \frac{C}{R(y)}\dd y$ for some constant $C>0$. 
\end{enumerate}
\end{proposition}
\begin{proof}
    Part (i) is a straightforward consequence of \cref{lem:BerSav}. 

    For part (ii), set $X'= \Phi_z (X)$. 
    Define $T'_b$ and $(V'_b,O'_b)$ correspondingly for $X'$. Then we have the identity $(V'_b,O'_b) = (V_{b-z},O_{b-z})$, and $(X'_{t+T_b})_{t\ge 0} = \Phi_z (X_{t+T_{b-z}}, t\ge 0)$.  
    It follows from (i) and \cref{prop:Phi-z} that $(V'_b,O'_b)\sim \rho$ and, conditionally on $(V'_b,O'_b)= (y_1, y_2)$, the process $(X'_{t+T'_b}-b)_{t\ge 0}$ has law $\Phi_z^*\mathbb{P}_{b- z + y_2} =\mathbb{P}_{b + y_2}$. This description characterises the law of $X'$, i.e.\  $X'$ has distribution $\mathbb{P}_{\infty}$.    

    For part (iii), set $m(\ddd y )\coloneqq  R(y) U^X (\infty, \ddd y)$. 
    we combine \eqref{eq:Phi-pot} and (ii) to deduce that 
    \[
     m_z(\ddd y)
	= m ( \ddd y) \quad \text{for all } z\in \bR,  
    \]	  
    where $m_z$ is the push-forward of $m$ by translation $y\mapsto y+z$. 
    This invariance property implies that $m$ is the Lebesgue measure, up to a constant. 
\end{proof}

Endow $\oR \coloneqq (-\infty,\infty]$ with the metric $\rho$ from \cref{sec:conti}, which makes it Polish.

\begin{proposition}\label{prop:convergence}
Under condition \eqref{eq:H0}, 
$\mathbb{P}_x\to \mathbb{P}_{\infty}$ converges weakly as $x\to \infty$, under the Skorokhod topology of  $(-\infty,\infty]$-valued c\`adl\`ag functions. 
\end{proposition}

\begin{proof}
By \cite[Proposition~7]{DDK17}, it suffices to check the following conditions: 
\begin{enumerate}
\item $\lim_{b\to\infty} \limsup_{x\to\infty} \mathbb{E}_x[T_b]=0$;
\item for every $b>0$, $\mathbb{P}_x\circ X_{T_{b}}$ converges weakly to a limit $\mu_b$ as $x\to \infty$;
\item On $\mathbb{R}$, $x\mapsto \mathbb{P}_x$ is continuous under the weak topology;
\item $\mathbb{P}_{\infty}$-a.s., $X_0=\infty$ and $X_t\in [0,\infty)$ for all $t>0$;
\item For $X\sim \mathbb{P}_{\infty}$, the process $(X_{t+T_b}-b, t\ge 0)$ has distribution $\mathbb{P}_{\mu_b}$. 
\end{enumerate}
Let us start with (i). Using the monotone convergence theorem and \cref{lem:BerSav}, we have
		\begin{align*}
 \mathcal{E}\left[\int_{-\infty}^{T_b} \frac{\dd r}{R(Z_r)}\right]
 &= 
\lim_{y\to \infty} \mathcal{E}\left[ \int_{T_y}^{T_b} \frac{\dd r}{R(Z_r)}\right]\\
 &=
  \lim_{y\to \infty} \mathcal{E}\left[\mathcal{E} \bigg[\int_{T_y}^{T_b} \frac{\dd r}{R(Z_r)} \bigg| Z_{T_y} \bigg]\right] 
			= \lim_{y\to \infty} \mathcal{E}\left[  \mathbb{E}_{Z_{T_y}}[T_b]\right] . 
			\end{align*}
By \cref{lem:H0}, we have  $\mathcal{E}\left[\int_{-\infty}^{T_b} \frac{\dd r}{R(Z_r)}\right]<\infty$ for each $b$ under \eqref{eq:H0}. It follows that  
\[
 \mathcal{E}\left[  \mathbb{E}_{Z_{T_y}}[T_b] \right]
 = \mathcal{E}\left[ \mathcal{E} \bigg[\int_{T_y}^{T_b} \frac{\dd r}{R(Z_r)} \bigg| Z_{T_y} \bigg]\right] \le \mathcal{E}\left[ \mathcal{E} \bigg[\int_{-\infty}^{T_b} \frac{\dd r}{R(Z_r)} \bigg| Z_{T_y} \bigg]\right] <\infty. 
\]
Therefore, by the dominated convergence we have 
\[
\mathcal{E}\left[\int_{-\infty}^{T_b} \frac{\dd r}{R(Z_r)}\right]
 = 
 \lim_{y\to \infty} \mathcal{E}\left[  \mathbb{E}_{Z_{T_y}}[T_b]\right] 
= \mathcal{E}\left[ \lim_{y\to \infty}  \mathbb{E}_{Z_{T_y}}[T_b]\right] .
\]
		Since $Z_{T_y}\eqdis y-Y_2$, where $Y_2$ is a random variable as in \cref{lem:BerSav}, we have $\lim_{y\to \infty} Z_{T_y} = \infty$ a.s.\ and thus
$$
	\limsup_{y\to \infty} \mathbb{E}_{Z_{T_y}}[T_b] = \limsup_{x\to \infty} \mathbb{E}_{x}[T_b] \qquad \mathcal{P}\text{-a.s..}
$$
We conclude that
\[
\limsup_{x\to\infty} \mathbb{E}_x[T_b]
=\mathcal{E}\left[\int_{-\infty}^{T_b} \frac{\dd r}{R(Z_r)}\right].
\]
Letting $b\to \infty$, then we deduce (i) from \eqref{eq:H0} by \cref{lem:H0} and the monotone convergence theorem. (ii) is implied by \cref{lem:overshoot}, and (iii) is proved in \cref{lem:X-conti}. 
(iv) can be seen from the construction of $\mathbb{P}_{\infty}$. Finally, (v) follows from \cref{lem:BerSav}.  
\end{proof}

Let $C_0(\overline{\bR})$ denote the class of continuous functions on $(-\infty,\infty]$ that vanishes at $-\infty$. 
Recall that a transition semigroup $(P_t,t\ge 0)$ has the Feller property, if for all $f\in C_0(\overline{\bR})$, 
\begin{enumerate}
	\item $P_0 f= f$;
	\item $P_t f \in C_0(\overline{\bR})$, for all $t\ge 0$; 
	\item $\lim_{t\downarrow 0}\|P_t f - f\| = 0$. 
\end{enumerate}

\begin{proof}[Proof of \cref{thm:1}; sufficient part]
It suffices to show that the process given by $(\mathbb{P}_x, x\in \overline{\mathbb{R}})$ has the Feller property, which then gives an extension at infinity of the time-changed process; as we are in the case when $+\infty$ is inaccessible, this means coming down from infinity.    

Let us check the three properties above for the semigroup. The first point is clear from the construction of $\bP_{\infty}$, and the second follows from \cref{lem:X-conti} and \cref{prop:convergence}. For the third condition, it suffices to note that, a consequence of right-continuity,
$\lim_{t\downarrow 0} \big|\bE_x [f(X_t)] - f(x)\big|=0$ for each $x\in \overline{\bR}$.
\end{proof}

\subsection{Necessary conditions}
In this subsection, we check that conditions \ref{H1} and \eqref{eq:H0} are also necessary for coming down from infinity. 
More precisely, we are in the case given by \cref{def:down from infty} and we assume that there is a $\mathbb{P}_{\infty}$ such that $(\mathbb{P}_{x}, x\in (-\infty,\infty])$ is a Feller process. 

The key to proving \ref{H1} is the invariance property of the time-changed processes, that we obtained in \cref{prop:Phi-z}.

\begin{proposition}\label{prop:P-infty-inv}
Suppose that the time-changed L\'evy process comes down from infinity.  
If $X\sim \mathbb{P}_{\infty}$, then $\Phi_z(X)\sim \mathbb{P}_{\infty}$ for all $z\in\R$. As a consequence, under $\mathbb{P}_{\infty}$, the undershoot process $(b- X_{T_b})_{b\in \bR}$ is strictly stationary, i.e. $(b- X_{T_b})_{b\in \bR}$ has the same law as $(z+ b- X_{T_{z+b}})_{b\in \bR}$ for each $z\ge 0$.
In particular, for each $b\in \bR$, the undershoot $b- X_{T_b}$ is $\P_\infty$-\as finite and they all have the same law, denoted by $\mu$. 
\end{proposition}

\begin{proof}
As a consequence of the Feller property, we have that $\mathbb{P}_n$  converges weakly to $\mathbb{P}_{\infty}$, as $n\to \infty$, under the Skorokhod topology of $(-\infty, \infty]$-valued \cadlag functions. 
Using Skorokhod's representation theorem, we may assume $X$ is the \as limit of a sequence of processes $X^{(n)}$, with $X^{(n)}\sim \mathbb{P}_{n}$. 

On the one hand, as $\Phi_z(X^{(n)})\sim \mathbb{P}_{z+n}$ by \cref{prop:Phi-z}, we deduce that $\Phi_z(X^{(n)})$ converges to $X$ in distribution. 

On the other hand, the fact that each $X^{(n)}$ is a time-changed L\'evy process implies that, the family $X^{(n)}$ together with their weak limit $X$, all oscillate, or diverge to $-\infty$, or diverge to $\infty$. 
Then by similar arguments as in the proof of \cref{lem:X-conti},  in each case \cref{cty hunt thm} applies and we deduce that $\Phi_z(X)=\lim_{n\to \infty} \Phi_z(X^{(n)})$ almost surely.
 We conclude that 
$\Phi_z(X) \eqdis X$, i.e.\ $\Phi_z(X)\sim \mathbb{P}_{\infty}$. 

Now note that the definition of $\Phi$ implies that	
$$
	z+y -\Phi_z(\omega) (T_{z+y}) =y - \omega(T_{y}) \quad \text{for all } y\in \bR,
$$
as the time-change does not affect the undershoot. Because $\Phi_z(X) \eqdis X$, the second statement follows.
\end{proof}

\begin{corollary}\label{cor:P-infty-inv}
Suppose that the time-changed L\'evy process comes down from infinity.  Let $B$ be a random variable of law $\mu$ as in \cref{prop:P-infty-inv}, independent of the L\'evy process $\xi$. Then the process $(  b-\xi_{T_{b}})_{b\le -\xi_0}$ is strictly stationary under $\mathrm{P}_{-B}$. Moreover, as $b\to -\infty$, the undershoot process $b-\xi_{T_{b}}$ converges weakly to $B$, i.e.\ \ref{H1} holds.   
\end{corollary}

This extends the results for positive self-similar Markov processes  in \cite{CKPV12,CabCha06}. 

\begin{proof}
By the strong Markov property, the process $X'_s \coloneqq  X_{T_0 +s},s\ge 0$ under $\mathbb{P}_{\infty}$ has law $\mathbb{P}_{X_{T_0}}$, with $X_{T_0}\eqdis -B$. Recall that $X'$ is connected with a L\'evy process of law $\mathrm{P}_{-B}$ via a time-change.  
Since the time-change does not modify the undershoot, the undershoot process $(b- \xi_{T_{b}}, b\le \xi_0)$ under $\mathrm{P}_{-B}$ has the same law as  $(b -X'_{T'_{b}}= b- X_{T_{b}}, b\le X_{T_0})$ under $\mathbb{P}_{\infty}$. By \cref{prop:P-infty-inv}, the latter is strictly stationary, and therefore so is the former. 
The second statement follows from the first by \cite[Proposition~1]{CabCha06}; moreover, this implies \ref{H1} by \cite[Corollary~3]{BerSav11}. 
\end{proof}

\begin{proof}[Proof of \cref{thm:1}: the necessary part]
We have proved \ref{H1} in \cref{cor:P-infty-inv}. 
We next turn to \eqref{eq:H0}.
Our assumption of coming down from infinity implies that the time-changed process is Feller on the space $\overline{\bR}$, and thus is quasi-left continuous. Therefore for any fixed $t>0$ it holds that $\bP_{\infty}(\Delta X_t >0)=0$. Then it follows from  \cref{prop:convergence} that, restricted to the space $D([0,t], \overline{\bR})$ endowed with Skorokhod's $J_1$-topology, $\bP_{x}$ converges weakly to $\bP_{\infty}$, as $x\to \infty$. 
Since $+\infty$ is not a trap under $\bP_{\infty}$, we can take $t,b>0$ large enough such that $\bP_{\infty} (\inf_{s\in [0,t]} X_s < b)>0$. 
Then the Portmanteau theorem yields that 
\[
	\liminf_{x\to\infty}\mathbb{P}_{x} (T_b \le t) =	\liminf_{x\to \infty} \bP_{x} (\inf_{s\in [0,t]} X_s < b) \ge \bP_{\infty} (\inf_{s\in [0,t]} X_s < b)>0. 
\]
	By \cite[Theorem 1.1]{BDS-2}, this condition is equivalent to 
\[
	\text{for some } b>0 ~(\text{then for all }b'\ge b),\quad 	\limsup_{x\to \infty} \mathbb{E}_x [T_b] <\infty. 
\]

Since  $\lim_{t\to -\infty}Z_t= \infty$ $\mathcal{P}$-a.s., we have $\lim_{y\to\infty}T_y=-\infty$ $\mathcal{P}$-\as Applying the monotone convergence theorem and \cref{lem:BerSav} then leads to  
\begin{align*}
	\mathcal{E}\left[\int_{-\infty}^{T_b} \frac{\dd r}{R(Z_r)}\right]
	=\limsup_{y\to\infty}\mathcal{E} \left[\int_{T_y}^{T_b} \frac{\dd r}{R(Z_r)}\right]
	= \limsup_{y\to \infty}\mathbb{E}\big[\mathbb{E}_{y-Y_2}[T_b]\big]
	\le \limsup_{x\to \infty}\mathbb{E}_{x}[T_b],
\end{align*}
where $Y_2$ is a random variable as in \cref{lem:BerSav}, and the last inequality is by the reverse Fatou lemma. 
So we have $	\mathcal{E}\left[\int_{-\infty}^{T_b} \frac{\dd r}{R(Z_r)}\right]<\infty$ for some (then all) $b\in\R$.  
This leads to \eqref{eq:H0}, in view of \cref{lem:H0}. 
\end{proof}


\subsection{The speed of coming down from infinity}
Under assumptions \ref{H1} and \eqref{eq:H0}, we have constructed a time-changed L\'evy process $X$ of law $\mathbb{P}_{\infty}$, starting from $+\infty$. 
We now study the speed of coming down from infinity.
Recall that, in particular, \eqref{eq:H1} implies $\gamma\coloneqq -\mathbb{E}[\xi_1]\ge 0$. 
We shall focus on the case $\gamma >0$. 

Let $Z'_s\coloneqq Z_{-s+}$ for $s\ge 0$. By the definition of the time-changed process $X$ under law $\mathbb{P}_{\infty}$, we have that $Z'\sim \widehat{\mathbb{P}}^{\uparrow}_{A'}$ is a L\'evy process conditioned to stay positive associated with $-\xi$, where $A'\sim \rho_1$ has the marginal law of the first coordinate of $\rho$. 
So there is the identity
\begin{equation}\label{eq:G}
Z'_{s-} = X_{G(s)}, \quad s\ge 0, \quad  \text{where } G(s)=\int_{s}^{\infty} \frac{\dd r}{R(Z'_r)}. 
\end{equation}
We need to compare $G$ with the corresponding deterministic function $\varphi\colon \bR_+\to \bR_+$
\[
 b\mapsto \varphi(b)= \int_b^{\infty} \frac{\dd x}{R(x)}. 
\]

In this direction, we first recall a natural generalisation of the notion of a function that is regularly varying at infinity -- see e.g.~\cite{Bul-Book} for an extensive study of this and related classes of functions. 
A function $f$ is \emph{regularly varying with index $\theta\in \bR$}, if 
\[
\lim_{x\to \infty} \frac{f(\lambda x)}{f(x)} = \lambda^{\theta}, \quad\text{ for all }\lambda>0. 
\]
We say $f$ \emph{preserves the equivalence of functions}, if, as $x\to \infty$, 
\begin{equation}\label{eq:preserve-equiv}
	u(x)/ v(x)\to 1, u(x)\to 0 \quad\Rightarrow\quad f(u(x)) / f(v(x)) \to 1. 
\end{equation}
This condition is satisfied if and only if the function $f$ is \emph{pseudo-regularly varying} (also known as \emph{intermediate regularly varying}) in the sense that 
\[
\limsup_{\lambda\to 1} \limsup_{t\to \infty} \frac{f(\lambda t)}{f(t)} =1.
\]
See \cite[Theorem 3.42]{Bul-Book}. In particular, any regularly varying function has property \eqref{eq:preserve-equiv}. 

\begin{proposition}\label{prop:speed}
Suppose that $\gamma\coloneqq -\mathbb{E}[\xi_1]> 0$, that \eqref{eq:H0} holds and that $R$ preserves the equivalence of functions. Then
\[
	\lim_{t\to 0+} \frac{\varphi(X_t)}{\gamma t}  = 1 \quad \mathbb{P}_{\infty}\text{-\as}
\]
\end{proposition}

If, additionally, the inverse function $\varphi^{-1}$ preserves the equivalence of functions, then the conclusion of \cref{prop:speed} leads to 
$\lim_{t\to 0+} \frac{X_t}{\varphi^{-1}(\gamma t)} = 1$ $\mathbb{P}_{\infty}$-\as 
See \cite[Section~7.3]{Bul-Book} for conditions under which the inverse function preserves the equivalence of functions. 
\begin{proof}
		When $\gamma>0$, the probability $\widehat{\bP}^{\uparrow}_{A'}$ is absolutely continuous with respect to $\widehat{\bP}_{A'}$. By the law of large numbers for a L\'evy process under $\widehat{\bP}_{A'}$, we also have  $\bP_{\infty}$-a.s.\  $Z'_s / \gamma s \to 1$, as $s\to \infty$. 
		
		Since $R$ preserves the equivalence of functions,  we have $\bP_{\infty}$-\as  $R(Z'_s) / R(\gamma s) \to 1$. Let $\varepsilon>0$, there exists a random $t_0>0$ large enough such that $|\frac{R(Z'_s)}{R(\gamma s)}-1|<\varepsilon$ for all $s>t_0$. Then, for $t>t_0$, 
\[
	G(t) = \int_t^{\infty}\frac{1}{R(\gamma s)}\frac{R(Z'_s)}{R(\gamma s)} \dd s
	\le \int_t^{\infty}\frac{1}{R(\gamma s)}(1+\varepsilon)\dd s
	= (1+\varepsilon ) \varphi(\gamma t)/\gamma, 
\] 
and similarly $G(t)\ge (1-\varepsilon ) \varphi(\gamma t)/\gamma$. 
We deduce that 
\[
	\lim_{s\to \infty}\frac{\gamma G(s)}{\varphi(\gamma s)} = 1 \qquad \bP_{\infty}\text{-\as}
\]
On the other hand, notice that $\varphi$ preserves the equivalence of functions by an extended Karamata’s Theorem \cite[Theorem~4.3]{Cli94}. Thus we have 
\[
	\lim_{s\to \infty}\frac{\varphi(Z'_s)}{\varphi(\gamma s)} = 1 \qquad \bP_{\infty}\text{-a.s..}
\]
Summarising, we have 
\[
	\lim_{t\to 0+}\frac{\varphi(X_t)}{\gamma t} 
	= \lim_{s\to \infty} \frac{\varphi(X_{G(s)})}{\gamma G(s)}
	=  \lim_{s\to \infty} \frac{\varphi(Z'_s)}{\varphi(\gamma s)}\frac{\varphi(\gamma s)}{\gamma G(s)}
	= 1 
	\qquad \mathbb{P}_{\infty}\text{-\as} \qedhere
\]
\end{proof}
A particular case is as follows.
\begin{corollary}
Suppose that $\gamma\coloneqq -\mathbb{E}[\xi_1]> 0$. 
If $R$ is regular varying with index $\theta>1$, then
$\lim_{t\to 0+} \frac{X_t}{\varphi^{-1}(\gamma t)} = 1$, $\mathbb{P}_{\infty}$-\as, and the speed function $\varphi^{-1}$ is regular varying with index $-\frac{1}{\theta-1}$. 
\end{corollary}

This result strengthens \cite[Theorem~3.4]{FLZ21}, which obtains similar conclusion for time-changed spectrally positive L\'evy processes under certain additional assumptions. 
\begin{proof}
As $x\mapsto x^{\theta}/R(x)$ is slowly varying, we have by Karamata’s Theorem \cite[Theorem~1.5.10]{RegularVariation} that $\varphi$ is regular varying with index $1-\theta$
and  that $\lim_{b\to\infty} \frac{b}{\varphi(b)R(b)}= \theta-1$. 
	By \cite[Theorem~1.5.12]{RegularVariation}, $\varphi^{-1}$ is regular varying with index $1/(1-\theta)$. The statement now follows from \cref{prop:speed}.
\end{proof}

\begin{remark}
    To study the speed for the case $\gamma=  -\mathbb{E}[\xi_1]= 0$ would require a finer limit theorem for L\'evy process conditioned to stay positive. But this  is not available in the current literature, as far as we are aware.
\end{remark}

\section{Recurrent extensions}\label{sec:rec}
Now we turn to studying recurrent extensions at infinity, and to proving \cref{thm:2}. In contrast to the previous section, in order to allow a recurrent extension the time-changed Lévy process must be able to explode in finite time, $\bP_x(T_{\infty}<\infty)>0$, $x\in \bR$. That is, $\infty$ is an exit boundary for the Markov process of law $(\bP_x,x\in \bR)$. 
Recall from the zero-one law in \cref{prop;explosion} that we in fact have $\bP_x(T_{\infty}<\infty)=1$. 

The central object in our construction of the recurrent extension is the excursion of the time-changed L\'evy process away from $\infty$, described by an excursion measure. 
Let $\mathbf{n}$ be a $\sigma$-finite measure on the subspace of $\D([0,\infty),\overline{\bR})$ consist of  c\`adl\`ag excursion functions away from $+\infty$:
\begin{equation}\label{eq:exc-space}
    \Big\{f\colon [0,\infty)\to (-\infty,+\infty]\;\Big|\; f\in \D([0,\infty),\overline{\bR}), \zeta(f)\le \infty, f(t)= +\infty , t\ge \zeta(f)\Big\},
\end{equation}
where $\zeta(f) = \inf\{t>0\colon f(t)=+\infty\}$. 
Let $(\mathbf{e}_t, t\ge 0)$ be the coordinate process under $\mathbf{n}$. 
Define a family of measures $(n_s, s>0)$ on $\bR$ by 
\[
n_s (A)=   \mathbf{n}( \mathbf{e}_s \in A, s< \zeta) , \qquad s> 0, \text{ Borel set }A\subset \bR.
\]
If $(n_s, s>0)$ is an entrance law for the semigroup $(p^X_t,t\ge 0)$ of the Markov process $X$ of law $(\bP_x,x\in \bR)$, i.e.\ $n_s p^X_t = n_{s+t}$ for all $s,t> 0$, 
then we call $\mathbf{n}$ \emph{an excursion measure} (away from $\infty$) of $X$.
For the special case when $e^{-X}$ is a self-similar process, an excursion theory has been established by \cite{Vuo94,Fit06,Riv05,Riv07}.  

In \cref{sec:exc-K} we construct an excursion measure using the Kuznetzov measures. 
Then, based on the study of excursions, we give sufficient and necessary conditions for the existence of a recurrent extension, in \cref{sec:exc-suff,sec:exc-nec} respectively. 
In \cref{sec:exc} we prove \cref{prop:n-cv}, a functional limit result, whose proof relies on an alternate construction of the excursion measure by means of the quasi-process introduced by Barczy and Bertoin \cite{BarBer11}.
\subsection{Construction 
	via the Kuznetsov measure}\label{sec:exc-K}

In this section we provide a construction of the excursion measure based on probabilistic potential theory, in particular the concept of Kuznetsov measures.  We refer to Getoor's monograph \cite{Get90} for the general theory of the Kuznetzov measures and its applications in the study of excessive measures. Our approach is motivated by Fitzsimmons \cite{Fit06} who exploited the compatibility of Kuznetsov measures with random time-changes to construct excursion measures for positive self-similar Markov processes from invariant measures of L\'evy processes. Fitzsimmons' ideas were recently extended by Dereich et al.~\cite{DDK17}.

Throughout this section, we consider a L\'evy process $\xi$ satisfying \ref{H2}, that is:  
$$
	\text{there exists } \theta>0 \text{ such that } \mathrm{E}[\ee^{-\theta \xi_1}] =\ee^{-\Psi(i\theta)}\le 1.
$$
Note that the function $q\mapsto \Psi(iq)$ is convex on the set $\{q\in \bR \colon \Psi(iq)<\infty\}$ and that $\Psi(0)=0$. Therefore the L\'evy process $\xi$ drifts to $+\infty$ under \ref{H2}. 
If equality holds in \ref{H2}, then it is called Cram\'er's condition and by convexity $\theta$ is unique.

Let $(S_t, t\ge 0)$ be the transition semigroup of $\xi$. 
Under \ref{H2}, the measure $m^{\xi} (\ddd y)\coloneqq \ee^{\theta y}\dd y$ on $\bR$ is excessive for $\xi$, i.e.\ $m^{\xi}S_t \le m^{\xi}$ for each $t\ge 0$. Specifically, Cram\'er's condition $\mathrm{E}[\ee^{-\theta \xi_1}] =1$ implies that the measure $m^{\xi}$ is invariant for $\xi$. When $\mathrm{E}[\ee^{-\theta \xi_1}] <1$, we have that $m^{\xi}$ is purely excessive for $\xi$, in the sense that $m^{\xi}$ is excessive and, for every non-negative function $f$ with 
\[
m^{\xi}[f]\coloneqq \int f \dd m^{\xi}<\infty,
\]
$m^{\xi}[S_t f] \to 0$ as $t\to \infty$. 

Let $\W$ be the space of paths $\omega \colon \bR\to \overline{\mathbb{R}} = (-\infty,\infty]$, which are $\mathbb{R}$-valued and \cadlag on an open interval $(\alpha(\omega),\beta(\omega))\subseteq \mathbb{R}$ and  $+\infty$-valued elsewhere. 
As the measure $m^{\xi}$ is excessive, by \cite[Theorems 6.3 and 6.7]{Get90}, there exists a unique sigma-finite measure $\mathcal{Q}$ on the space $\W$, such that, for any $ t_1<\cdots <t_n $, $y_1, \ldots , y_n \in \mathbb{R}$,   
\begin{align*}
	&\mathcal{Q}(\alpha<t_1, Z_{t_1}\in \dd y_1, \ldots, Z_{t_n}\in \dd y_n, t_n <\beta)\\
	&= m^{\xi} (\ddd y_1) S_{t_2-t_1} (y_1, \dd y_2)\cdots S_{t_n-t_{n-1}} (y_{n-1}, \dd y_{n}).
\end{align*}
Moreover, we have $\mathcal{Q}(\alpha>-\infty) = 0$ under Cram\'er's condition and $\mathcal{Q}(\alpha=-\infty) = 0$ when $\mathrm{E}[\ee^{-\theta \xi_1}] <1$. 
The measure $\mathcal{Q}$ is called the \emph{Kuznetsov measure} corresponding to the excessive measure $m^{\xi}$ and $(S_t)$. 

To further understand the measure $\mathcal{Q}$, we make use of (weak) duality of Markov processes, transferred to the level of the Kuznetsov measures.  
Let $(\widehat{S}_t)$ be the semigroup of the dual L\'evy process $\widehat{\xi}= -\xi$. 
Condition \ref{H2} implies that $k \coloneqq  -\log \mathrm{E}[\ee^{-\theta \xi_1}]\ge 0$ and that $(\ee^{\theta \widehat{\xi}_t + k t},t\ge 0)$ is a $\mathrm{P}_x$-martingale. Therefore, one can introduce a change of measure:
\begin{equation}\label{eq:barP}
\overline{\mathrm{P}}_x(A, t<\zeta)
=  \ee^{-\theta x}\widehat{\mathrm{E}}_x \Big[ \ee^{\theta \widehat{\xi}_t}\mathbf{1}_{A} \Big]\quad  \text{for all } A\in \mathcal{F}_t, 
\end{equation}
with $\overline{\mathrm{P}}_x( t<\zeta) = \ee^{- kt}$. 
Write $\overline{\xi}$ for a process of law  $\overline{\mathrm{P}}_x$ and $(\overline{S}_t)$ for its semigroup. 
This change of measure is well-known as the Esscher transform (see e.g.~\cite[Theorem~3.9]{Kyp-Book}); it follows that $\overline{\xi}$ is again a L\'evy process, the so-called tilted L\'evy process, possibly killed (via a jump to $+\infty$) at an independent exponential time $\zeta$ with parameter $k\ge 0$, and the characteristic exponent of $\overline{\xi}$ is 
\[
\overline{\mathrm{E}} [\ee^{iq \overline{\xi}_1}] = \mathrm{E} [\ee^{-(iq+\theta) \xi_1}]
= \exp (\Psi(-q +i\theta)).
\]
When $k=0$, $\overline{\xi}$ drifts to $+\infty$, i.e.\ $\lim_{t\to \infty}\overline{\xi}_t=+\infty$ almost surely. 
This change of measure implies that, for all non-negative measurable $f,g$ and $x\mapsto h(x)\coloneqq  f(x )\ee^{\theta x}$,
\[\int g(x) \widehat{S}_t h(x)\dd x= \int\overline{S}_t f(x)  g(x )\ee^{\theta x}\dd x,\] 
Using this and Hunt's switching identity (see \cite[Theorem~II.5]{Ber-Book}), which is
\[
\int S_t g(x)  h (x)\dd x = \int g(x) \widehat{S}_t h(x) \dd x, 
\]
we deduce the duality relation between semigroups: 
\[ 
\int \overline{S}_t f(x) g(x) \ee^{\theta x}\dd x = \int  S_t g (x) f(x)  \ee^{\theta x}\dd x. 
\]
It follows that 
\begin{align}
\begin{split}
	\mathcal{Q}&(\alpha <t_1, Z_{t_1}\in \ddd y_1, \ldots, Z_{t_n}\in \ddd y_n, t_n <\beta) \\
	& = m^{\xi} (\ddd y_n) \overline{S}_{t_n-t_{n-1}} (y_n, \ddd y_{n-1})\cdots \overline{S}_{t_2-t_1} (y_{2}, \ddd y_{1}).
\end{split}\label{eq:Q-dual}
\end{align}
This is one expression of the time-reversal property, which is that the time-reversal of $\mathcal{Q}$ is the Kuznetsov measure of the tilted process $\overline{\xi}$ with respect to the excessive duality measure. 

Using the time-reversal property, we now derive an integrability condition that allows us to define a time-change of $Z$. Assumption \eqref{eq:H-rec-suff}, appearing in \cref{thm:2}, is crucial for this.
\begin{lemma}
	Suppose that \begin{equation}\label{eq:H-rec-suff}
		\int_{\cdot}^{\infty} \frac{\ee^{\theta y}}{R(y)}\dd y<\infty. 
	\end{equation}
 Then it holds that
	\begin{equation}\label{eq:H2'}
		\lim_{t\downarrow \alpha} \int_{(\alpha,t]} \frac{1}{R(Z_s)} \dd s = 0  \qquad \mathcal{Q}\text{-a.e.} 
	\end{equation}
\end{lemma} 
\begin{proof}
By the change of measure, we derive a connection
\[
\overline{U}(\ddd y) = \ee^{\theta y} \widehat{U}(\ddd y),  
\]
where $\overline{U}$ and $\widehat{U}$ are the potential measures of $\overline{\xi}$ and $\widehat{\xi}$ respectively. 
By the renewal theorem we deduce from \eqref{eq:H-rec-suff} that $\int_{\cdot}^{\infty} \frac{1}{R(y)} \overline{U} (\ddd y)<\infty$, 
which implies that  
\begin{equation}\label{eq:H2}
    \overline{\mathrm{P}}_x \left(\int_{[0,\zeta)}\frac{1}{R(\overline{\xi}_s)} \dd s < \infty \right) =1
    \quad \text{for all } x\in \bR. 
\end{equation}
Reversing paths via the time-reversal property \eqref{eq:Q-dual} we deduce that
	\[
	\mathcal{Q} \left(\int_{(\alpha,t]} \frac{1}{R(Z_s)} \dd s =\infty \right)
	= \int_{\mathbb{R}} m^{\xi} (\ddd x) \overline{\mathrm{P}}_x \left(\int_{[0,\zeta)}\frac{1}{R(\overline{\xi}_s)} \dd s =\infty\right)=0.
	\qedhere
	\]
\end{proof}
We next apply a time-change to the Kuznetsov measure $\mathcal{Q}$, using results developed in Kaspi \cite{Kas88}, see also \cite[Chapter 8]{Get90}, and then perform a Palm-measure type restriction to obtain an excursion measure $\mathbf{n}$; for that we shall need the integrability condition \eqref{eq:H2'} above. 
On the measure space $(\W,\mathcal{Q})$, consider the (non-atomic) random measure $B(\omega,\dd t)\coloneqq  \frac{1}{R(\omega_t)}\dd t$ on $\bR$. This is a so-called \emph{homogeneous random measure} (HRM) \cite[Definition (8.2)]{Get90} in the following sense: for any $s\ge 0$ and measurable set $A$ the identity $B(\omega ,A+s) = \int_A \frac{1}{R(\omega_{t+s})}\dd t = B(\theta_s\circ \omega ,A)$ holds, where $\theta_s$ is the shift operator.  

Recall that $X$ and $\xi$ are related by the additive functional $I_t(\omega)= \int_0^t \frac{1}{R(\omega_{r})}\dd r$ on $D([0,\infty), \R)$. The HRM $B$ can be viewed as $I_t$ extended to $\W$, in the sense that for each $\omega\in \W$ and $s\in \bR$, it holds that $I_t(\omega(s+\,\cdot\,) ) = B(\omega, (s,s+t])$.  
Roughly speaking, the HRM $B$ is a change of time on $\bR$ and we can study the time-changed $\mathcal{Q}$ via $B$. 
More precisely, when \eqref{eq:H2'} holds, by \cite[Theorem (2.3)]{Kas88} there exists an entrance law for $X$, given by \cite[Equation (2.9)]{Kas88}:
\begin{equation}\label{eq:n-entrance-Q}
	n_s(A)  =  \mathcal{Q}\left(Z_{\eta_s}\in A, \eta_s\in [0,1) \right), \qquad s>0,
\end{equation}
where $\eta_s = \inf \{u\in \bR\colon \int_{\alpha}^u \frac{1}{R(Z_r)} \dd r >s\}$. 
That is to say, there is a sigma-finite measure $\mathbf{n}$ on the space of excursions $\exc$ (which are indexed by $(0,\infty)$) such that for $0<t_1\le t_2 \le \cdots \le t_n$  
\begin{equation}\label{eq:n}
\mathbf{n}(\exc_{t_1}\!\in\! \dd y_1, \ldots, \exc_{t_n}\!\in\! \dd y_n, t_n\!<\!\zeta)
= n_{t_1} (\ddd y_1) S^X_{t_2-t_1} (y_1, \dd y_2)\cdots S^X_{t_n-t_{n-1}} (y_{n-1}, \dd y_{n}),
\end{equation}
where $(S^{X}_t,t\ge 0)$ is the semigroup of $X$. 

We also observe a dichotomy in the way the excursion $\exc$ under $\mathbf n$ leaves $+\infty$.
\begin{enumerate}
    \item For the case when the L\'evy process $\overline{\xi}$ drifts to $+\infty$ and is not killed (so $k=0$), we know by \eqref{eq:Q-dual} that $\lim_{s\downarrow -\infty} Z_s= +\infty$,  $\mathcal{Q}$-a.e..
By \cite[Corollary~(2.11)]{Kas88}, this property translates to $\exc$ after the time-change and we have  
$ \liminf_{s\downarrow 0} \exc_s =+\infty$, $\mathbf{n}$-a.e..
Therefore, we can finally ``extend'' the measure $\mathbf{n}$ to $D([0,\infty),\oR)$, the space of c\`adl\`ag paths on $[0,\infty)$, by setting $\mathbf{n}(\exc_0\ne \infty)=0$. 
    \item For the other case, when $k>0$, we have that $\lim_{t\downarrow 0} \exc_t\in \bR$, $\mathbf{n}$-a.e.. Indeed, in this case, the process $Z$ is bounded from above, and a finite limit $\lim_{t\downarrow \alpha} Z_t$ exists. Then after the time-change, we have $\mathcal{Q}$-a.e.\ $\lim_{t\downarrow 0} \exc_t = \lim_{t\downarrow \alpha} Z_t$. 
Set $\exc_0$ to be the right limit at time zero, i.e.\ $\mathbf{n}(\exc_0\ne \lim_{t\downarrow 0} \exc_t)=0$. 
In this case, $\exc_0$ is finite and we understand the excursion as ``jumping in'' from $+\infty$. 
\end{enumerate}

Now define the mean occupation measure
\begin{equation}\label{eq:occup-meas-defn}
	m^X(A) =  \mathbf{n}\bigg( \int_0^{\zeta} \mathbf{1}_{A} (\exc_s) \dd s\bigg)= \int_0^{\infty} n_s (A) \dd s \quad \text{for all Borel sets }A\subseteq \bR. 
\end{equation}
It is well-known that $m^X$ is purely excessive for the semigroup $(S^{X}_t,t\ge 0)$ of the process $X$, see e.g.\ \cite[Theorem~5.25]{Get90}. 

\begin{proposition}\label{prop:exc-K-m}
 For any non-negative measurable function $f$, the mean occupation measure $m^X$ given in \eqref{eq:occup-meas-defn} satisfies 
\begin{equation}\label{eq:exc-K-m}
	m^X[f]  = \int_{\bR} \frac{f(y )\ee^{\theta y}}{R(y)}\dd y.
\end{equation}
\end{proposition}
\begin{proof}
By \eqref{eq:n-entrance-Q} and a change of variable, we have the following identity for the mean occupation measure: for every measurable set $A$, 
\[
\int_{[0,\infty)} n_s(A) \dd s = \int \mathcal{Q}(\ddd w)\int_{[0,1)} \frac{\mathbf{1}_{\{Z_{u}(w)\in A\}}}{R(Z_u(w))} \dd u . 
\]
By \eqref{eq:Q-dual}, the right-hand side is equal to $\int_{\bR} \frac{\mathbf{1}_A(y )\ee^{\theta y}}{R(y)}\dd y $. 
\end{proof}

Our next observation connects the special form of the mean occupation measure with the $R$-scaling invariance property as in \cref{defn:Phi-ss} of the excursion measure.
\begin{proposition}\label{prop:occup-meas}
Let  $\mathbf{n}'$ be a $\sigma$-finite excursion measure of $X$, and let $\Phi_z$ be the function defined in \eqref{eq:Phi-z}. 
Then $\mathbf{n}'$ satisfies the invariance property
\begin{equation}\label{eq:n-scaling}
	\Phi_z^{*} \mathbf{n'} = \ee^{-\theta z} \mathbf{n'}, \qquad \text{for all } z\in \bR, 
\end{equation}  
if and only if there exists a constant $C>0$ such that the mean occupation measure $m'$ of $\mathbf{n}'$ is given by 
\begin{equation}\label{eq:occup-meas-lem}
	m'(\ddd y) = C  \frac{\ee^{\theta y}}{R(y)} \dd y. 
\end{equation}
\end{proposition}
\begin{proof}
	Let $z\in \bR$. 
	Changing variables $u= h_z(s)$ with $\frac{\dd s}{\dd u}= \frac{R(\exc_u)}{R(z+\exc_u)}$, we have 
	\begin{align*}
		\Phi_z^{*}\mathbf{n}'
		\bigg[ \int_0^{\zeta} f (\exc_s) \dd s\bigg]
		&= \mathbf{n}'
		\bigg[ \int_0^{\infty} \mathbf{1}_{\{h_z(s)<\zeta\}} f (z+ \exc_{h_z(s)}) \dd s\bigg] \\
		&= \mathbf{n}'
		\bigg[ \int_0^{\infty}  \mathbf{1}_{\{u<\zeta\}}f (z+ \exc_{u})  \frac{R(\exc_u)}{R(z+\exc_u)} \dd u\bigg] \\
		&= \int_{\bR}\frac{ f(z+y) R(y)}{R(z+y)} m'(\ddd y)\\
		&= \int_{\bR}\frac{ f(y) R(y-z)}{R(y)} m'_z(\ddd y), 
	\end{align*}
	where $m'_z$ denotes the push-forward of $m'$ by the translation $y\mapsto y+z$. 
	
	If $\mathbf{n}'$ satisfies \eqref{eq:n-scaling}, then 
	\[
	\Phi_z^{*}\mathbf{n}'
	\bigg[ \int_0^{\zeta} f (\exc_s) \dd s\bigg] =  \ee^{-\theta z}\mathbf{n}'\bigg[ \int_0^{\zeta} f(\exc_s) \dd s\bigg] = \ee^{-\theta z}\int_{\bR} f(y)m'(\ddd y). 
	\]
	We deduce the identity
	\[
	\frac{R(y-z)}{R(y)} m'_z(\ddd y)
	= \ee^{-\theta z} m'(\ddd y) \qquad \text{for all } z\in \bR. 
	\]
	Let $\widetilde{m}(\ddd y) = \ee^{-\theta y}R(y) m'(\ddd y)$. Then it follows that
	$\widetilde{m}_z(\ddd y) = \tilde{m}(\ddd y)$ for all $z\in \bR$,
	where $\widetilde{m}_z$ is the push-forward of $\tilde{m}$ via the translation $y\mapsto y+z$. We conclude that $\widetilde{m}$ is the Lebesgue measure on $\bR$, up to a multiplicative constant. This leads to the desired form of the mean occupation measure $m'$. 
	
	Conversely, suppose that $m'(\ddd y) = C  \frac{\ee^{\theta y}}{R(y)} \dd y$. 
	Then the computation above leads to
	\[
	\Phi_z^{*}\mathbf{n}'
	\bigg[ \int_0^{\zeta} f (\exc_s) \dd s\bigg]
	= \int_{\bR}\frac{ f(z+y)}{R(z+y)} \ee^{\theta y}\dd y 
	=  \ee^{-\theta z}\mathbf{n}'\bigg[ \int_0^{\zeta} f(\exc_s) \dd s\bigg]. 
	\]
	It follows that $m' =\int_0^{\infty} \ee^{\theta z}n^{(z)}_s \dd s$ with $n^{(z)}_s [f]\coloneqq  \Phi_z^{*}\mathbf{n}'[f(\exc_s)\ind{s<\zeta}]$ giving an entrance law for $X$. 
	Recall that the mean occupation measure  $m'$ is purely excessive, and thus by \cite[Theorem (5.25)]{Get90} such a representation of $m'$ is unique. Comparing with the definition of $m'$ as in \eqref{eq:occup-meas-defn}, this allows us to identify $(\ee^{\theta z}n^{(z)}_s ,s>0)$ with the entrance law of $\mathbf{n}'$, and thus $\Phi_z^{*}\mathbf{n}'= \ee^{-\theta z}\mathbf{n}'$. 
\end{proof}

Using this invariance property we give further descriptions of the excursion measure. 
\begin{proposition}
Let $\mathbf{n}$  be the excursion measure built in \eqref{eq:n}. 
\begin{enumerate}
	\item For every $z\in \bR$, $\mathbf{n}$ satisfies the invariance property $\Phi_z^{*}\mathbf{n}= \ee^{-\theta z}\mathbf{n}$. 
	\item There exists $C>0$ such that $\mathbf{n}(T_x <\infty) = C \ee^{\theta x}$ for every $x\in \bR$. 
	\item Either the excursion leaves $\infty$ continuously with $\mathbf{n}(\exc_0 \ne \infty)=0$, or it leaves $\infty$ via a jump, with jumping-in measure $\eta(\ddd x) = C' \ee^{\theta x} \dd x$ on $\bR$, where $C'>0$ and $\mathbf{n} (\,\cdot\,)= \int_{\bR} \mathbb{P}_x (\,\cdot\,) \eta(\ddd x)$. 
\end{enumerate}
\end{proposition}
\begin{proof}
Combining the two previous propositions yields the invariance property. 

To prove the second statement, we consider for $x\in \bR$ the subspace of excursions whose global minimum is smaller than $x$, that is $A_x := \inf\{\mathbf{e} \text{ an excursion} \colon T_{x}(\mathbf{e})<\infty\}$. We claim that 
\begin{equation}\label{eq:A_x}
    \Phi_z (A_{x+z}) = A_x, \qquad \forall x\in \bR. 
\end{equation}
Indeed, it is clear from the definition of $\Phi_{z}$ that $\Phi_z (A_{x+z}) \subseteq A_x$. On the other hand, \cref{lem:Phiz-inverse} below implies that $A_x\subseteq \Phi_z(\Phi_{-z} (A_{x}))\subseteq \Phi_z (A_{x+z})$. So $\Phi_{z}$ gives a bijection between these two sets and we have \eqref{eq:A_x}. 
Therefore, we have by the invariance property that 
\[
	\mathbf{n} (T_{x}<\infty)=\mathbf{n} \big(\Phi_{z}(T_{x+z}<\infty)\big)= \ee^{-\theta z}\mathbf{n}(T_{x+z}<\infty) . 
\]
Setting $C\coloneqq  \mathbf{n} (T_{0}<\infty)$, we have $\mathbf{n}(T_x <\infty) = C \ee^{\theta x}$. 

The dichotomy between leaving $\infty$ continuously versus jumping in was established during the construction of $\mathbf{n}$. To identify the jumping-in measure denoted by $\eta$, we use the invariance property of $\mathbf{n}$ as well as that for $\mathbb{P}_x$ given by \cref{prop:Phi-z} and deduce that for any bounded functional $f$, 
\[
	e^{-\theta z}\mathbf{n}[f] = \Phi_z^{*} \mathbf{n}[f]
	= \int_{\bR} \mathbb{E}_{x+z} [f] \eta(\ddd x). 
\]
As a consequence, $e^{-\theta x} \eta (\ddd x)$ is invariant with respect to translation, and thus it is the Lebesgue measure up to a constant. 
\end{proof}

\begin{lemma}\label{lem:Phiz-inverse}
For any $z\in \bR$ and $w\in \D$ an excursion away from $+\infty$, we have $\Phi_{-z}(\Phi_z(w))=w$. 
\end{lemma}
\begin{proof}
Recall that 
$h^{(w)}_z(t) = \inf\big\{s\ge 0\colon I(s)>t\big\}$ with $I(s)\coloneqq  \int_0^{s\wedge \zeta} \frac{R(w_u)}{ R(z+ w_u)}\dd u$. 
Since $I( h^{(w)}_z(t)) =t$ for $t< \zeta':= \int_0^{\zeta} \frac{R(w_u)}{ R(z+ w_u)}\dd u$, we deduce by the chain rule that 
\[
\frac{R(w_{h^{(w)}_z(t)})}{ R(z+w_{h^{(w)}_z(t)})} \frac{\dd h_z^{(w)}(t)}{\dd t}  =1, \qquad t < \zeta'. 
\]
Set  $w'\coloneqq \Phi_z(w) = z+ w_{h^{(w)}_z(t)}$ and $I'(t)\coloneqq  \int_0^{t} \frac{R(w'_u)}{R(-z+w'_u)}\dd u$ for $t < \zeta'$. Then we have
\begin{equation}\label{eq:dh_z}
    	\frac{\dd I'(t)}{\dd t} 
	= \frac{R(w'_t)}{R(-z+w'_t)} = \frac{R(z+ w_{h_z^{(w)}(t)})}{R(w_{h_z^{(w)}(t)})}= \frac{\dd h_z^{(w)}(t)}{\dd t}, \qquad t < \zeta'. 
\end{equation}
It follows that $I'(t) = h_z^{(w)}(t)$, $t < \zeta'$. 
Set $h^{(w')}_{-z}(t) = \inf\big\{s\ge 0\colon I'(s)>t\big\}$, then we have 
$h_{-z}^{(w')} (h_z^{(w)}(t))=h_{-z}^{(w')} (I'(t)) = t$. Therefore, 
\[
	\Phi_{-z}(w')_t = -z +z + w_{h_{-z}^{(w')} (h_z^{(w)}(t))} = w_t, \qquad t < \zeta'. \qedhere
\]
\end{proof}

\subsection{Sufficient condition for existence of recurrent extension}\label{sec:exc-suff}

In this section, we suppose that \ref{H2} and \eqref{eq:H-rec-suff} hold.  
Condition \ref{H2} implies $\mathrm{E}[\xi_1]\in (0,\infty)$ and 
\eqref{eq:H-rec-suff} implies that $\int_{\cdot}^{\infty} \frac{1}{R(y)}\dd y<\infty$, which by \cref{prop;explosion} implies explosion. 
Under these conditions we defined an excursion measure $\mathbf{n}$ in \cref{sec:exc-K}.
We aim at building a recurrent extension of $X$, for which we use the classical theory developed in \cite{Blu83,Salisbury86} to construct a Markov process from the excursions generated by $\mathbf{n}$.
 
To this end, let $\mathcal{N}$ be a Poisson random measure on the space $\bR_+\times \D$ with intensity $\mathrm{Leb}\otimes \mathbf{n}$, where as above $\D = D([0,\infty),\oR)$. 
The following lemma ensures that the lengths of excursions in $\mathcal{N}$ are summable.  

\begin{lemma}\label{lem:rec-suff}
Write $\zeta = \inf\{s> 0\colon \mathbf{e}_s =+\infty\}$ for the length of an excursion.  
Suppose that \eqref{eq:H-rec-suff} holds. 
Then
$
	\mathbf{n}[ 1- \ee^{-\zeta}]<\infty. 
$
\end{lemma} 
\begin{proof}
Using \eqref{eq:exc-K-m}, we deduce that 
\begin{align*}
	\mathbf{n}[\zeta\mathbf{1}_{\{T_a=\infty\}}]
 \le  \mathbf{n}\bigg[\int_0^{\zeta} \mathbf{1}_{\{\exc_s\ge a\}}\dd s \bigg] = \int_{\bR} \mathbf{1}_{\{y\ge a\}} m(\ddd y) =  \int_{a}^{\infty} \frac{\ee^{\theta y}}{R(y)}\dd y <\infty. 
\end{align*}
Markov's inequality yields $\mathbf{n}(\zeta > 1, T_a=\infty)<\infty$. It follows that 
\[
	\mathbf{n}(\zeta > 1)\le \mathbf{n}(\zeta > 1, T_a=\infty)+  \mathbf{n}(T_a <\infty)<\infty. 
\] 
Here we have used the fact that $\mathbf{n}(T_a <\infty)$; indeed,  as $m^{\xi}((-\infty, a]) = \int_{-\infty}^{a} \ee^{\theta y} \dd y <\infty$, we have $\mathcal{Q}(T_a<\infty) <\infty$, which implies $\mathbf{n}(T_a<\infty)<\infty$.
 
Therefore we have 
\begin{equation*}
	\mathbf{n}[ 1- \ee^{-\zeta}]\le 
	\mathbf{n}[\zeta\ind{\zeta \le  1}] +\mathbf{n}(\zeta > 1) 
	\le \mathbf{n}[\zeta\ind{T_a=\infty}]+ \mathbf{n}(T_a<\infty)
	+	\mathbf{n}(\zeta >  1)  <\infty.
	\qedhere
\end{equation*}
\end{proof} 

Set $\sigma_t= \int_{[0,t]\times \D}\zeta(\mathbf{e}) \mathcal{N}(\ddd s, \dd \mathbf{e})$ and let
\begin{equation}\label{eq:X-n}
	Y_t\coloneqq  \sum_{(s,\mathbf{e}^{(s)}) \text{ atoms of }\mathcal{N}}\mathbf{e}^{(s)} (t - \sigma_{s-})\mathbf{1}_{\{\sigma_{s-}<t\le \sigma_{s}\}}, \qquad t\ge 0.
\end{equation}
Note that, $\sigma_t$ is \as finite due to \cref{lem:rec-suff}.
Denote by $\overline{\bP}_{\infty}$ the law of the process $Y$. For $x\in \bR$, let $X'\sim \bP_{x}$ and $\zeta' =\inf\{t > 0\colon X'_t =\infty\}$. Write $\overline{\bP}_{x}$ for the law of the process 
\begin{equation}\label{eq:X-RecExtension}
\overline{X}_t = X'_t \mathbf{1}_{\{t<\zeta'\}} + Y_{t-\zeta'}\mathbf{1}_{\{t\ge\zeta'\}}, \qquad t\ge 0. 
\end{equation}

We next use \cite[Theorem~2]{Sal86-Ito} to deduce that $(\overline{\bP}_{x},x\in \overline{\bR})$ gives a right-continuous strong Markov process with excursions described by $\mathcal{N}$, which boils down to checking the following properties of $\mathbf{n}$: by the construction of $\mathbf{n}$,  \eqref{eq:n} holds and either $\mathbf{n}(\exc_0 = \infty) = 0$ (the jumping-in case) or $\mathbf{n}(\exc_0 = \infty) = \infty$ (the continuous entrance case); the fact that $\mathbf{n}[ 1- \ee^{-\zeta}]<\infty$ is proved in \cref{lem:rec-suff}. 
We stress that, although the excursion measure in the jumping-in case satisfies $\mathbf{n}(\exc_0 = \infty)=0$, we still have $Y_0 = \infty$ $\overline{\bP}_{\infty}$-a.s.\ by the right-continuity of paths.  


\begin{proposition}\label{prop:Phiz-n}
If $\overline{X}\sim \overline{\bP}_{x}$ then $\Phi_z(\overline{X})\sim \overline{\bP}_{x+z}$.
If $\overline{X}\sim \overline{\bP}_{\infty}$ then $\Phi_z(\overline{X})\sim \overline{\bP}_{\infty}$.
\end{proposition}

\begin{proof}
	We start with the second statement. Let $\overline{X}\sim \overline{\bP}_{\infty}$ be associated with a Poisson random measure $\mathcal{N}\coloneqq  \sum_{i\ge 1}\delta_{(t_i, \mathbf{e}_i)}$ as above. 
	Set $\widehat{\mathcal{N}}\coloneqq  \sum_{i\ge 1}\delta_{(t_i, \Phi_{z}(\mathbf{e}_i))}$. Then $\widehat{\mathcal{N}}$ is a Poisson random measure on the space $\bR_+\times \D$ with intensity $\mathrm{Leb}\otimes \Phi_{z}^*\mathbf{n} = \ee^{-\theta z}\mathrm{Leb}\otimes \mathbf{n}$, due to the scaling property \eqref{eq:n-scaling}. 
	
	For every $t\ge 0$, set $\sigma_t= \int_{[0,t]\times \D}\zeta(\mathbf{e}) \mathcal{N}(\ddd s, \dd \mathbf{e})$ and define $\widehat{\sigma}_t$ associated with $\widehat{\mathcal{N}}$ in the same way. 
	Then for every $s\ge 0$, we have 
	\[
	\int_0^{\sigma_s} \frac{R(\overline{X}_u)}{R(z+\overline{X}_u)}\dd u
	= \int_{[0,t]\times \D} \zeta(\Phi_{z}(\mathbf{e})) \mathcal{N}(\ddd s, \dd \mathbf{e}) 
	=  \int_{[0,t]\times \D} \zeta(\widehat{\mathbf{e}}) \widehat{\mathcal{N}}(\ddd s, \dd\widehat{ \mathbf{e}}) = \widehat{\sigma}_s. 
	\]
	Define $h_z$ and $\Phi_{z}$ for $\overline{X}$ as in \eqref{eq:Phi-z}. 
	We deduce that $h_z (\widehat{\sigma}_s) = \sigma_s$ for all $s\ge 0$. 
	Therefore, for any $t\ge 0$, with $s\ge 0$ such that $\{\widehat{\sigma}_{s-}<t\le \widehat{\sigma}_{s}\}$, we have 
	\[
	\sigma_{s-}<h_z(t)\le \sigma_{s}, \quad 
	\text{and} \quad z+\mathbf{e}_s \big(h_z (t)-\sigma_{s-}\big)=  \Phi_z(\mathbf{e}_s) \big(t- \widehat{\sigma}_{s-}\big). 
	\] 
	It follows that 
	\[
	\sum_{(s,\widehat{\mathbf{e}}_s) \text{ atoms of }\widehat{\mathcal{N}}}\widehat{\mathbf{e}}_s (t - \widehat{\sigma}_{s-})\mathbf{1}_{\{\widehat{\sigma}_{s-}<t\le \widehat{\sigma}_{s}\}}
	=z+ \overline{X}_{h_z(t)} = \Phi_z(\overline{X})_t. 
	\]
	Since $\widehat{\mathcal{N}}$ is a Poisson random measure with intensity $e^{-\theta z}\mathrm{Leb}\otimes \mathbf{n}$, we deduce that 
	$ \Phi_z(\overline{X})\sim \overline{\bP}_{\infty}$, completing the proof of the second statement. 
	Combing this and \cref{prop:Phi-z} leads to the first statement. 
\end{proof}

\begin{lemma}\label{one more continuity lemma}
 For any $y_n \to y$ in $\overline{\mathbb{R}}$, it holds that 
 $\overline{\mathbb{P}}_{y_n} \to \overline{\mathbb{P}}_{y}$ weakly. 
\end{lemma}
\begin{proof}
By construction a process of law $\overline{\mathbb{P}}_{y}$ can be viewed as the concatenation of a process of law $\mathbb{P}_{y}$ stopped upon hitting $+\infty$ and a process starting from $+\infty$ as in \eqref{eq:X-n}. Then it suffices to prove the convergence of the first segment, and for each $y\in \mathbb{R}$, this convergence result is given by \cref{lem:X-conti}.  

It remains to study the case $y = +\infty$. 
Let us start with some limit properties of time-changed L\'evy processes. 
For any $a\in \mathbb{R}$, since $\xi$ drifts to $\infty$, we have, as $x\to \infty$,  
\[
\mathbb{P}_x(T_a <\zeta) = \mathrm{P}_{0}(\inf_{s\ge 0}\xi_s \le a-x) \to 0.
\]
We also have, for each $a\in \mathbb{R}$, 
\begin{align*}
\mathrm{E}_0\bigg[\int_0^{\infty} \frac{1}{R(x+\xi_s)} \ind{\xi_s\ge a}\dd s \bigg]
 = \int_a^{\infty} \frac{1}{R(x+y)} U(\ddd y)  \le C\int_{x+a}^{\infty} \frac{1}{R(z)} \dd z.  
\end{align*}
Therefore, for any fixed $\epsilon>0$, 
\begin{align*}
\mathbb{P}_x(\zeta>\epsilon, T_a >\zeta)
\le \frac{1}{\epsilon} \mathrm{E}_0\left[\ind{T_a=\infty} \int_0^{\infty} \frac{1}{R(x+\xi_s)} \ind{\xi_s\ge a}\dd s \right]\to 0, \text{ as } x\to \infty. 
\end{align*}
We deduce that    
$
\lim_{x\to\infty}\mathbb{P}_x(\zeta>\epsilon)= 0. 
$

A recurrent extension starting from $x\in \mathbb{R}$ can be viewed as the concatenation of the segment $(X_t, t< \zeta)$ under law $\mathbb{P}_x$ and a recurrent extension starting from $\infty$. 
Then the fact that $\lim_{x\to\infty}\mathbb{P}_x(\zeta>\epsilon)= 0$ and 
$\lim_{x\to\infty}\mathbb{P}_x(T_a <\zeta)=0$ yields that
 $\bP_x$ converges weakly to a Dirac mass $\delta_{+\infty}$ at a constant process at $+\infty$, as $x\to +\infty$. 
The desired statement follows. 
\end{proof}

\begin{proof}[Proof of \cref{thm:2}: sufficiency]
Build the process $\overline{X}$ as in \eqref{eq:X-RecExtension}. 
By \cref{prop:Phiz-n}, $\overline{X}$ is $\Phi$-invariant. 
With the construction of $\overline{X}$, the proof boils down to checking that $\overline{X}$ is a Feller process satisfying \cref{def:regular boundary}. The Feller property follows from the same arguments as in that of \cref{thm:1}, given the continuity at the initial state given by \cref{one more continuity lemma}, and the rest is clear from the construction of $\overline X$.
\end{proof}

\subsection{Necessary conditions}\label{sec:exc-nec}
In this section we prove the converse direction, supposing that there exists a recurrent extension $\overline{X}$ of law $(\overline{\mathbb{P}}_x, x\in \overline{\bR})$, which satisfies the invariance property as in \cref{defn:Phi-ss}, and deducing that \eqref{eq:H-rec-suff} and \ref{H2} hold. 
In particular, we are in the explosion case of \cref{prop;explosion}.  

By It\^o's excursion theory, one can describe the excursions of $\overline{X}$ as a Poisson point process with characteristic measure $\mathbf{n}$. The following lemma shows that $\mathbf{n}$ is an infinite measure, which implies that $+\infty$ is regular and instantaneous for $\overline{X}$ (see e.g.\ \cite[Proposition~1]{Sal86-Ito}). 


\begin{proposition}\label{prop:Phiz-n-bis}
	Suppose that $\Phi_z^*(\overline{\bP}_{\infty}) = \overline{\bP}_{\infty}$ for all $z\in \bR$.  
	Then $\mathbf{n}$ is an infinite measure, and there exists a constant $\theta> 0$ such that the excursion measure satisfies 		
	\begin{equation*}
		\Phi_z^*(\mathbf{n}) = \ee^{-\theta z} \mathbf{n} \qquad \text{for all } z\in \bR.
	\end{equation*}
\end{proposition}
\begin{proof}
As $\Phi_z^{*} \overline{\bP}_{\infty} = \overline{\bP}_{\infty}$ for all $z\in \bR$, we can follow essentially the same arguments as in the proof of \cref{prop:Phiz-n} to deduce that $\Phi_z^{*}\mathbf{n} =c_z \mathbf{n}$ up to a certain finite constant $c_z>0$ with $c_0=1$. 
Specifically, consider a process $\overline{X}$ of law $\overline{\bP}_{\infty}$.  For every $a\in \mathbb{R}$, the first excursion of $\overline{X}$ (resp.\ $\Phi_z \overline{X}$) that passages below level $a$ has distribution  $\mathbf{n} (\ccdot \mid T_a <\infty)$ (resp.\ $\Phi_z^{*} \mathbf{n} (\ccdot \mid T_{a} <\infty))$. 
Since $\Phi_z \overline{X}$ has the same distribution as $\overline{X}$ by the assumption, we have $\Phi_z^{*} \mathbf{n} (\ccdot \mid T_{a} <\infty) = \mathbf{n} (\ccdot \mid T_{a} <\infty)$, for every $a\in \mathbb{R}$. It follows that $\Phi_z^{*}\mathbf{n} =c_z \mathbf{n}$ up to a certain finite constant $c_z>0$. 
 
 Using \eqref{eq:A_x}, we have 
\begin{equation}\label{eq:c_z}
	c_z \mathbf{n}(T_{x+z}<\infty) =(\Phi_z^*(\mathbf{n})) (T_{x+z}<\infty) = \mathbf{n} \big(\Phi_{z}(T_{x+z}<\infty)\big) =\mathbf{n} (T_{x}<\infty). 
\end{equation}
As $\{T_{x}<\infty\} \supseteq \{T_{x+z} <\infty\}$ for $z\ge 0$, we have $c_z \le 1$ for $z\ge 0$. 
We also deduce from \eqref{eq:c_z} that $z\mapsto c_z$ is non-increasing. 
	
For $z,z'\in \bR$, applying \eqref{eq:c_z} twice leads to 
$
	c_{z'} c_z \mathbf{n}(T_{x+z+z'}<\infty)  =\mathbf{n} (T_{x}<\infty). 
$
On the other hand,
$
	c_{z'+z} \mathbf{n}(T_{x+z+z'}<\infty) =\mathbf{n} \big(\Phi_{z'+z}(T_{x+z+z'}<\infty)\big) = \mathbf{n} (T_{x}<\infty).
$
So we have $c_{z'+z}=c_{z'} c_z$. 
This relation together with monotonicity implies that $z\mapsto c_z$ is an exponential function. 
Setting $\theta \coloneqq  -\log c_1 \ge 0$, it follows that $c_z = \ee^{-\theta z}$.  

It remain to show that $\mathbf{n}$ is an infinite measure. 
suppose by contradiction that $\mathbf{n}$ is a finite measure, then we have $\mathbf{n}[1] =\Phi_z^*(\mathbf{n})[1] =\ee^{-\theta z} \mathbf{n}[1] = \ee^{-\theta z}$, for every $z\in \mathbb{R}$, so $\theta =0$ and $\mathbf{n}[1]=1$. 
It follows from \eqref{eq:c_z} that $\mathbf{n}(T_{x+z}<\infty) = \mathbf{n}(T_{x}<\infty)$ for all $z\in \bR$. Since $\lim_{z\to \infty} \mathbf{n}(T_{z}<\infty) =\mathbf{n}[1]=1$, we deduce that  $\mathbf{n}(T_{x}<\infty) =\mathbf{n}[1] = 1$ for all $x\in \bR$, i.e.\ any excursion a.s.\ reaches below every level $x$.  But recall that we are in the explosion case, so in particular $\lim_{s\to \infty} \xi_s =+\infty$, and so any excursion of the time-changed process has a finite infimum. Thus we have a contradiction. So  $\mathbf{n}$ cannot be a finite measure.
\end{proof}

\begin{proof}[Proof of \cref{thm:2}, necessity]
(1). We first prove condition \eqref{eq:H-rec-suff}.  	
Let $\mathbf{n}$ be the excursion measure of the process $\overline{X}$. Then it satisfies 
\[
	\mathbf{n}[ 1- \ee^{-\zeta}] = \mathbf{n}\bigg[\int_0^{\zeta} \ee^{-s}\dd s \bigg] = \int_0^{\infty} \ee^{-s} \mathbf{n} (\zeta >s) \dd s<\infty.
\]
Set $h(x) = \mathbb{E}_x[\ee^{-\zeta}]$. 
We also have (see \cite[p.480]{Riv05} for a detailed proof), by the Markov property, that
\[
	\mathbf{n}[ 1- \ee^{-\zeta}] = \mathbf{n}\bigg[\int_0^{\zeta} h(\exc_s)\dd s\bigg].  
\]
It suffices to prove that there exists $M>0$ and $b>0$, such that
 \begin{equation}\label{eq:superfinite}
     	\inf_{x>b} \mathbb{P}_x\Big(\zeta <M\Big)=: c_0 >0. 
 \end{equation}
Indeed, if \eqref{eq:superfinite} holds, then we have $h(x) = \mathbb{E}_x[\ee^{-\zeta}] \ge \ee^{-M} \mathbb{P}_x(\zeta <M) = \ee^{-M} c_0 $ for every $x>b$. 
It follows that, by the mean occupation measure obtained from \cref{prop:occup-meas}, 
\[
	\int_b^{\infty} \frac{\ee^{\theta y}}{R(y)} \dd y = \mathbf{n}\bigg[\int_0^{\zeta} \ind{\exc_s >b}\dd s\bigg] \le \frac{\ee^{M}}{ c_0}\mathbf{n}\bigg[\int_0^{\zeta} h(\exc_s)\dd s\bigg] = \frac{\ee^{M}}{ c_0}\mathbf{n}( 1- \ee^{-\zeta}) <\infty. 
\]

It remains to prove \eqref{eq:superfinite}. 
Suppose for contradiction that for every fixed $M>0$, there exists a sequence $x_k \uparrow \infty$ (depending on $M$) such that 
$
	\lim_{k\to \infty} \mathbb{P}_{x_k}(\zeta <M) = 0. 
$
Recall that, for every $a\in \mathbb{R}$, since $\xi$ drifts to $\infty$, we have 
$\mathbb{P}_x(T_a <\zeta) = \mathrm{P}_{0}(\inf_{s\ge 0}\xi_s \le a-x) \to 0$ as $x\to \infty$. 
Therefore 
$$
	\lim_{k\to \infty} \mathbb{P}_{x_k}\Big(\inf_{t\in [0,M]}\xi_t >a, \zeta >M\Big) = 1. 
$$
Because the recurrent extension $\overline{X}$ coincides with the time-changed L\'evy process before hitting $\infty$, it follows that
\[
	\lim_{k\to \infty} \overline{\mathbb{P}}_{x_k}\Big(\inf_{t\in [0,M]}\xi_t \ge a, \zeta >M\Big) = 1 \quad \text{for all } a>0. 
\]
Since we have assumed that the process $\overline{X}$ is Feller, by the weak convergence of $\overline{\mathbb{P}}_{x_k}$ to $\overline{\mathbb{P}}_{\infty}$ we have that
\[
\overline{\mathbb{P}}_{\infty}\Big(\inf_{t\in [0,M]}\xi_t \ge a\Big)
\ge \limsup_{k\to \infty} \overline{\mathbb{P}}_{x_k}\Big(\inf_{t\in [0,M]}\xi_t \ge a, \zeta >M\Big) = 1 \quad \text{for all } a>0. 
\]
Since the choice of $a$ was arbitrary, we have $\overline{\mathbb{P}}_{\infty}(\xi_t= \infty, t\in [0,M])=1$. But $M$ is also arbitrary, contradicting the fact that $\infty$ is not a trap under $\overline{\mathbb{P}}_{\infty}$.

(2). We next prove \ref{H2} in a few steps. 

By the invariance property of the process $\overline{X}$, we find by \cref{prop:Phiz-n-bis} that the excursion measure $\mathbf{n}$ has the invariance property $\Phi_z \mathbf{n} = \ee^{-\theta z} \mathbf{n}$. Therefore, the mean occupation measure has the form $m(\ddd y) = C \frac{\ee^{\theta y}}{R(y)}\dd y$, due to \cref{prop:occup-meas}. 
Recall that the mean occupation measure is purely excessive for $X$. 

Because of the time-change, the measure $\ee^{\theta y}\dd y$ is excessive for the L\'evy process $\xi$ and hence $\mathrm{E}[\ee^{-\theta \xi_1}]\le 1$.
Therefore, we are in the exact situation of \cref{sec:exc-K}. 
With arguments developed there, we obtain the Kuznetsov measure of $m^{\xi}(\ddd y) = \ee^{\theta y} \dd y$ and $\xi$ as before, determined by \eqref{eq:Q-dual}. 
We already obtained \eqref{eq:H-rec-suff} in the first part of this proof. This permits us to again  preform the Kaspi time-change and thereby obtain an excursion measure of $X$ from $\mathcal{Q}$, now denoted by $\mathbf{n}'$, such that 
\begin{equation*}
	\mathbf{n}'(\omega_{t_1}\in A_1,\cdots,\omega_{t_k}\in A_k, t_k <\zeta)  =  \mathcal{Q}\big(Z_{\eta_{t_1}}\in A_1,\cdots, Z_{\eta_{t_k}}\in A_k, \eta_{t_k}\in [0,1) \big),
\end{equation*}
where $\eta_s = \inf \{u\in \bR\colon \int_{-\infty}^u \frac{1}{R(Z_r)} \dd r >s\}$. 

Let us identify $\mathbf{n}'$ with $\mathbf{n}$: by \cref{prop:exc-K-m}, the mean occupation measure of $\mathbf{n}'$ also has the form $m(\ddd y) = C' \frac{\ee^{\theta y}}{R(y)}\dd y$. By \cite[Theorem (5.25)]{Get90}, the two excursion measures have the same entrance law (up to a multiplicative constant). Therefore $\mathbf{n}'$ and $\mathbf{n}$ are equal (up to a multiplicative constant).

Finally, we distinguish the two cases with $\mathrm{E}[\ee^{\theta \xi_1}]< 1$ or  $\mathrm{E}[\ee^{\theta \xi_1}]= 1$.
If $\mathbf{n}$ enters from $+\infty$ continuously, which means that
$ \lim_{s\downarrow 0} w_s =+\infty$, $\mathbf{n}$-a.e., then we deduce from \cite[Corollary~(2.11)]{Kas88} that the property is preserved by the time-change, in the sense that $\lim_{s\downarrow -\infty} Z_s =+\infty$, $\mathcal{Q}$-a.e.. By \eqref{eq:Q-dual}, this corresponds to the fact that the Markov process with the semigroup $\overline{S}_t$ does not jump to $+\infty$. 
Therefore, $\overline{S}_t 1 = \mathrm{E}_0[ \ee^{-\theta \xi_t}] =1$. This is Cramér's condition.
		
On the other hand, if $\mathbf{n}$ enters from $+\infty$ by a jump, then the Markov process with semigroup $\overline{S}_t$ jumps to $+\infty$. 
Therefore $\overline{S}_t 1 = \mathrm{E}_0[ \ee^{-\theta \xi_t}] <1$.
\end{proof}

We end this section with a discussion of condition \eqref{eq:superfinite}.
\begin{remark}
We may rewrite \eqref{eq:superfinite} in terms of $\xi$, as
\[
	\inf_{x>b} \mathrm{P}_x\bigg(\int_0^{\infty} \frac{1}{R(\xi_s)} \dd s <M\bigg)>0. 
\]
This is the condition that $(b,\infty)$ is a \emph{super-finite set} in the sense of Baguley, D\"oring, and Kyprianou \cite{BDK}
We have seen from the proof above \eqref{eq:superfinite} that this is a necessary condition for the recurrent extension $\overline{X}$ to be Feller. 

Recall that the explosion condition states that $\int_0^{\infty} \frac{1}{R(\xi_s)} \dd s<\infty$, $\mathrm{P}_x\text{-a.s.}$.
A natural question is whether \eqref{eq:superfinite} always holds in the explosive case.
If $R$ is eventually decreasing, then we easily deduce by monotonicity that \eqref{eq:superfinite} follows from the explosion condition. But the answer is in general negative -- see the counterexample below that we adapt from \cite{KolSav19}. This suggests the possibility of constructing a non-Feller recurrent extension.  
\end{remark}

\begin{example}[An explosive example that does not satisfy \eqref{eq:superfinite}]\label{eg:KS}
	Consider $\xi$ as in \cite[Page 1459-1460]{KolSav19}: an increasing Lévy process with no drift and jumps of infinite activity and
size at most 1. Then there exists a continuous and non-negative $f$ that satisfies $\int_{\cdot}^{\infty} f(y) \dd y=\infty$ and such that $\int_0^{\infty} f(\xi_s) \dd s<\infty$ a.s.
	
Then we claim that for any  $M>0$ large and $\delta>0$ small, there exists a sequence $x_k\to \infty$ (depending on $M,\delta$), such that 
$\mathrm{P}_{x_k} (\int_0^{\infty} f(\xi_s) \dd s >M) >1-\delta$, $\forall k$. 
Indeed, for every $l>0$, if $\sup_{x>l} \mathrm{P}_{x} (\int_0^{\infty} f(\xi_s) \dd s >M) <1-\delta$, then by \cite[Lemma~4.5]{KolSav19} or \cite[Lemma~2.5]{BDK}, we have 
\[
	\mathrm{E}_x \left[ \int f\mathbf{1}_{[l,\infty)}(\xi_s) \dd s \right] <\infty. 
\]
As $\xi$ has finite mean, this is equivalent to $\int_{l}^{\infty} f(y) \dd y<\infty$, which contradicts to the property of $f$. 
Therefore, for each $l$, we can find $x>l$ s.t.\ $\mathrm{P}_{x} (\int_0^{\infty} f(\xi_s) \dd s >M) >1-\delta$. 
\end{example}
\subsection{A pathwise construction and a functional limit theorem}\label{sec:exc}

We shall now prove \cref{thm:convergence to n}, which entails
focusing on the case when Cramér's condition $\mathrm{E}_0[\ee^{-\theta \xi_1}] =1$ and the integrability condition \eqref{eq:H-rec-suff} hold. 
Under these conditions we have constructed an excursion measure $\mathbf{n}$ in \cref{sec:exc-K}. 

In proving \cref{prop:n-cv} we will provide a different construction of the excursion measure $\mathbf{n}$, relying crucially on the analysis by Barczy and Bertoin \cite{BarBer11}. Roughly speaking, we will give a path-wise construction for an excursion conditionally on passaging below level $y\in \bR$.  We stress that to use their approach, we require an additional integrability condition given by \eqref{eq:H-exp}, that we recall here: 
\begin{equation*}
	\mathrm{E}\big[|\xi_1| \ee^{-\theta \xi_1}\big] <\infty.
\end{equation*}
We will identify the excursion measure we construct here with that in \cref{sec:exc-K} and give further descriptions of the excursion measure, \`a la Pitman and Yor \cite{PitYor82}. 

Cramér's condition allows us to introduce the tilted measure 
\begin{equation*}\label{eq:tildeP}
	\widetilde{\mathrm{P}}_x\,\big|_{\mathcal{F}_t}=  \ee^{\theta x} \ee^{-\theta \xi_t}\mathrm{P}_x \,\big|_{\mathcal{F}_t},\qquad  t\ge \bR. 
\end{equation*}
This is the Esscher transform that we used above in \eqref{eq:barP}. Then  $\widetilde{\xi}$ under $\widetilde{\mathrm{P}}_x$ is a L\'evy process that drifts to $-\infty$,  with characteristic exponent $\widetilde{\Psi}(q) = \Psi(q+i\theta)$. 
Under assumption \eqref{eq:H-exp}, $\widetilde{\xi}$ has finite and negative first moment. In particular, it satisfies \eqref{eq:H1}. 
Since $\widetilde{\xi}$ tends to $-\infty$, it is well-known that its renewal function $\widetilde{\nu}_+$ is bounded. Then \eqref{eq:H0} is equivalent to
\[
\int_{\cdot}^{\infty} \frac{1}{R(y)}\dd y<\infty. 
\] 
Therefore $\widetilde{\xi}$ satisfies both assumptions of \cref{thm:1}; as we have done in \cref{sec:P-infty}, one can use methods developed by Bertoin and Savov \cite{BerSav11} to obtain  $\widetilde{\mathcal{P}}$ associated with $\widetilde{\xi}$, 
with $\widetilde{\rho}$ defined as in \cref{lem:overshoot} for $\widetilde{\xi}$.
Write $\widetilde{Z}$ for the canonical process under $\widetilde{\mathcal{P}}$.  

Barczy and Bertoin \cite{BarBer11} further introduce $\mathcal{P}$ by a change of measure  
\begin{equation}\label{eq:ChangeMeasure}
	\mathcal{P}\,\big|_{\mathcal{G}_t}\coloneqq  c(\theta) \exp(\theta \widetilde{Z}_t) 	\widetilde{\mathcal{P}}\,\big|_{\mathcal{G}_t}, \qquad t\in \bR, 
\end{equation}
where $c(\theta)$ is a normalizing constant defined by 
\begin{equation}\label{eq:c-theta}
	\frac{1}{c(\theta)}= \int_{\bR_+^2} \ee^{\theta y_2} \widetilde{\rho} (\ddd y_1, \dd y_2)=\widetilde{\mathcal{E}}[\exp(\theta \widetilde{Z}_0) ] <\infty.
\end{equation}
Let us collect a few properties of $\mathcal{P}$ from \cite{BarBer11}. 
Denote by $(Z_t, t\in \mathbb{R})$ the canonical process of  $\mathcal{P}$. We have 
$
	\mathcal{P} (\inf_{t\in \bR} Z_t <y) = \ee^{\theta y} \text{ for }y\le 0, 
$
and 
$$
	\rho (\ddd y_1, \dd y_2)\coloneqq  \mathcal{P}(Z_{0-}\in \dd y_1, -Z_0\in \dd y_2)=  c(\theta) \ee^{\theta y_2} \widetilde{\rho} (\ddd y_1, \dd y_2).
$$
Moreover, under the conditional law $\mathcal{P}(\cdot \mid Z_{0-}= y_1,- Z_0= y_2)$, 
the two processes $(-Z_{-t-}, t\ge 0)$ and $(Z_t, t\ge 0)$ are independent with respective laws $\widetilde{\mathrm{P}}^{\downarrow}_{y_1}$ as in \eqref{eq:P-downarrow} (in other words, $(Z_{-t-}, t\ge 0)$ is the dual process $-\widetilde{\xi}$ conditioned to stay positive) and $\mathrm{P}_{y_2}$.  
For $y\in \mathbb{R}$, let $\mathcal{P}_y$ denote the law of the process $y+Z$ under $\mathcal{P}$. 
The law $\mathcal{P}$ is spatially quasi-stationary in the following sense \cite[Equation~(9)]{BarBer11}: for all $y,z\in \bR$ with $y\le z$, 
\begin{equation}\label{eq:quasi-stationary}
	\text{under the conditional law } \mathcal{P}_z(\ccdot \mid \inf_{t\in \bR} Z_t <y),\,  (Z_{T_y+ s},s\in \bR) \sim \mathcal{P}_y. 
\end{equation}

Let us introduce a time-change to the process $Z$ of law $\mathcal{P}$.
Let $\psi$ be the map on $\W$ (the path space of excursions introduced in \cref{sec:exc-K}) defined by $\psi:w\mapsto w\circ\eta(w)$, where
\begin{align}\label{eq:psi-bis}
	\eta(w)_t = \inf\bigg\{s\ge0 : \int_{-\infty}^s f(w_u) \dd u>t\bigg\}, \quad t\ge0,
\end{align}
for some continuous $f:\R\to(0,\infty)$; in what follows we take $f=1/R$.
Barcy and Bertoin \cite[p2025]{BarBer11} mention that $\lim_{t\to\pm}Z_t=\infty$ \as, which implies that the time-changed process $\psi(Z)$ is a.s.~\cadlag.
As in \cref{sec:P-infty}, let $\widetilde{\mathbb{P}}_{\infty}$ be the push-forward probability of $\widetilde{\mathcal{P}}$ via $\psi$ and $\widetilde{\mathbb{E}}_{\infty}$ the expectation under $\widetilde{\mathbb{P}}_{\infty}$. 

Similarly, for any $y\in \bR$, let $\mathbb{P}^{(y)}_{\infty}$ be the push-forward probability of $\mathcal{P}_y$ via $\psi$ and $\mathbb{E}^{(y)}_{\infty}$ the expectation under $\mathbb{P}^{(y)}_{\infty}$. 
As we have seen in \eqref{eq:T0}, to make sure that the time-changed process of law $\mathbb{P}^{(y)}_{\infty}$ is non-trivial, a sufficient condition is for any $t\in \mathbb{R}$
\begin{equation}\label{eq:integral}
   \mathcal{E}_y \bigg[ \int_{-\infty}^t  \frac{1}{R(Z_r) } \dd r \bigg] <\infty.
\end{equation}
Applying the change of measure, we have
\[
\mathcal{E}_y \bigg[ \int_{-\infty}^t  \frac{1}{R(Z_r) } \dd r \bigg] 
=c(\theta)\widetilde{\mathcal{E}}\bigg[ \int_{-\infty}^t  \frac{\ee^{\theta (y+\widetilde{Z}_r)}}{R(y + \widetilde{Z}_r) } \dd r \bigg]. 
\]
By \cref{lem:H0}, recalling that $\widetilde{\nu}_+$ is bounded, this is finite under assumption \eqref{eq:H-rec-suff}.
So we indeed have \eqref{eq:integral}. 

\begin{lemma}\label{lem:ChangeofMeasure-bP}
	For $y\in \bR$, we have 
	\begin{equation}\label{eq:ChangeofMeasure-bP}
		\mathbb{P}^{(y)}_{\infty}\,\big|_{\mathcal{F}_t}
		= c(\theta) \exp(\theta (X_t-y)) 	\mathbf{1}_{\{T_y<\infty, t<\zeta\}}\,\widetilde{\mathbb{P}}_{\infty}\,\big|_{\mathcal{F}_t}, \qquad t\in \bR_+. 
	\end{equation} 
\end{lemma}
\begin{proof}
	For any $w\in \W$, let $(\eta^{(y)}_t(w), t\ge 0)$ be the time-change associated with $y+w$. 
	Note that (see e.g.\ \cite[Proposition~7.9]{Kallenberg}) for a $\mathcal{G}$-adapted process $Z$, $(\eta^{(y)}_t,t\ge 0)$ is a family of stopping times, generating a right-continuous filtration $\mathcal{F}_t = \mathcal{G}_{\eta^{(y)}_t}$. 
	By the change of measure \eqref{eq:ChangeMeasure}, we have, for $f$ such that $f(\infty) =\infty$,
	\begin{equation}\label{eq:shift}
		\mathbb{E}^{(y)}_{\infty}\Big[f(X_t)\mathbf{1}_{\{t<\zeta\}}\Big]
		= \mathcal{E} \Big[f(y + Z_{\eta^{(y)}_t})\Big] 
		= c(\theta )\ee^{-\theta y} \widetilde{\mathcal{E}} \Big[f(y + Z_{\eta^{(y)}_t}) \exp(\theta (y+Z_{\eta^{(y)}_t})\Big]. 
	\end{equation}
	For any $w\in D(\bR)$, let $h_y(w)$ be the path shifted by the hitting time $T_y$, i.e. 
	\begin{equation}\label{eq:h_y}
		h_y(w)_t = w_{T_y+t}\mathbf{1}_{\{T_y<\infty\}}+ \infty\mathbf{1}_{\{T_y=\infty\}},\qquad  t\in \bR. 		
	\end{equation}
	
	Since $y+Z$ and $Z'\coloneqq  h_y(Z)$ have the same distribution under $\widetilde{\mathcal{P}}$, the last expectation in \eqref{eq:shift} is equal to 
	\begin{align*}
		\widetilde{\mathcal{E}} \Big[f(Z'_{\eta(Z')_t}) \exp(\theta Z'_{\eta(Z')_t})\Big]
		&=\widetilde{\mathcal{E}} \Big[f(Z_{\eta_t}) \exp(\theta Z_{\eta_t})\mathbf{1}_{\{T_y<\infty\}}\Big]\\
		&= \widetilde{\mathbb{E}}_{\infty}\Big[f(X_t)\mathbf{1}_{\{T_y<\infty, t<\zeta\}}\ee^{\theta X_t}\Big], 
	\end{align*} 
	where the first equality follows from the observation that, if $T_y(w) <\infty$, then $w_{\eta_t(w)} = h_y(w)_{\eta_t(h_y(w))}$, for all $t\ge 0$. 
	This proves the relation between the marginal distributions. Similar arguments work for the finite-dimensional distributions, details are omitted.  
\end{proof}

We now introduce an excursion measure, which is still denoted by $\mathbf{n}$, as we shall see in \cref{prop:n-unique} that it is indeed the same (up to a multiplicative constant) as the excursion measure constructed in \cref{sec:exc-K}. 
\begin{proposition}[First description of $\mathbf{n}$]\label{prop:first}
	If $\mathcal{E}\big[\int_{-\infty}^{\infty} \frac{\dd r}{R(Z_r)}\big]<\infty$ then there exists a unique excursion measure $\mathbf{n}$, which is a $\sigma$-finite measure on the subspace of excursions away from $+\infty$ in $D([0,\infty), \oR)$, such that for $y\in \mathbb{R}$, 
	\begin{enumerate}
		\item $\mathbf{n}( \exc_0\ne +\infty)= 0$,
		\item $\mathbf{n}( T_y <\infty)=  \ee^{\theta y}$, 
		\item the probability measure $\mathbf{n}(\ccdot\mid T_y <\infty)$ is equal to $\mathbb{P}^{(y)}_{\infty}$.  
	\end{enumerate}
\end{proposition}

\begin{proof}
	Due to \cref{lem:ChangeofMeasure-bP}, $(\ee^{\theta y}\mathbb{P}^{(y)}_{\infty}, y\in \bR)$ form a consistent family of finite measures. 
	Then it follows that there exists a unique $\sigma$-finite measure $\mathbf{n}$ such that $\mathbf{n} \,|_{\{T_y<\infty\}} = \ee^{\theta y}\mathbb{P}^{(y)}_{\infty}$, $y\in \bR$. 
	This description implies the three properties. 
\end{proof}

In what follows, we will use the notation $\mathbf{n}(f)\coloneqq  \int f \dd \mathbf{n}$. 
With this pathwise construction, we can give an alternative proof of the scaling property \eqref{eq:n-scaling}. 
\begin{proposition}\label{prop:n-scaling}
	For $z\in \bR$, let $\Phi_z^*(\mathbf{n})$ be the push-forward measure of $\mathbf{n}$ via $\Phi_z$ defined in \eqref{eq:Phi-z}. 
	Then 
	\begin{equation*}
		\Phi_z^*(\mathbf{n}) = \ee^{-\theta z} \mathbf{n}, \qquad \text{for all } z\in \bR. 
	\end{equation*} 
\end{proposition}

\begin{proof}
	Fix any $z,y\in \bR$. 
	By the definition of $\Phi_z$, it is easy to see that $\Phi_z(\omega)_0= z+\omega_0$ and that $\inf_{t\ge 0}\Phi_z(\omega)_t = z+ \inf_{t\ge 0}\omega_t$. 
	Therefore, using the description in \cref{prop:first}, we have 
	\begin{enumerate}
		\item 
		$\Phi_z^*(\mathbf{n}) ( \omega_0\ne \infty)=\mathbf{n}( \omega_0\ne \infty)= 0$, 
		\item  $\Phi_z^*(\mathbf{n})  ( T_y <\infty)= \mathbf{n}( T_{y+z} <\infty)= \ee^{-\theta (y+z)}$,
		\item 
		if $X\sim \mathbf{n}( \ccdot \mid T_y \!<\!\infty)$ then $\Phi_z (X)\sim \Phi_z^*(\mathbf{n}) ( \ccdot \mid T_{y+z} \!<\!\infty)$. 
	\end{enumerate}
	It remains to prove that $\Phi_z (X)\sim \mathbf{n}( \ccdot \mid T_{y+z} \!<\!\infty)$. 
	By the change of measure \eqref{eq:ChangeofMeasure-bP} and the invariance property of $\widetilde{\mathbb{P}}_{\infty}$ given in \cref{prop:P-infty-inv}, for any $t\ge 0$ and bounded continuous functional $f$, 
	\begin{align*}
		\mathbf{n}[ f(\Phi_z (X)_s, s\le t) \mid T_y \!<\!\infty]
		&= \mathbb{E}^{(y)}_{\infty} \left[ f(\Phi_z (X)_s, s\le t) \right]\\
		&= c(\theta )\ee^{-\theta y} \widetilde{\mathbb{E}}_{\infty} \left[\exp(\theta (\Phi_z (X)_{t}-z)) f(\Phi_z (X)_s, s\le t)\right] \\
		&= c(\theta )\ee^{-\theta (y+z)} \widetilde{\mathbb{E}}_{\infty} \left[\exp(\theta  X_{t}) f(X_s, s\le t)\right] \\
		& = \mathbb{E}^{(y+z)}_{\infty} \left[ f(X_s, s\le t) \right]\\
		&= \mathbf{n}[ f(X_s, s\le t) \mid T_{y+z} \!<\!\infty].  \hspace{12em} \qedhere
	\end{align*}
\end{proof}

\begin{proposition}[Uniqueness]\label{prop:n-unique}
	The excursion measure $\mathbf{n}$ constructed in \cref{prop:first} is the same as the one in \cref{sec:exc-K}, up to a multiplicative constant. 
\end{proposition}
\begin{proof}
	Note that each  excursion measure is uniquely determined by its entrance law, and thus uniquely determined by the  mean occupation measure, due to \cite[Theorem~(5.25)]{Get90}. 
	Therefore, it suffices to show that the mean occupation measures are the same. 
	By \cref{prop:n-scaling,prop:occup-meas,prop:exc-K-m}, this is indeed the case. 
\end{proof}

Let $\tilde{\xi}$ be a L\'evy process of distribution $\widetilde{\mathrm{P}}_x$. Let $\widetilde{\mathbb{P}}_x$ be the law of the time-changed process $\widetilde{X}$ defined analagously to \eqref{eq:time-change}.
It is straightforward to establish the following change of measure. 
\begin{lemma}[Change of measure]
	\begin{equation*}\label{eq:tildebP}
		\widetilde{\mathbb{P}}_x \,\big|_{\mathcal{F}_t}=  \ee^{\theta x} \ee^{-\theta X_t} \,\mathbb{P}_x\,\big|_{\mathcal{F}_t},\qquad  \text{for all } t\ge 0. 
	\end{equation*}
\end{lemma}
Recall that $\tilde{\xi}$ satisfies \eqref{eq:H1} and that the properties developed in \cref{sec:P-infty} hold for the process $\widetilde{X}$. In particular, by \cref{prop:convergence}, $\widetilde{\mathbb{P}}_x\to \widetilde{\mathbb{P}}_{\infty}$ weakly as $x\to \infty$. 
\begin{proposition}
	Under $\mathbf{n}(\ccdot\mid T_y <\infty)$, $(\mathbf{e}_{t}, 0\le t\le T_y)$ is a process of law $\widetilde{\mathbb{P}}_{\infty}$ stopped at first crossing of $y$. Conditionally on $(\mathbf{e}_{t}, 0\le t\le T_y)$, the process $(\mathbf{e}_{T_y+s}, s\ge 0)$ has distribution $\mathbb{P}_{\mathbf{e}_{T_y}}$. 
\end{proposition}
\begin{proof}
 Since the probability measure $\mathbf{n}(\ccdot\mid T_y <\infty)$ is equal to $\mathbb{P}^{(y)}_{\infty}$, the proof boils down to give a decomposition of a process of law 	$\mathbb{P}^{(y)}_{\infty}$. 
 Recall that $\mathbb{P}^{(y)}_{\infty}$ is obtained from $\mathcal{P}_y$ via a time-change.  
 For a process $Z$ of law $\mathcal{P}_y$, using the description above equation \eqref{eq:quasi-stationary},  we deduce that the backward process $(Z_{-t-}-y, t\ge 0)$ starts from $Z_{0-}-y = Y_1$, with $Y_1$ of the first marginal distribution of $\widetilde{\rho}$, and has the law of the dual process $-\widetilde{\xi}$ conditioned to stay positive. After the time-change, this coincides with the construction of $\widetilde{\mathbb{P}}_{\infty}$ given by \eqref{time-change eq}. Moreover, the forward process $(Z_t, t\ge 0)$ has distribution $P_{y_2}$, conditionally on $Z_0 =y_2$. Then via the time-change we have the law of the time-changed process.  
\end{proof}

\begin{proposition}[Second description of $\mathbf{n}$]\label{prop:second}
	The measure $\mathbf{n}$ defined in \cref{prop:first} is the unique $\sigma$-finite measure that satisfies the following properties:  
	\begin{enumerate}
		\item $\mathbf{n}( \mathbf{e}_0\ne \infty)= 0$,
		\item Under $\mathbf{n}$, the excursion $(\mathbf{e}_t, t\ge 0)$ is Markovian with the transition probabilities of $\bP$ and entrance law $n_s(\ddd y) \coloneqq  \mathbf{n}(\mathbf{e}_s\in \dd y, s<\zeta)$, $s>0$, defined by 
		\begin{equation}\label{eq:n-entrance}
			n_s(\ddd y) = c(\theta) \ee^{\theta y}\widetilde{\mathbb{P}}_{\infty}(X_s\in \dd y),\qquad y\in \bR. 
		\end{equation} 
	\end{enumerate}
\end{proposition}

As an immediate consequence of \cref{prop:second} we notice that for $s>0$, 
\begin{equation*}\label{eq:n-zeta}
	\mathbf{n}(\zeta>s) = c(\theta) \widetilde{\mathbb{E}}_{\infty}\left[\ee^{\theta X_{s}} , s<\zeta \right]. 
\end{equation*}
\begin{proof}
	It is clear that at most one measure satisfies these properties. 
	Let us check that $\mathbf{n}$ defined by the first description in \cref{prop:first} has the entrance law given by \eqref{eq:n-entrance}. Let  $s>0$ and $f$ a bounded continuous function with compact support in $\bR$. For $y\in \bR$ above the support of $f$, 
	$
	\mathbf{n}[\mathbf{e}_s) \ind{s<\zeta}]  = \mathbf{n}[f(\mathbf{e}_s), \ind{s<\zeta, T_y<s}] . 
	$
	Therefore, applying the change of measure in \eqref{eq:ChangeofMeasure-bP} and \cref{prop:convergence}, we have 
	\begin{align*}
		\mathbf{n}[f(\mathbf{e}_s)\ind{s<\zeta, T_y<s}]
		& =\ee^{\theta y}\mathbb{E}^{(y)}_{\infty}	[f(X_s), s<\zeta, T_y<s]\\
		& = c(\theta) \widetilde{\mathbb{E}}_{\infty}	\Big[\ee^{\theta X_s}f(X_s), s<\zeta, T_y<s\Big]\\
		&= c(\theta)  \int_0^s \widetilde{\mathbb{P}}_{\infty}(T_y\!\in\! \dd u) \int  \widetilde{\rho}(\ddd y_2)\widetilde{\mathbb{E}}_{y-y_2}\Big[\ee^{\theta X_{s-u}} f(X_{s-u}), s<\zeta \Big]\\
		&\to   c(\theta)\widetilde{\mathbb{E}}_{\infty}\Big[\ee^{\theta X_{s}} f(X_{s}), s<\zeta \Big] \quad \text{as } y\to \infty, 
	\end{align*}
	where the convergence  is due to the fact that $(x,t)\mapsto \widetilde{\mathbb{E}}_x[g(X_t)]$ is jointly continuous for any bounded continuous $g$, a consequence of the Feller property of $\widetilde{\mathbb{P}}$ given in \cref{prop:convergence}, see e.g.\ \cite[Theorem~2.2.1]{ChuWal-Book}. 
	This implies \eqref{eq:n-entrance}. 
\end{proof}

We finally prove the convergence result, \cref{prop:n-cv}, extending the time-change mapping $\psi$ given in \cref{sec:conti} to processes indexed by $\R$. 
Let $\DRR$ be the class of all c\`adl\`ag functions $:\R\to\oR$.
Skorokhod's $J_1$ topology on $\DRR$ is induced by the metric
$$
	d(v,w) = \int_{-\infty}^0  \int_0^\infty \ee^{s-t}d_{s,t}(v,w) \dd t \dd s,
$$
where $d_{s,t}(v,w)$ is the analogue of $d_t$ but taken over the interval $[s,t]$ rather than $[0,t]$ (see \cite[Theorem 2.5]{Whitt80}).
Let $\psi$ be the path transform on $\W$ defined in \eqref{eq:psi-bis}.
As in \cref{sec:conti}, we use the notation $D\coloneqq D([0, \infty),\oR)$.

\begin{lemma}\label{lem:psi-conti-R}
Let $L\subset \DRR$ be the set of paths $w$ that have limits $\lim_{t\to\pm\infty} w_t = \infty$ and for which $\int_{-\infty}^s f(w_u) \dd u<\infty$ for all $s<\infty$.
Then the map
$
	\psi|_L : (L, d)\to (D,d)
$
is continuous.
\end{lemma}
\begin{proof}
For any path $v\in \DRR$ and any $t\in\R$ let $\tv$ be the path in $D$ defined by
$
	\tv_s = v_{t+s},\, s\ge0.
$
A sequence of paths $(w^n)$ converges to $w$ in $(\DRR, d)$ if and only if for every $t\in\R$, $(\tw^n)$ converges to $\tw$ in $(D, d)$. 

Now consider any sequence $w^n\to w$ in $(L, d)$. For every $t\in\R$, $\tw^n \to \tw$ in $(L,d)$, and by \cref{cty prop 2} we deduce that $\psi(\tw^n)\to \psi(\tw)$ in $(D, d)$. It remains to use this to show that 
${}^t\psi(w^n)\to {}^t\psi(w)$ for arbitrarily small $t\in(0,\infty)$, from which it will follow that 
$\psi(w^n)\to \psi(w)$.

For any path $v\in \DRR$, we have that
$
	\psi(\tv)(\,\cdot\,) = v(t + \eta(\tv)(\,\cdot\,)),
$
where
\begin{align*}
	\eta(\tv)(a)
	= \inf_{s>0} \Big\{ \int_0^s f(\tv_u)\dd u > a\Big\}
	&= \inf_{s>0} \Big\{ \int_t^{t+s} f(w_u)\dd u > a\Big\}\\
	&= -t + \inf_{s>0} \Big\{ \int_t^s f(v_u)\dd u > a\Big\}\\
	&= -t + \inf_{s>-\infty} \{ I(v)_u > a + I(v)_t \}\\
	&= -t + \eta(v)(a + I(v)_t).
\end{align*}
Thus
$
	\psi(v)_s = \psi(\tv)_{s - I(v)_t}
$
for $s>I(v)_t$, that is,
$
	^{I(v)_t}\psi(v) = \psi(\tv).
$
Since $w^n\to w$ we have that $I(w^n)\to I(w)$ pointwise, and it follows that as $n\to\infty$,
$$
	d({}^{I(w)_t}\psi(w^n), {}^{I(w)_t}\psi(w))
	\le d({}^{I(w)_t}\psi(w^n), {}^{I(w^n)_t}\psi(w^n)) + d( \psi(\tw^n), \psi(\tw)) \to 0.
$$
Since $I(w)$ is increasing and continuous, we can make $I(w)_t$ arbitrarily small by sending $t\to-\infty$, and thus we have that $\psi(w^n)\to\psi(w)$ under $d$.
\end{proof}

\begin{proof}[Proof of \cref{prop:n-cv}]
It follows from \cite[Lemma~4.1]{KallenbergRM} that it suffices to show for
 every $y\in \bR$ the following statements: 
	\begin{enumerate}
		\item $\mathbf{n} (\inf_{s\ge 0} \mathbf{e}_s = y) = 0$;
		\item $c(\theta) \ee^{\theta x} \mathbb{P}_x(T_y <\infty)\Rightarrow \mathbf{n}(T_y <\infty)$ as $x\to \infty$;
		\item $\mathbb{P}_x(\ccdot \mid T_y <\infty)\Rightarrow \mathbf{n}(\ccdot \mid T_y <\infty)$ weakly, as $x\to \infty$.
	\end{enumerate}
	(i) is clear from the description of $\mathbf{n}$. 
	(ii) follows from classical Cram\'er's estimate, see \cite[Remark~2]{BerDon94}, which tells us that $ c(\theta) \ee^{\theta x} \mathbb{P}_x(T_y <\infty) = \ee^{\theta x} \mathrm{P}_x(T_y <\infty)\to  \ee^{\theta y}$, as $x\to \infty$.  

 To prove (iii), we regard a L\'evy process $\xi$ as a process indexed by $\mathbb{R}$ with $\xi_{t}= \xi_0$ for all $t<0$. Recall from \eqref{eq:h_y} that  $h_y(\xi)$ denotes the path transform by shifting $T_y$ to zero and consider $h_y(\xi)$. 
By \cite[Theorem 1]{BarBer11}, we have that, under the conditional law $\mathrm{P}_x(\ccdot\mid T_y <\infty)$, as $x\to \infty$, the process $h_y(\xi)$ converges weakly to $\mathcal{P}_y$. 
By \cref{lem:psi-conti-R}, the weak convergence still holds after the time-change, hence we derive that $\mathbb{P}_x(\ccdot \mid T_y <\infty) \Rightarrow \mathbb{P}^{(y)}_{\infty} = \mathbf{n}(\ccdot \mid T_y <\infty)$. 

When \eqref{eq:H-exp} does not hold, we have by \cite{BerDon94} that $ \ee^{\theta x} \mathbb{P}_x(T_y <\infty)\to 0$, as $x\to \infty$. Therefore the limit is a zero measure on every set away from a neighborhood of $+\infty$. 
\end{proof}

\begin{funding}
%
QS was supported in part by the National Key R\&D Program of China (No. 2022YFA1006500) and National Natural Science Foundation of China (No. 12288201 and 12301169). 
\end{funding}

\bibliographystyle{imsart-number}
\bibliography{bibliography}

\begin{thebibliography}{54}

\bibitem{BDK}
\begin{barticle}[author]
\bauthor{\bsnm{Baguley},~\bfnm{S.}\binits{S.}},
  \bauthor{\bsnm{Döring},~\bfnm{L.}\binits{L.}} \AND
  \bauthor{\bsnm{Kyprianou},~\bfnm{A.~E.}\binits{A.~E.}}
(\byear{2024}).
\btitle{General path integrals and stable {SDE}s}.
\bjournal{Journal of the European Mathematical Society}
\bvolume{26}
\bpages{3243–3286}.
\end{barticle}
\endbibitem

\bibitem{BDS-2}
\begin{barticle}[author]
\bauthor{\bsnm{Baguley},~\bfnm{S.}\binits{S.}},
  \bauthor{\bsnm{Döring},~\bfnm{L.}\binits{L.}} \AND
  \bauthor{\bsnm{Shi},~\bfnm{Q.}\binits{Q.}}
(\byear{in preparation}).
\btitle{On the entrance at infinity of Feller processes with two-sided jumps}.
\end{barticle}
\endbibitem

\bibitem{BarBer11}
\begin{barticle}[author]
\bauthor{\bsnm{Barczy},~\bfnm{M\'{a}ty\'{a}s}\binits{M.}} \AND
  \bauthor{\bsnm{Bertoin},~\bfnm{Jean}\binits{J.}}
(\byear{2011}).
\btitle{Functional limit theorems for {L}\'{e}vy processes satisfying
  {C}ram\'{e}r's condition}.
\bjournal{Electron. J. Probab.}
\bvolume{16}
\bpages{no. 73, 2020--2038}.
\bdoi{10.1214/EJP.v16-930}
\bmrnumber{2851054}
\end{barticle}
\endbibitem

\bibitem{Ber-Book}
\begin{bbook}[author]
\bauthor{\bsnm{Bertoin},~\bfnm{Jean}\binits{J.}}
(\byear{1996}).
\btitle{L\'{e}vy processes}.
\bseries{Cambridge Tracts in Mathematics}
\bvolume{121}.
\bpublisher{Cambridge University Press, Cambridge}.
\bmrnumber{1406564}
\end{bbook}
\endbibitem

\bibitem{BerDon94}
\begin{barticle}[author]
\bauthor{\bsnm{Bertoin},~\bfnm{J.}\binits{J.}} \AND
  \bauthor{\bsnm{Doney},~\bfnm{R.~A.}\binits{R.~A.}}
(\byear{1994}).
\btitle{Cram\'{e}r's estimate for {L}\'{e}vy processes}.
\bjournal{Statist. Probab. Lett.}
\bvolume{21}
\bpages{363--365}.
\bdoi{10.1016/0167-7152(94)00032-8}
\bmrnumber{1325211}
\end{barticle}
\endbibitem

\bibitem{BerKor16}
\begin{barticle}[author]
\bauthor{\bsnm{Bertoin},~\bfnm{Jean}\binits{J.}} \AND
  \bauthor{\bsnm{Kortchemski},~\bfnm{Igor}\binits{I.}}
(\byear{2016}).
\btitle{Self-similar scaling limits of {M}arkov chains on the positive
  integers}.
\bjournal{Ann. Appl. Probab.}
\bvolume{26}
\bpages{2556--2595}.
\bdoi{10.1214/15-AAP1157}
\bmrnumber{3543905}
\end{barticle}
\endbibitem

\bibitem{BerSav11}
\begin{barticle}[author]
\bauthor{\bsnm{Bertoin},~\bfnm{Jean}\binits{J.}} \AND
  \bauthor{\bsnm{Savov},~\bfnm{Mladen}\binits{M.}}
(\byear{2011}).
\btitle{Some applications of duality for {L}\'{e}vy processes in a half-line}.
\bjournal{Bull. Lond. Math. Soc.}
\bvolume{43}
\bpages{97--110}.
\bdoi{10.1112/blms/bdq084}
\bmrnumber{2765554}
\end{barticle}
\endbibitem

\bibitem{BerYor}
\begin{barticle}[author]
\bauthor{\bsnm{Bertoin},~\bfnm{Jean}\binits{J.}} \AND
  \bauthor{\bsnm{Yor},~\bfnm{Marc}\binits{M.}}
(\byear{2002}).
\btitle{The entrance laws of self-similar {M}arkov processes and exponential
  functionals of {L}\'{e}vy processes}.
\bjournal{Potential Anal.}
\bvolume{17}
\bpages{389--400}.
\bdoi{10.1023/A:1016377720516}
\bmrnumber{1918243}
\end{barticle}
\endbibitem

\bibitem{RegularVariation}
\begin{bbook}[author]
\bauthor{\bsnm{Bingham},~\bfnm{N.~H.}\binits{N.~H.}},
  \bauthor{\bsnm{Goldie},~\bfnm{C.~M.}\binits{C.~M.}} \AND
  \bauthor{\bsnm{Teugels},~\bfnm{J.~L.}\binits{J.~L.}}
(\byear{1989}).
\btitle{Regular variation}.
\bseries{Encyclopedia of Mathematics and its Applications}
\bvolume{27}.
\bpublisher{Cambridge University Press, Cambridge}.
\bmrnumber{1015093}
\end{bbook}
\endbibitem

\bibitem{BCKW}
\begin{barticle}[author]
\bauthor{\bsnm{Biskup},~\bfnm{Marek}\binits{M.}},
  \bauthor{\bsnm{Chen},~\bfnm{Xin}\binits{X.}},
  \bauthor{\bsnm{Kumagai},~\bfnm{Takashi}\binits{T.}} \AND
  \bauthor{\bsnm{Wang},~\bfnm{Jian}\binits{J.}}
(\byear{2021}).
\btitle{Quenched invariance principle for a class of random conductance models
  with long-range jumps}.
\bjournal{Probability Theory and Related Fields}
\bvolume{180}
\bpages{847-889}.
\bdoi{10.1007/s00440-021-01059-z}
\end{barticle}
\endbibitem

\bibitem{BG}
\begin{bbook}[author]
\bauthor{\bsnm{Blumenthal},~\bfnm{R.~M.}\binits{R.~M.}} \AND
  \bauthor{\bsnm{Getoor},~\bfnm{R.~K.}\binits{R.~K.}}
(\byear{1968}).
\btitle{Markov Processes and Potential Theory}.
\bseries{Pure and Applied Mathematics}.
\bpublisher{Academic Press}.
\end{bbook}
\endbibitem

\bibitem{Blu83}
\begin{barticle}[author]
\bauthor{\bsnm{Blumenthal},~\bfnm{R.~M.}\binits{R.~M.}}
(\byear{1983}).
\btitle{On construction of {M}arkov processes}.
\bjournal{Z. Wahrsch. Verw. Gebiete}
\bvolume{63}
\bpages{433--444}.
\bdoi{10.1007/BF00533718}
\bmrnumber{705615}
\end{barticle}
\endbibitem

\bibitem{Bul-Book}
\begin{bbook}[author]
\bauthor{\bsnm{Buldygin},~\bfnm{Valeri\u{\i}~V.}\binits{V.~V.}},
  \bauthor{\bsnm{Indlekofer},~\bfnm{Karl-Heinz}\binits{K.-H.}},
  \bauthor{\bsnm{Klesov},~\bfnm{Oleg~I.}\binits{O.~I.}} \AND
  \bauthor{\bsnm{Steinebach},~\bfnm{Josef~G.}\binits{J.~G.}}
(\byear{2018}).
\btitle{Pseudo-regularly varying functions and generalized renewal processes}.
\bseries{Probability Theory and Stochastic Modelling}
\bvolume{91}.
\bpublisher{Springer, Cham}.
\bdoi{10.1007/978-3-319-99537-3}
\bmrnumber{3839312}
\end{bbook}
\endbibitem

\bibitem{CLU}
\begin{barticle}[author]
\bauthor{\bsnm{Caballero},~\bfnm{M.~A.}\binits{M.~A.}},
  \bauthor{\bsnm{Lambert},~\bfnm{A.}\binits{A.}} \AND \bauthor{\bsnm{{Uribe
  Bravo}},~\bfnm{G.}\binits{G.}}
(\byear{2009}).
\btitle{Proof(s) of the {L}amperti representation of continuous-state branching
  processes}.
\bjournal{Probability Surveys}
\bvolume{6}
\bpages{62-89}.
\end{barticle}
\endbibitem

\bibitem{CabCha06}
\begin{barticle}[author]
\bauthor{\bsnm{Caballero},~\bfnm{M.~E.}\binits{M.~E.}} \AND
  \bauthor{\bsnm{Chaumont},~\bfnm{L.}\binits{L.}}
(\byear{2006}).
\btitle{Weak convergence of positive self-similar {M}arkov processes and
  overshoots of {L}\'{e}vy processes}.
\bjournal{Ann. Probab.}
\bvolume{34}
\bpages{1012--1034}.
\bdoi{10.1214/009117905000000611}
\bmrnumber{2243877}
\end{barticle}
\endbibitem

\bibitem{ChaumontDoney}
\begin{barticle}[author]
\bauthor{\bsnm{Chaumont},~\bfnm{L.}\binits{L.}} \AND
  \bauthor{\bsnm{Doney},~\bfnm{R.~A.}\binits{R.~A.}}
(\byear{2005}).
\btitle{On {Lévy} processes conditioned to stay positive}.
\bjournal{Electron. J. Probab.}
\bvolume{10}
\bpages{948-961}.
\end{barticle}
\endbibitem

\bibitem{CKPV12}
\begin{barticle}[author]
\bauthor{\bsnm{Chaumont},~\bfnm{Lo\"{\i}c}\binits{L.}},
  \bauthor{\bsnm{Kyprianou},~\bfnm{Andreas}\binits{A.}},
  \bauthor{\bsnm{Pardo},~\bfnm{Juan~Carlos}\binits{J.~C.}} \AND
  \bauthor{\bsnm{Rivero},~\bfnm{V\'{\i}ctor}\binits{V.}}
(\byear{2012}).
\btitle{Fluctuation theory and exit systems for positive self-similar {M}arkov
  processes}.
\bjournal{Ann. Probab.}
\bvolume{40}
\bpages{245--279}.
\bdoi{10.1214/10-AOP612}
\bmrnumber{2917773}
\end{barticle}
\endbibitem

\bibitem{ChuWal-Book}
\begin{bbook}[author]
\bauthor{\bsnm{Chung},~\bfnm{Kai~Lai}\binits{K.~L.}} \AND
  \bauthor{\bsnm{Walsh},~\bfnm{John~B.}\binits{J.~B.}}
(\byear{2005}).
\btitle{Markov processes, {B}rownian motion, and time symmetry},
\bedition{second} ed.
\bseries{Grundlehren der mathematischen Wissenschaften [Fundamental Principles
  of Mathematical Sciences]}
\bvolume{249}.
\bpublisher{Springer, New York}.
\bdoi{10.1007/0-387-28696-9}
\bmrnumber{2152573}
\end{bbook}
\endbibitem

\bibitem{Cli94}
\begin{barticle}[author]
\bauthor{\bsnm{Cline},~\bfnm{Daren B.~H.}\binits{D.~B.~H.}}
(\byear{1994}).
\btitle{Intermediate regular and {$\Pi$} variation}.
\bjournal{Proc. London Math. Soc. (3)}
\bvolume{68}
\bpages{594--616}.
\bdoi{10.1112/plms/s3-68.3.594}
\bmrnumber{1262310}
\end{barticle}
\endbibitem

\bibitem{DDK17}
\begin{barticle}[author]
\bauthor{\bsnm{Dereich},~\bfnm{Steffen}\binits{S.}},
  \bauthor{\bsnm{D\"{o}ring},~\bfnm{Leif}\binits{L.}} \AND
  \bauthor{\bsnm{Kyprianou},~\bfnm{Andreas~E.}\binits{A.~E.}}
(\byear{2017}).
\btitle{Real self-similar processes started from the origin}.
\bjournal{Ann. Probab.}
\bvolume{45}
\bpages{1952--2003}.
\bdoi{10.1214/16-AOP1105}
\bmrnumber{3650419}
\end{barticle}
\endbibitem

\bibitem{DonMal02}
\begin{barticle}[author]
\bauthor{\bsnm{Doney},~\bfnm{R.~A.}\binits{R.~A.}} \AND
  \bauthor{\bsnm{Maller},~\bfnm{R.~A.}\binits{R.~A.}}
(\byear{2002}).
\btitle{Stability of the overshoot for {L}\'{e}vy processes}.
\bjournal{Ann. Probab.}
\bvolume{30}
\bpages{188--212}.
\bdoi{10.1214/aop/1020107765}
\bmrnumber{1894105}
\end{barticle}
\endbibitem

\bibitem{D12}
\begin{barticle}[author]
\bauthor{\bsnm{D{\"o}ring},~\bfnm{Leif}\binits{L.}}
(\byear{2012}).
\btitle{A Jump-Type SDE Approach to Real-Valued Self-Similar Markov Processes}.
\bjournal{Transactions of the American Mathematical Society}
\bvolume{367}.
\bdoi{10.1090/S0002-9947-2015-06270-9}
\end{barticle}
\endbibitem

\bibitem{DK16}
\begin{barticle}[author]
\bauthor{\bsnm{D\"{o}ring},~\bfnm{Leif}\binits{L.}} \AND
  \bauthor{\bsnm{Kyprianou},~\bfnm{Andreas~E.}\binits{A.~E.}}
(\byear{2016}).
\btitle{{Perpetual integrals for {L}\'{e}vy processes}}.
\bjournal{J. Theoret. Probab.}
\bvolume{29}
\bpages{1192--1198}.
\bdoi{10.1007/s10959-015-0607-y}
\bmrnumber{3540494}
\end{barticle}
\endbibitem

\bibitem{DK20}
\begin{barticle}[author]
\bauthor{\bsnm{D\"{o}ring},~\bfnm{Leif}\binits{L.}} \AND
  \bauthor{\bsnm{Kyprianou},~\bfnm{Andreas~E.}\binits{A.~E.}}
(\byear{2020}).
\btitle{Entrance and exit at infinity for stable jump diffusions}.
\bjournal{Ann. Probab.}
\bvolume{48}
\bpages{1220--1265}.
\bdoi{10.1214/19-AOP1389}
\bmrnumber{4112713}
\end{barticle}
\endbibitem

\bibitem{DT23}
\begin{barticle}[author]
\bauthor{\bsnm{D\"oring},~\bfnm{Leif}\binits{L.}} \AND
  \bauthor{\bsnm{Trottner},~\bfnm{Lukas}\binits{L.}}
(\byear{2023}).
\btitle{Stability of overshoots of {M}arkov additive processes}.
\bjournal{Ann. Appl. Probab.}
\bvolume{33}
\bpages{5413--5458}.
\bdoi{10.1214/23-aap1951}
\bmrnumber{4677737}
\end{barticle}
\endbibitem

\bibitem{BD12}
\begin{barticle}[author]
\bauthor{\bsnm{Döring},~\bfnm{Leif}\binits{L.}} \AND
  \bauthor{\bsnm{Barczy},~\bfnm{Matyas}\binits{M.}}
(\byear{2012}).
\btitle{Jump type SDEs for self-similar processes}.
\bjournal{Electron. J. Probab.}
\bvolume{17}
\bpages{no. 94, 1-39}.
\bdoi{10.1214/EJP.v17-2402}
\end{barticle}
\endbibitem

\bibitem{EriMal05}
\begin{bincollection}[author]
\bauthor{\bsnm{Erickson},~\bfnm{K.~Bruce}\binits{K.~B.}} \AND
  \bauthor{\bsnm{Maller},~\bfnm{Ross~A.}\binits{R.~A.}}
(\byear{2005}).
\btitle{Generalised {O}rnstein-{U}hlenbeck processes and the convergence of
  {L}\'{e}vy integrals}.
In \bbooktitle{S\'{e}minaire de {P}robabilit\'{e}s {XXXVIII}}.
\bseries{Lecture Notes in Math.}
\bvolume{1857}
\bpages{70--94}.
\bpublisher{Springer, Berlin}.
\bdoi{10.1007/978-3-540-31449-3\_6}
\bmrnumber{2126967}
\end{bincollection}
\endbibitem

\bibitem{Feller52}
\begin{barticle}[author]
\bauthor{\bsnm{Feller},~\bfnm{W.}\binits{W.}}
(\byear{1952}).
\btitle{The parabolic differential equations and the associated semi-groups of
  transformations}.
\bjournal{Annals of Mathematics}
\bvolume{55}
\bpages{468–519}.
\end{barticle}
\endbibitem

\bibitem{Feller54}
\begin{barticle}[author]
\bauthor{\bsnm{Feller},~\bfnm{W.}\binits{W.}}
(\byear{1954}).
\btitle{The general diffusion operator and positivity preserving semi-groups in
  one dimension}.
\bjournal{Annals of Mathematics}
\bvolume{60}
\bpages{417-436}.
\end{barticle}
\endbibitem

\bibitem{Fit06}
\begin{barticle}[author]
\bauthor{\bsnm{Fitzsimmons},~\bfnm{P.~J.}\binits{P.~J.}}
(\byear{2006}).
\btitle{On the existence of recurrent extensions of self-similar {M}arkov
  processes}.
\bjournal{Electron. Comm. Probab.}
\bvolume{11}
\bpages{230--241}.
\bdoi{10.1214/ECP.v11-1222}
\bmrnumber{2266714}
\end{barticle}
\endbibitem

\bibitem{Fou19}
\begin{barticle}[author]
\bauthor{\bsnm{Foucart},~\bfnm{Cl\'{e}ment}\binits{C.}}
(\byear{2019}).
\btitle{Continuous-state branching processes with competition: duality and
  reflection at infinity}.
\bjournal{Electron. J. Probab.}
\bvolume{24}
\bpages{Paper No. 33, 38}.
\bdoi{10.1214/19-EJP299}
\bmrnumber{3940763}
\end{barticle}
\endbibitem

\bibitem{FLZ21}
\begin{barticle}[author]
\bauthor{\bsnm{Foucart},~\bfnm{Clément}\binits{C.}},
  \bauthor{\bsnm{Li},~\bfnm{Pei-Sen}\binits{P.-S.}} \AND
  \bauthor{\bsnm{Zhou},~\bfnm{Xiaowen}\binits{X.}}
(\byear{2021}).
\btitle{{Time-changed spectrally positive Lévy processes started from
  infinity}}.
\bjournal{Bernoulli}
\bvolume{27}
\bpages{1291 -- 1318}.
\bdoi{10.3150/20-BEJ1274}
\end{barticle}
\endbibitem

\bibitem{Get90}
\begin{bbook}[author]
\bauthor{\bsnm{Getoor},~\bfnm{R.~K.}\binits{R.~K.}}
(\byear{1990}).
\btitle{Excessive measures}.
\bseries{Probability and its Applications}.
\bpublisher{Birkh\"{a}user Boston, Inc., Boston, MA}.
\bdoi{10.1007/978-1-4612-3470-8}
\bmrnumber{1093669}
\end{bbook}
\endbibitem

\bibitem{GoinYor03}
\begin{barticle}[author]
\bauthor{\bsnm{G{\"o}ing-Jaeschke},~\bfnm{Anja}\binits{A.}} \AND
  \bauthor{\bsnm{Yor},~\bfnm{Marc}\binits{M.}}
(\byear{2003}).
\btitle{A survey and some generalizations of {B}essel processes}.
\bjournal{Bernoulli}
\bvolume{9}
\bpages{313--349}.
\bdoi{10.3150/bj/1068128980}
\bmrnumber{1997032 (2004g:60098)}
\end{barticle}
\endbibitem

\bibitem{Kallenberg}
\begin{bbook}[author]
\bauthor{\bsnm{Kallenberg},~\bfnm{Olav}\binits{O.}}
(\byear{2002}).
\btitle{Foundations of modern probability},
\bedition{second} ed.
\bseries{Probability and its Applications (New York)}.
\bpublisher{Springer-Verlag}, \baddress{New York}.
\bmrnumber{1876169 (2002m:60002)}
\end{bbook}
\endbibitem

\bibitem{KallenbergRM}
\begin{bbook}[author]
\bauthor{\bsnm{Kallenberg},~\bfnm{Olav}\binits{O.}}
(\byear{2017}).
\btitle{Random measures, theory and applications}.
\bseries{Probability Theory and Stochastic Modelling}
\bvolume{77}.
\bpublisher{Springer, Cham}.
\bdoi{10.1007/978-3-319-41598-7}
\bmrnumber{3642325}
\end{bbook}
\endbibitem

\bibitem{KT}
\begin{bbook}[author]
\bauthor{\bsnm{Karr},~\bfnm{Alan~F.}\binits{A.~F.}}
(\byear{1983}).
\btitle{A Second Course in Stochastic Processes (Samuel Karlin and Howard M.
  Taylor)}
\bvolume{25}.
\bdoi{10.1137/1025063}
\end{bbook}
\endbibitem

\bibitem{Kas88}
\begin{barticle}[author]
\bauthor{\bsnm{Kaspi},~\bfnm{H.}\binits{H.}}
(\byear{1988}).
\btitle{Random time changes for processes with random birth and death}.
\bjournal{Ann. Probab.}
\bvolume{16}
\bpages{586--599}.
\bmrnumber{929064}
\end{barticle}
\endbibitem

\bibitem{KolSav19}
\begin{barticle}[author]
\bauthor{\bsnm{Kolb},~\bfnm{Martin}\binits{M.}} \AND
  \bauthor{\bsnm{Savov},~\bfnm{Mladen}\binits{M.}}
(\byear{2020}).
\btitle{{A characterization of the finiteness of perpetual integrals of
  {L}\'{e}vy processes}}.
\bjournal{Bernoulli}
\bvolume{26}
\bpages{1453--1472}.
\bdoi{10.3150/19-BEJ1167}
\bmrnumber{4058374}
\end{barticle}
\endbibitem

\bibitem{Kyp-Book}
\begin{bbook}[author]
\bauthor{\bsnm{Kyprianou},~\bfnm{Andreas~E.}\binits{A.~E.}}
(\byear{2014}).
\btitle{Fluctuations of {L}\'{e}vy processes with applications},
\bedition{second} ed.
\bseries{Universitext}.
\bpublisher{Springer, Heidelberg}
\bnote{Introductory lectures}.
\bdoi{10.1007/978-3-642-37632-0}
\bmrnumber{3155252}
\end{bbook}
\endbibitem

\bibitem{KypEtAl-17}
\begin{barticle}[author]
\bauthor{\bsnm{Kyprianou},~\bfnm{Andreas~E.}\binits{A.~E.}},
  \bauthor{\bsnm{Pagett},~\bfnm{Steven~W.}\binits{S.~W.}},
  \bauthor{\bsnm{Rogers},~\bfnm{Tim}\binits{T.}} \AND
  \bauthor{\bsnm{Schweinsberg},~\bfnm{Jason}\binits{J.}}
(\byear{2017}).
\btitle{A phase transition in excursions from infinity of the ``fast''
  fragmentation-coalescence process}.
\bjournal{Ann. Probab.}
\bvolume{45}
\bpages{3829--3849}.
\bdoi{10.1214/16-AOP1150}
\bmrnumber{3729616}
\end{barticle}
\endbibitem

\bibitem{KypPar-Book}
\begin{bbook}[author]
\bauthor{\bsnm{Kyprianou},~\bfnm{Andreas~E.}\binits{A.~E.}} \AND
  \bauthor{\bsnm{Pardo},~\bfnm{Juan~Carlos}\binits{J.~C.}}
(\byear{2022}).
\btitle{Stable {L}\'evy processes via {L}amperti-type representations}.
\bseries{Institute of Mathematical Statistics (IMS) Monographs}
\bvolume{7}.
\bpublisher{Cambridge University Press, Cambridge}.
\bdoi{10.1017/9781108648318}
\bmrnumber{4692990}
\end{bbook}
\endbibitem

\bibitem{LiYangZhou}
\begin{barticle}[author]
\bauthor{\bsnm{Li},~\bfnm{Pei-Sen}\binits{P.-S.}},
  \bauthor{\bsnm{Yang},~\bfnm{Xu}\binits{X.}} \AND
  \bauthor{\bsnm{Zhou},~\bfnm{Xiaowen}\binits{X.}}
(\byear{2019}).
\btitle{{A general continuous-state nonlinear branching process}}.
\bjournal{The Annals of Applied Probability}
\bvolume{29}
\bpages{2523 -- 2555}.
\bdoi{10.1214/18-AAP1459}
\end{barticle}
\endbibitem

\bibitem{mitro}
\begin{barticle}[author]
\bauthor{\bsnm{Mitro},~\bfnm{Joanna~B.}\binits{J.~B.}}
(\byear{1979}).
\btitle{Dual Markov processes: Construction of a useful auxiliary process}.
\bjournal{Zeitschrift f{\"u}r Wahrscheinlichkeitstheorie und Verwandte Gebiete}
\bvolume{47}
\bpages{139--156}.
\bdoi{10.1007/BF00535279}
\end{barticle}
\endbibitem

\bibitem{PitYor82}
\begin{barticle}[author]
\bauthor{\bsnm{Pitman},~\bfnm{Jim}\binits{J.}} \AND
  \bauthor{\bsnm{Yor},~\bfnm{Marc}\binits{M.}}
(\byear{1982}).
\btitle{A decomposition of {B}essel bridges}.
\bjournal{Z. Wahrsch. Verw. Gebiete}
\bvolume{59}
\bpages{425--457}.
\bdoi{10.1007/BF00532802}
\bmrnumber{656509}
\end{barticle}
\endbibitem

\bibitem{RevuzYor}
\begin{bbook}[author]
\bauthor{\bsnm{Revuz},~\bfnm{Daniel}\binits{D.}} \AND
  \bauthor{\bsnm{Yor},~\bfnm{Marc}\binits{M.}}
(\byear{1999}).
\btitle{Continuous martingales and {B}rownian motion},
\bedition{third} ed.
\bseries{Grundlehren der Mathematischen Wissenschaften [Fundamental Principles
  of Mathematical Sciences]}
\bvolume{293}.
\bpublisher{Springer-Verlag, Berlin}.
\bdoi{10.1007/978-3-662-06400-9}
\bmrnumber{1725357}
\end{bbook}
\endbibitem

\bibitem{Riv05}
\begin{barticle}[author]
\bauthor{\bsnm{Rivero},~\bfnm{V\'{\i}ctor}\binits{V.}}
(\byear{2005}).
\btitle{Recurrent extensions of self-similar {M}arkov processes and
  {C}ram\'{e}r's condition}.
\bjournal{Bernoulli}
\bvolume{11}
\bpages{471--509}.
\bdoi{10.3150/bj/1120591185}
\bmrnumber{2146891}
\end{barticle}
\endbibitem

\bibitem{Riv07}
\begin{barticle}[author]
\bauthor{\bsnm{Rivero},~\bfnm{V\'{\i}ctor}\binits{V.}}
(\byear{2007}).
\btitle{Recurrent extensions of self-similar {M}arkov processes and
  {C}ram\'{e}r's condition. {II}}.
\bjournal{Bernoulli}
\bvolume{13}
\bpages{1053--1070}.
\bdoi{10.3150/07-BEJ6082}
\bmrnumber{2364226}
\end{barticle}
\endbibitem

\bibitem{Salisbury86}
\begin{barticle}[author]
\bauthor{\bsnm{Salisbury},~\bfnm{Thomas~S.}\binits{T.~S.}}
(\byear{1986}).
\btitle{Construction of right processes from excursions}.
\bjournal{Probab. Theory Related Fields}
\bvolume{73}
\bpages{351--367}.
\bdoi{10.1007/BF00776238}
\bmrnumber{859838}
\end{barticle}
\endbibitem

\bibitem{Sal86-Ito}
\begin{barticle}[author]
\bauthor{\bsnm{Salisbury},~\bfnm{Thomas~S.}\binits{T.~S.}}
(\byear{1986}).
\btitle{On the {I}t\^{o} excursion process}.
\bjournal{Probab. Theory Related Fields}
\bvolume{73}
\bpages{319--350}.
\bdoi{10.1007/BF00776237}
\bmrnumber{859837}
\end{barticle}
\endbibitem

\bibitem{Vuo94}
\begin{barticle}[author]
\bauthor{\bsnm{Vuolle-Apiala},~\bfnm{J.}\binits{J.}}
(\byear{1994}).
\btitle{It\^{o} excursion theory for self-similar {M}arkov processes}.
\bjournal{Ann. Probab.}
\bvolume{22}
\bpages{546--565}.
\bmrnumber{1288123}
\end{barticle}
\endbibitem

\bibitem{Whitt80}
\begin{barticle}[author]
\bauthor{\bsnm{Whitt},~\bfnm{W.}\binits{W.}}
(\byear{1980}).
\btitle{Some useful functions for functional limit theorems}.
\bjournal{Mathematics of operations research}.
\end{barticle}
\endbibitem

\bibitem{Whitt}
\begin{bbook}[author]
\bauthor{\bsnm{Whitt},~\bfnm{W.}\binits{W.}}
(\byear{2002}).
\btitle{Stochastic-Process Limits}.
\bseries{Springer series in operations research}.
\bpublisher{Springer-Verlag}.
\end{bbook}
\endbibitem

\bibitem{Zanzotto02}
\begin{barticle}[author]
\bauthor{\bsnm{Zanzotto},~\bfnm{P.~A.}\binits{P.~A.}}
(\byear{2002}).
\btitle{On stochastic differential equations driven by a {C}auchy process and
  other stable {L}évy motions}.
\bjournal{Annals of Probability}
\bvolume{30}
\bpages{802-825}.
\end{barticle}
\endbibitem

\end{thebibliography}
\end{document}